\documentclass[12pt]{amsart}
\usepackage{amsmath,amscd,amssymb,euscript}
\usepackage{latexsym}
\usepackage{graphicx}
\usepackage{tikz}
\usetikzlibrary{backgrounds}
\usetikzlibrary{decorations.fractals}
\usetikzlibrary{calc,intersections,through,backgrounds}
\usepackage{mathrsfs}
\usepackage{epstopdf}
\usepackage{rotating}
\usepackage{lscape}
\input xy
\xyoption{all}
\newcommand{\del}{\partial}
\newcommand{\bC}{{\mathbb C}}
\newcommand{\bG}{{\mathbb G}}
\newcommand{\bH}{{\mathbb H}}
\newcommand{\bK}{{\mathbb K}}

\newcommand{\bP}{{\mathbb P}}
\newcommand{\bQ}{{\mathbb Q}}

\newcommand{\bZ}{{\mathbb Z}}
\newcommand{\cA}{{\mathcal A}}

\newcommand{\cC}{{\mathcal C}}

\newcommand{\cF}{{\mathcal F}}
\newcommand{\cG}{{\mathcal G}}
\newcommand{\cH}{{\mathcal H}}

\newcommand{\cJ}{{\mathcal J}}

\newcommand{\cM}{{\mathcal M}}
\newcommand{\cN}{{\mathcal N}}
\newcommand{\cO}{{\mathcal O}}

\newcommand{\cR}{{\mathcal R}}

\newcommand{\cW}{{\mathcal W}}
\newcommand{\cX}{{\mathcal X}}
\newcommand{\cY}{{\mathcal Y}}

\newcommand{\oW}{\overline{W}}

\newcommand{\tC}{\widetilde{C}}

\newcommand{\tG}{\widetilde{G}}

\newcommand{\tcG}{\widetilde{\mathcal G}}

\newcommand{\tX}{\widetilde{X}}

\newcommand{\tcX}{\widetilde{\mathcal X}}

\newcommand{\ra}{\rightarrow}
\newcommand{\lra}{\longrightarrow}

\newcommand{\inj}{\hookrightarrow}
\newcommand{\surj}{-\hspace{-4pt}\rightarrow\hspace{-19pt}{\rightarrow}\hspace{3pt}}

\newcommand{\G}{\Gamma}

\newcommand{\D}{\Delta}
\newcommand{\T}{\Theta}
\newcommand{\tT}{\widetilde{\Theta}}

\newcommand{\Sg}{\mathfrak{S}}

\newcommand{\Pic}{\operatorname{Pic}}

\newcommand{\al}{{\alpha}}
\newcommand{\bb}{{\beta}} 
\newcommand{\dd}{{\delta}}
\newcommand{\ga}{{\gamma}}

\newcommand{\te}{{\theta}}

\theoremstyle{definition}

\newtheorem{proposition}{Proposition}[section]
\newtheorem{lemma}[proposition]{Lemma}
\newtheorem{claim}[proposition]{Claim}
\newtheorem{theorem}[proposition]{Theorem}

\newtheorem{corollary}[proposition]{Corollary}

\newtheorem{remark}[proposition]{Remark}
\newtheorem{notation}[proposition]{Notation}

\numberwithin{equation}{section}

\begin{document}

\title{The primitive cohomology of the theta divisor of an abelian fivefold}

\author{E. Izadi}

\address{Department of Mathematics, University of California San Diego, 9500 Gilman Drive \# 0112, La Jolla, CA 92093-0112, USA}

\email{eizadi@math.ucsd.edu}

\author{Cs. Tam\'as}
\address{Department of Mathematics and Statistics, Langara College, 100 West 49th Avenue, Vancouver, BC,
Canada V5Y 2Z6}
\email{ctamas@langara.bc.ca}

\author{J. Wang}

\address{Department of Mathematics, Boyd
Graduate Studies Research Center, University of Georgia, Athens, GA
30602-7403, USA}
\email{jiewang@math.uga.edu}

\dedicatory{Dedicated to Herb Clemens}

\thanks{The first author was partially supported by the National
Science Foundation and the National Security Agency. Any opinions, findings
and conclusions or recommendations expressed in this material are those
of the author(s) and do not necessarily reflect the views of the
NSF or NSA}

\subjclass[2010]{Primary 14C30 ; Secondary 14D06, 14K12, 14H40}

\begin{abstract}

  The primitive cohomology of the theta divisor of a principally
  polarized abelian variety of dimension $g$ is a Hodge structure of
  level $g-3$. The Hodge conjecture predicts that it is contained in
  the image, under the Abel-Jacobi map, of the cohomology of a family
  of curves in the theta divisor. In this paper we use the Prym map to
  show that this version of the Hodge conjecture is true for the theta
  divisor of a general abelian fivefold.

\end{abstract}

\maketitle

\tableofcontents

\section*{Introduction}

Let $A$ be a principally polarized abelian variety ({\em ppav}) of
dimension $g\geq 4$, let $\T$ be a symmetric theta divisor for $A$, and assume
that $\T$ is smooth. The cohomology group
\[
H^{g-1}(\T,\bZ)
\]
contains a natural sublattice of rank $g! - \frac{1}{g+1}{2g \choose g}$ (see \cite[p. 561]{I4})
\[
\bK:=Ker(H^{g-1}(\T, {\bZ}) \stackrel{j_*}{\longrightarrow}
H^{g+1}(A,{\bZ}))
\] 
which we call the primitive cohomology of $\T$. There is also a Hodge structure
$\bH \subset H^{g-1}(\T,\bZ)$ which fits in an exact sequence
\[
0\lra\bK\lra \bH \lra H^{g-3} (A,\bZ) \lra 0.
\]
By \cite[p. 562]{I4}, these Hodge structures are all of level $g-3$: For a
rational Hodge structure $V := (V_{\bQ} , V_{\bQ }\otimes\bC =\oplus_{
  p+q=n } V^{ p,q })$ of weight $n$, the level $l(V)$ of $V$ is
defined as the positive integer
\[
l(V) := max\{ |p-q| : V^{ p,q }\neq 0\}.
\]
Grothendieck's version of the Hodge conjecture states that if $H^{g-1} (\T, \bQ
)$ contains a Hodge substructure of level $g-3$, then it is
contained in the image, under Gysin push-forward, of the cohomology of a smooth (possibly reducible) variety of dimension $g-2$. After tensoring with $\bQ$ we have
\[
\bH_{\bQ}  := \bH \otimes \bQ = \bK_{\bQ} \oplus \theta \cdot H^{g-3} (A, \bQ)
\]
where $\theta := [\T]$ is the cohomology class of $\T$ and $\theta \cdot H^{g-3} (A, \bQ)$ is the image of $H^{g-3} (A, \bQ) \cong H^{g-3} (\T,\bQ)$ in $H^{g-1} (\T, \bQ)$. This is also a Hodge substructure of level $g-3$ and satisfies the Hodge conjecture since it is in the image, for instance, of the cohomology  of an intersection of a translate of $\T$ with $\T$. Therefore the Hodge conjecture for $\bH_{\bQ}$ is equivalent to the Hodge conjecture for $\bK_{\bQ}$.

An equivalent formulation of the Hodge conjecture is that $\bH_{\bQ}$ or $\bK_{\bQ}$ is contained in the image, under the Abel-Jacobi map, of the cohomology of some family of curves in $\T$ (see e.g. \cite[pp. 492-493]{I11} for a proof of this elementary fact). For $g=4$, it was proved in \cite{I4} that the family of Prym-embedded curves in $\T$ is a solution to this problem for $\bH_{\bQ}$.

Denote $\cA_g$ the coarse moduli space of principally polarized abelian varieties of dimension $g$.
Our main result is
\begin{theorem}\label{mainthm}

For $(A,\T)$ in a non-empty Zariski open subset of $\cA_5$, the general Hodge conjecture holds for the Hodge structure $\bH_\bQ \subset H^4(\T, \bQ)$
  and hence $\bK_\bQ \subset H^4(\T, \bQ)$.

\end{theorem}

As the rational cohomology of $\Theta$ is the sum  of $\bK_{\bQ}$ and the rational cohomology of $A$, our result, together with the main result of \cite{hazama94}, implies

\begin{corollary}

  For $(A, \T)$ in the complement of countably many proper Zariski closed subsets of $\cA_5$, the general Hodge conjecture holds for $\T$.

\end{corollary}

A general ppav of dimension $5$ is the Prym variety of an \'etale
double cover of smooth curves $\tX\ra X$ with $X$ general of genus $6$.

We construct a
family of curves in $\T$ (see Section \ref{seccurves})
$$\xymatrix{F_r\ar[r]^-{\rho_2}\ar[d]_-{\rho_1}&\T\\
\tG^1_5&}$$
dependent on the choice of a general point $r\in \tX$. Here ${\tG^1_5}$ is an \'etale double cover of the variety $G^1_5(X)$ parametrizing pencils of degree $5$ on $X$ ($\cong W^1_5(X)$ if $X$ is not a plane quintic), which is a smooth irreducible surface for $X$ sufficiently general. The Abel-Jacobi map for this
family of curves is, by definition, 
$$\rho_{2*}\rho_1^* : H^{2}({\tG^1_5}) \ra
H^{4}(\T).$$ The image of the Abel-Jacobi map defines a Hodge
substructure of level $\leq 2$ of the cohomology of $\T$. Theorem \ref{mainthm} is a direct consequence of

\begin{theorem}\label{thmrho2rho1}
The Hodge structure $\bH_\bQ$ is the sum of $\theta\cdot H^2(A,\bQ)$ and the image of $\rho_{2*}\rho_1^*$.
\end{theorem}

We prove Theorem \ref{thmrho2rho1} by specializing the \'etale double cover $\tX\ra X$ to a Wirtinger cover $\tC_{pq}\ra C_{pq}$, where $C_{pq}$ is the nodal curve of genus $6$ obtained from a general curve $C$ of genus $5$ by identifying $p$ and $q$ in $C$ and the Wirtinger cover $\tC_{pq}$ is obtained as the union
of two copies $C_1$ and $C_2$ of $C$ where $p_k (=p)$ on $C_k$ is identified
with $q_{3-k}$ on $C_{ 3-k}$ for $k=1,2$. The Prym variety of the Wirtinger cover $\tC_{pq}\ra C_{pq}$ is naturally isomorphic to the polarized Jacobian $(J(C),\T_C)$ of the curve $C$ (see e.g. Section \ref{subsecPrymfamily} below).

In most of the paper we work with a one-parameter family $\cX \ra T$ of curves of genus $6$ over an analytic disc $T$ with smooth total space, with general fiber $X_t$ a general curve of genus $6$ and special fiber $X_0 = C_{pq}$ at $0\in T$ a general one-nodal curve of genus $6$. We also assume given an \'etale double cover $\tcX \ra\cX$ whose special fiber $(\tX_0 \ra X_0) = (\tC_{pq} \ra C_{pq})$ is the Wirtinger cover described above. To this family one associates the family of polarized Prym varieties $(\cA, \T) \ra T$ with special fiber $(A_0 , \T_0)$.

The plan of the paper is as follows.

In Section \ref{seccurves} we construct the family of curves $F_r$ in the general case. In Section \ref{limitcurve} we describe the family of curves in the Wirtinger double cover case. We also explicitly describe the flat limit $G_0$ of the base $G_t :=\widetilde{G}^1_5(X_t)$ of the family. This is the transverse union of two smooth isomorphic surfaces. We prove that the total space $\cG \ra T$ of the family of the $G_t$ is smooth.

In Section \ref{sectheta} we describe the total space of the family of theta divisors $\T \ra T$. The singular locus of $\T_0$ is a translate of the smooth genus $11$ curve $W^1_4 \subset \Pic^4 C \cong JC$ parametrizing pencils of degree $4$. 
We prove that the total space $\T$ has ten ordinary double points corresponding to the five $g_i\in W^1_4$, $i=1,...,5$, such that $h^0(g_i-p-q)>0$, and their residuals $h_i:= |K_C-g_i|$.

In Section \ref{secsstheta}, we construct a semistable reduction $\widetilde{\T}$ of the family $\{\T_t\}$. The central fiber $\widetilde{\T}_0$ of the new family 
has two components $M_1$ and $M_2$, where $M_1$ is a resolution of $\T_0$ and $M_2$ is the exceptional divisor. During this process $T$ is replaced by a double cover ramified only at $0$ and we also replace the family $\cG$ by $\widetilde{\cG}$, which is a resolution of the base change of $\cG$ to this double cover.

In Section \ref{secgeneralMH} we recall the necessary background material about the Clemens-Schmid exact sequence and limit mixed Hodge structures.

In Section \ref{seccohtheta} we compute the limit mixed Hodge structure induced by the family $\tT$ on the cohomology of $\T_t$. The weight filtration is nonzero only in weights 3, 4, 5 with associated graded pieces as follows:
$$Gr_3H^4(\T_t)\cong Gr_5H^4(\T_t)\cong\bQ^{12}$$
and $$Gr_4H^4(\T_t)\cong \bQ^{264}.$$

To extend the family of curves to the central fiber we first assume given a section $r: T \ra \tcX, t \mapsto r_t$ of the family of curves $\tcX$. Next we replace the families $F_{r_t}$ by their images in the products $G_t \times \T_t$. The Abel-Jacobi map on the fiber at $t$ can then be described as the map induced by the cycle $(\rho_1,\rho_2)_*[F_{r_t}]\in H^6(G_t\times \T_t)$:
$$\xymatrix{H^2(G_t)\ar[r]^-{\rho_1^*}&H^2(G_t\times\T_t)\ar[rr]^-{\cup(\rho_1,\rho_2)_*[F_{r_t}]}&&H^8( G_t\times\T_t)\ar[r]^-{\rho_{2*}}&H^4(\T_t).}$$

To compute the limit of these maps at $0$, we need a semi-stable reduction of the fiber product $\widetilde{\cG} \times_T \tT$. This is constructed in Section \ref{secssprod}.
The resulting space $\widetilde{\tilde{\cG}\times_T\tilde{\T}}$ is a small resolution of the fiber product $\widetilde{\cG}\times_T\widetilde{\T}$. 

In Section \ref{secAJgensp} we show how the computation of the Abel-Jacobi map on the general fiber can be reduced to computing it on (the strata of) the special fiber. We summarize the latter computations in Propositions \ref{AJ21}-\ref{AJ2global} and show how Theorem \ref{thmrho2rho1} follows from them.

Section \ref{secF'''} describes the limit families of curves at $t=0$.

In Section \ref{secAJ1} we prove Propositions \ref{AJ21}-\ref{AJ2global}. In other words, we compute the image of the Abel-Jacobi map $AJ$ on the graded level with respect to the weight filtration:

\begin{eqnarray}\label{Gr4}Gr_2H^2(\widetilde{\cG})\ra Gr_4H^4(\widetilde{\T})
\end{eqnarray}
and
\begin{eqnarray}\label{Gr3}Gr_1H^2(\widetilde{\cG})\ra Gr_3H^4(\widetilde{\T})
\end{eqnarray}

Finally in the Appendix (Section \ref{secAppx}) we gather some technical results needed in the rest of the paper.

\begin{remark}

For $g\leq 2$, $g! - \frac{1}{g+1}{2g \choose g} = 0$ so $\bK =0$. For $g=3$, the lattice $\bK$ has rank $1$ and level $0$, i.e., it is generated by a Hodge class. The abelian variety $(A, \T) = (JC, \T_C) $ is the Jacobian of a curve of genus $3$. The theta divisor is isomorphic to the second symmetric power $C^{(2)}$ of $C$ and $\bK$ is generated by the class $\theta - 2\eta$ where $\eta$ is the cohomology class of the image of $C$ in $C^{(2)}$ via addition of a point $p$ of $C$:
\[
\begin{array}{ccc}
C & \inj & C^{(2)} \\
t & \mapsto & t+p.
\end{array}
\]

\end{remark}

\section*{Notation and Conventions}
\begin{enumerate}

\item Unless otherwise specified, all singular cohomology groups are with $\bQ$-coefficients.

\item For a smooth curve $C$ of genus $g$ and integer $k>0$, choose a symplectic basis
 $$\xi_i\in H^1(C,\bZ)\cong H^1(Pic^kC,\bZ),\ \ \  i=1,...,2g.$$
 Put $\xi_i':=\xi_{i+g}$, $\sigma_i=\xi_i\xi_i'$ for $i=1,..,g$ and denote $\theta=\sum_{i=1}^g\sigma_i$ the class of the theta divisor in $Pic^kC$. We also denote $\xi_i$, $\sigma_i$ and $\theta$ the pull backs to the k-th symmetric product $C^{(k)}$ under the natural map
  $$C^{(k)}\ra Pic^kC.$$
  Finally, denote $\eta\in H^2(C^{(k)},\bZ)$ the class of the cycle $pt+C^{(k-1)}\subset C^{(k)}$.
  
\item We will interchangeably refer to elements of $Pic^k C$ as invertible sheaves or complete linear systems. We use $\equiv$ to denote linear equivalence between divisors and $D_1\le D_2$ means $D_2-D_1$ is an effective divisor. As usual, we denote $W^r_d \subset Pic^d$ the scheme parametrizing complete linear systems of degree $d$ and dimension $r$. By a $g^r_d$ we will mean a linear system of degree $d$ and dimension $r$.

\item \label{notesym} For products of symmetric powers of $C$, we denote $\omega_k := pr_k^* \omega \in H^{\bullet}(C^{(n_1)}\times ...\times C^{(n_k)}\times ...)$, where $\omega\in H^{\bullet}(C^{(n_k)})$ and $pr_k$ is the $k$-th projection.

\item Via translation by an invertible sheaf of degree $g-1$, we identify $JC = Pic^0 C$ with $Pic^{g-1} C$ so that $\T_C$ is identified with Riemann's theta divisor $W^0_{g-1} \subset Pic^{g-1} C$.

\item As usual, $\omega_C$ will denote the dualizing sheaf of $C$ and $K_C$ an arbitrary canonical divisor on $C$.
\end{enumerate}

\section{The family of curves in $\T$: the general case}\label{seccurves}

Let $X$ be a smooth curve of genus $6$ with an \'etale double cover
$\tX$ of genus 11. For a pencil $M$ of degree $5$ on $X$ consider the curve $B_M$
defined by the pull-back diagram
\[
\begin{array}{ccc}
B_M& \subset & \tX^{ (5) } \\
\downarrow & &\downarrow \\
\bP^1 = |M| & \subset & X^{ (5) }.
\end{array}
\]
By \cite[p. 360]{beauville82} the curve $B_M$ has two isomorphic connected
components, say $B_{M}^1$ and $B_{M}^2$.
Put $M ' = | K_X - M |$. Then
for any $D\in B_M\subset\tX^{ (5) }$ and any $D'\in B_{M'}\subset\tX^{ (5)
}$, the push-forward to $X$ of $D+D'$ is a canonical divisor on
$X$. Hence the image of
\[\begin{array}{ccc}
B_M\times B_{M'} & \lra & Pic^{ 10 }\tX \\
(D,D') &\longmapsto &\cO_{\tX } (D+D' )
\end{array}
\]
is contained in the preimage of $\omega_X$ by the Norm map $Nm :Pic^{
10 }\tX\ra Pic^{ 10 }X$. This preimage has two connected components,
say $A_1$ and $A_2$,
each a principal homogeneous space under the Prym variety $(A, \T)$ of the cover
$\tX\ra X$ and parametrizing divisors whose spaces of global sections
are even, respectively odd, dimensional. If we have labeled the connected
components of $B_M$ and $B_{M'}$ in such a way that $B_{M}^1\times B_{M'}^1$ maps
into $A_1$, then $B_{M}^2\times B_{M'}^2$ also maps into $A_1$ while $B_{M}^1\times
B_{M'}^2$ and $B_{M}^2\times B_{M'}^1 $ map into $A_2$.

In order to obtain a family of curves in the theta divisor $\T=\T_{\tX\ra X}=\frac{1}{2}\T_{\tX}|_{A_1}$ of the Prym
variety $A$, we need to globalize the above construction.

The scheme $W^1_5 (X)$ parametrizing complete linear systems of degree
$5$ and dimension at least $1$ on $X$ has a determinantal structure
which is smooth for $X$ sufficiently general. Let
$G^1_5 (X)$ denote the scheme over $W^1_5 (X)$ parametrizing pencils
of degree $5$. Note that $W^1_5 (X)$ is a surface unless
$X$ is hyperelliptic.

The universal family $P_5^1$ of divisors of the elements of $G^1_5$ is a
$\bP^1$ bundle over $G^1_5$ with natural maps
\[
\begin{array}{ccc}
P^1_5 & \lra & X^{ (5) } \\
\downarrow & & \\
G^1_5 & & \\
\downarrow & & \\
W^1_5 & & 
\end{array}
\]
whose pull-back via $\tX\ra X$ gives us the family of the
curves $B_M$ as $M$ varies:
\[
\begin{array}{ccc}
B & \lra & \tX^{ (5) } \\
\downarrow & &\downarrow \\
P^1_5 & \lra & X^{ (5) } \\
\downarrow & & \\
G^1_5 . & & 
\end{array}
\]
The parameter space of the connected components of the curves $B_M$ is an \'etale double cover $\widetilde{G}^1_5$ of
$G^1_5$.

The family of curves in the theta divisor of the Prym variety will be constructed as follows.
Assuming that
$X$ is not a plane quintic, the natural map
$G^1_5\ra W^1_5$ is an isomorphism. We have the involution $\iota :
M\mapsto M' := |K_X -M |$ on $W^1_5$ and hence also on $G^1_5$. First define a family of surfaces $'F$ over $G^1_5$ as the fiber product
\[
\begin{array}{ccc}
'F & \lra & B \\
\downarrow & & \downarrow^{\iota\varrho } \\
B & \stackrel{\varrho}{\lra } & G^1_5.
\end{array}
\]
As noted above, the image of $'F$ in $Pic^{ 10 }\tX$ maps into $Nm^{-1}(\omega_X)\subset Pic^{10}\tX$ which also shows that $'F$ has two connected
components. One component, denoted $'F_1$, maps into 
$A_1$ and the other, denoted $'F_2$, maps into $A_2$. The fiber of $'F_1$ over a point $|M|\in G^1_5$ has two
connected components $B_{M}^1\times B_{M'}^1$ and $B_M^2\times B_{M'}^2$.

Therefore, if we make the base change
\[
\xymatrix{''F_1\ar[r]\ar[d]&'F_1\ar[d]\\
\widetilde{G}^1_5\ar[r]&G^1_5,}
\]
$''F_1$ splits into two connected components (both isomorphic to $'F_1$ over $\bC$ but
their maps to $\tG^1_5$ differ by the involution of $\tG^1_5$). We denote $F$ the component which has
fiber $B_M^1\times B_{M'}^1$ over the point parametrizing $B_M^1$.

Finally, we think of $F$ as a correspondence
\[
\xymatrix{F\ar@^{(->}[r]&\widetilde{G}^1_5\times
  \widetilde{X}^{(5)}\times \widetilde{X}^{(5)}}
\]
and define our family of curves $F_r$ by intersecting $F$ with the pull back of the
divisor $r+\widetilde{X}^{(4)}$ in the first factor $\widetilde{X}^{(5)}$ for a general point $r\in
\widetilde{X}$. The variety $F_r'$ is the image of $F_r$ in $\tG_5^1\times\T$:
$$\xymatrix{F_r\ar^-{(\rho_1,\rho_2)}[r]\ar[d]^-{\rho_1}&F_r'\subset\tG_5^1\times\T\\
\widetilde{G}^1_5.&}$$

\begin{remark}It is easy to check that $F_r$ maps generically one-to-one to $\tG_5^1\times\T$. So the push-forward of the cycle class $[F_r]$ is the cycle class $[F'_r]$. 
\end{remark}

\section{The family of curves in $\T$: the degeneration to a Wirtinger cover}\label{limitcurve}

Let $\tcX\rightarrow \cX$
be the family of \'etale double covers over $T$ specializing to the Wirtinger cover $\widetilde{C}_{pq} \ra C_{pq}$ at $0\in T$ as explained in the introduction.

Also assume that $\cX$ and
$\widetilde{\cX}$ are smooth. 

Consider the smooth one-parameter family
\[
\begin{array}{ccc}
J^5 C_{ pq } & \lra & \cJ^5 \\
\downarrow & & \downarrow \\
0 & \in & T
\end{array}
\]
obtained as a compactification of the relative degree $5$ Picard scheme of $\cX$. The fiber of $\cJ^5\ra T$ is $Pic^5 X_t$ for $t\ne0$ and the fiber at $t=0$ is  the usual compactification $J^5 C_{ pq }$ of $Pic^5 C_{ pq }$ obtained as
follows.

\subsection{The compactified Jacobian of $C_{ pq }$}\label{sssectcomp}

Let $\bP Pic^5 C_{ pq }$ be the unique projective line bundle over
$Pic^5 C$ containing the $\bG_m$-bundle $Pic^5 C_{ pq }\ra Pic^5
C$. Then $\bP Pic^5 C_{ pq }\setminus Pic^5 C_{ pq }$ is the union of
the zero section $Pic^5_0\cong Pic^5 C$ and the infinity section
$Pic^5_{\infty}\cong Pic^5 C$ of $\bP Pic^5 C_{ pq }\ra Pic^5 C$. The
compactification $J^5 C_{ pq }$ is obtained from $\bP Pic^5 C_{ pq }$
by identifying $x\in Pic^5 C = Pic^5_0$ with $x\otimes\cO_C (p-q)\in
Pic^5 C = Pic^5_{\infty}$. The points of $J^5 C_{pq}\setminus Pic^5 C_{ pq }$
are the push-forwards $\nu_*N$ where $\nu : C\ra C_{ pq }$ is the
normalization map and $N\in Pic^4 C$.

\subsection{The support of $W^1_5 (C_{ pq })$ and of its compactification
$\oW^1_5 (C_{ pq })$}\label{ssectWCpqsupp}

Let $\oW^1_5 (C_{ pq })$ be the
subvariety of $J^5 C_{ pq }$ parametrizing torsion-free rank 1 sheaves $M$
of degree $5$ such that $h^0 (M )\ge 2$. Let $W_{pq}\subset W^1_5 (C)\subset Pic^5 C$ be the surface consisting
of those $L$ such that $h^0 (L -p-q )> 0$, and
let $X_p$ and $X_q$ be the two curves $p+W^1_4(C)$ and $q+W^1_4(C)$ in $W_{pq}$. Pull-back via the normalization map gives a morphism
\[
\nu^* : W^1_5 (C_{pq}) \lra W_{pq}
\]
whose image is $W_{pq} \setminus X_p \cup X_q$. We have

\begin{lemma}\label{lemW15bar}

The morphism $\nu^* : W^1_5 (C_{pq}) \ra W_{pq}$ is injective. Its inverse
extends to a birational morphism
\[
(\nu^*)^{-1} : W_{ pq }\: \surj \:\oW_5^1 (C_{ pq })
\]
that is bijective on $W_{ pq }\setminus X_p\cup X_q$ and sends
$p+g^1_4$ and $q+g^1_4$ to $\nu_* g^1_4 $.

The involution $\iota$ extends to $W_{pq}$ and sends $L$ to $|K_C+p+q-L|$. It also descends to $\oW_5^1 (C_{ pq })$ and sends $\nu_*g^1_4$ to $\nu_*(|K_C-g^1_4|)$.

\end{lemma}

\begin{proof} If $M\in \oW_5^1 (C_{ pq })$ is invertible,
then the pull-back $\nu^*M $ is an invertible sheaf of degree
$5$ on $C$ and we have the usual exact sequence
\[
0\lra M \lra \nu_*\nu^*M \lra sk \lra 0
\]
where $sk$ is a skyscraper sheaf of length $1$ supported at the
singular point of $C_{ pq }$. It follows that if $h^0 (M )\geq 2$,
then $h^0 (\nu^*M )\geq 2$ also. Since $C$ is a general curve of
genus $5$ and $\nu^*M$ has degree $5$, we have $h^0 (\nu^*M )\leq
2$. So the map $H^0 (\nu^*M ) = H^0 (\nu_*\nu^*M )\ra H^0 (sk)$
obtained from the above sequence is zero. Since this map factors
through the evaluation map $H^0 (\nu^*M )\ra (\nu^*M )_p\oplus
(\nu^*M )_q$, a moment of reflection will show that a map
$\nu_*\nu^*M\; \surj \; sk$ that is zero on global sections and has locally
free kernel exists if and only if neither $p$ nor $q$ are base points
of $|\nu^*M |$ and the unique nonzero section of $\nu^*M$ that
vanishes at $p$ also vanishes at $q$.

Conversely, given an invertible sheaf $L$ of degree $5$ on $C$ such
that neither $p$ nor $q$ are base-points of $|L|$ and the unique
nonzero section of $L$ vanishing at $p$ also vanishes at $q$, one sees
immediately that there is a unique quotient map
\[
\nu_* L\:\surj \: sk
\]
onto a skyscraper sheaf of rank $1$ supported at the singular point of
$C_{ pq }$ such that the resulting map on global sections
\[
H^0 (\nu_* L )\lra H^0 (sk)
\]
is zero. The kernel of such a map is also immediately seen to be an
invertible sheaf of degree $5$ on $C_{ pq }$.
Thus $W^1_5(C_{pq})$ maps injectively into $W_{pq}$ under $\nu^*$.

If $M$ is not locally free, then it is the direct image of a $g^1_4\in W^1_4(C)$. We have two
exact sequences
\[
0\lra\nu_*g^1_4\lra\nu_*(p+g^1_4)\lra sk\lra 0,
\]
\[
0\lra\nu_*g^1_4\lra\nu_*(q+g^1_4)\lra sk\lra 0,
\]
that give us two representations of $M$ as the kernel of a
surjective map from the pushforward of an invertible sheaf to $sk$. Thus $\nu^*$ maps $p+g^1_4$ and $q+g^1_4$ to $\nu_* g^1_4 $. The statements about $\iota$ are immediate.
\end{proof}

Note that $W_{pq}\subset Pic^5 C$ naturally embeds in $C^{(3)}$ via
two different maps: $q_1 : L\mapsto \G_3 :=|K_C-L|$ and $q_2 : L\mapsto \G_3' :=|L
-p-q|$. We have

\begin{proposition}

The surface $W_{pq}$ is smooth for $C$, $p$ and $q$ general.

\end{proposition}
\begin{proof}
For $L\in W_{pq}$, via the two embeddings of $W_{pq}$ in $C^{ (3) }$, the
tangent space to $W_{pq}$ at $L$ is contained in the tangent
spaces to $C^{ (3) }$ at $\G_3$ and $\G'_3$. Embedding $C^{ (3) }$ in
$Pic^0 C$ via subtraction of a fixed divisor of degree $3$, the
projectivizations of these two tangent spaces can be identified (after
a translation) with the respective spans $\langle \G_3\rangle$ and
$\langle \G'_3\rangle$ of $\G_3$ and $\G'_3$ in the canonical space
$|K_C|^*\cong\bP T_0 Pic^0 C$. We therefore need to prove that $\langle \G_3\rangle \neq
\langle \G_3'\rangle$.

Using Riemann Roch and Serre Duality it
is immediately seen that a divisor of degree $\geq 5$ on $C$ cannot
span a space of dimension $\leq 2$ in $|K_C|^*$. So, if $\langle
\G_3\rangle =\langle \G'_3\rangle$, then $\G_3$ and $\G'_3$ have a divisor
of degree at least $2$ in common: $\G_3 = \G_2 + t$ and $\G'_3 = \G_2 +
t'$ for some $\G_2\in C^{ (2) }$ and $t, t'\in C$. Note that by our
assumptions $\G_3 + \G'_3 +p+q\in |K_C|$ is a canonical divisor.

If $t=t'$, then the span $\langle \G_3 + \G'_3 +p+q\rangle$ is a
hyperplane in $|K_C|^*$ which is tangent to the canonical image of $C$
at three distinct points or has even higher tangency to the canonical
curve. Such hyperplanes form a family of dimension $1$ for $C$
general, hence choosing $p$ and $q$ sufficiently general, this can be
avoided.

If $t\neq t'$, then the span $\langle \G_2 +t+t'\rangle$ is a plane, and, by Riemann Roch
and Serre Duality, $|\G_2 +t+t'|\in W^1_4 (C)$, hence $|K_C-\G_2 -t-t'|\in W^1_4
(C)$. However $|K_C -\G_2 -t-t' | = |p+q+ \G_2 |$. The divisors of $g^1_4 := |\G_2 +t+t'|$
and $h^1_4 := |K_C -\G_2 -t-t' | = |p+q+ \G_2 |$ are cut on $C$ by the two rulings of a
quadric of rank $4$ (since $C$ is general, it is not contained in any quadrics of rank
$3$). Since $\G_2$ appears in both $g^1_4$ and $h^1_4$, the line $\langle \G_2\rangle$
contains the singular point of this quadric of rank $4$. There is a one-parameter family
of such secants to $C$ and for each such secant $\langle \G_2\rangle$, there are exactly
$5$ (counted with multiplicities) divisors $p+q$ such that $h^0 (\G_2 +p+q)\geq
2$. Therefore there is a one parameter family of divisors $p+q$ such that $h^0 (\G_2
+p+q)\geq 2$ for some $\G_2$ such that $\langle \G_2\rangle$ contains the singular point
of some quadric of rank $4$ containing $C$. Taking $p+q$ outside this one-parameter family
this case is also eliminated.

So $\langle \G_3\rangle\neq\langle \G'_3\rangle$, the intersection $\langle
\G_3\rangle\cap\langle \G'_3\rangle$ is a projective line which contains
$\bP T_{L} W_{ pq}$. So $T_{L} W_{ pq}$ has dimension $2$ and
$W_{pq}$ is smooth at $L$.
\end{proof}

The computation of the Hilbert polynomial of $\oW^1_5 (C_{ pq })$ in Lemma \ref{lemWflat} shows that $\overline{W}^1_5(C_{pq})$, with its reduced scheme structure, is the flat limit of $W^1_5(X_t)$. Therefore, if $\cW^1_5\subset\cJ^5$ is the family of sheaves with at least two independent global sections on each fiber $X_t$, then $\cW^1_5\ra T$ is
flat. Moreover, we have

\begin{proposition}\label{Wsm}

The total space $\cW^1_5$ is smooth.

\end{proposition}

\begin{proof}

This is clear at any invertible sheaf $M\in W^1_5 (X_t)$ or $M\in
W^1_5 (C_{ pq })$. Suppose therefore that $M\in \oW^1_5 (C_{ pq
})\setminus W^1_5 (C_{ pq })$. We need to prove that the Zariski
tangent space to $\cW^1_5$ at $M$ is three-dimensional, i.e., it is
equal to the Zariski tangent space to $\oW^1_5 (C_{ pq })$ at
$M$. By Lemma \ref{lemWflat}, the morphism $\cW^1_5\ra T$ is flat and its
scheme-theoretical fiber at $0$ is $\oW^1_5 (C_{ pq })$ with its
reduced structure. So the tangent space to $\oW^1_5 (C_{ pq })$ at
$M$ is the kernel of the differential of the map $\cW^1_5\ra T$ at
$M$ and it is equal to the tangent space to $\cW^1_5$ if and only
this differential is zero. Now the fact that this differential is zero
follows from the fact that the differential of the map $\cJ^5\ra T$ is
zero at $M$ because $\cJ^5$ is smooth and $\cJ^5\ra T$ is flat.
\end{proof}

\subsection{The surface $\tG^1_5$ in the Wirtinger double cover case}\label{subsecG15}

Let $P^1_5$ be the universal $\bP^1$ bundle over $\oW^1_5(C_{pq})$ whose fiber over $M\in \overline{W}^1_5(C_{pq})$ is $\bP H^0(C_{pq},M)$. The following fibered diagram is the limit of the analogous diagram in the smooth case:
\begin{eqnarray}\label{singularB}\xymatrix{B\ar[r]\ar[d]&\tC_{pq}^{(5)}\ar[d]\\
P^1_5\ar[r]\ar[d]&C_{pq}^{(5)}\\
\oW^1_5(C_{pq}).&}
\end{eqnarray}

The horizontal map in the second row is as follows. If $M\in W^1_5(C_{pq})$ and $0\ne s\in H^0(C_{pq},M)$, then the image of $s$ in $C_{pq}^{(5)}$ is $div(s)=\nu_*(div(\nu^*s))$. If $M\in\overline{W}^1_5(C_{pq})\setminus W^1_5(C_{pq})$, then $M=\nu_*g^1_4$ and the image of $s\in H^0(C_{pq},M)=H^0(C,g^1_4)$ is $\nu_*(div(s)+p)=\nu_*(div(s)+q)\in C_{pq}^{(5)}$.

\begin{lemma}\label{lemtGpq}

The surface $\tG^1_5 (C_{pq})$ is the union of two copies of $W_{pq}$, denoted $W_1$ and $W_2$,
where $X_{kp} = W^1_4 (C) +p\subset W_k$ is identified with $X_{3-k,q} =
W^1_4 (C) +q\subset W_{3-k}$ for $k=1,2$.

\end{lemma}

\begin{proof}
 
First note that $\tC_{pq}^{(5)}$ has the following irreducible components:
\[
C_1^{(5)}\cup (C_1^{(4)}\times C_2)\cup( C_1^{(3)}\times C_2^{(2)})\cup
(C_1^{(2)}\times C_2^{(3)})\cup (C_1\times C_2^{(4)})\cup C_2^{(5)}.
\]
Accordingly, for a given $M\in W^1_5 (C_{pq})$, the two
connected components $B^1_{M}$ and $B^2_{M}$ of the curve $B_M$
embed, respectively, into
\[
(C_1^{(4)}\times C_2)\cup (C_1^{(2)}\times C_2^{(3)})\cup C_2^{(5)}
\]and
\[
C_1^{(5)}\cup (C_1^{(3)}\times C_2^{(2)})\cup (C_1\times C_2^{(4)}).
\]

This first shows that $\tG^1_5 (C_{pq})$ has two irreducible
components and that the double cover $\tG^1_5 (C_{pq})\ra G^1_5
(C_{pq})$ is split away from $\nu_* W^1_4 (C)$. The claim of the lemma
over $\nu_* W^1_4 (C)$ follows from the fact that, for a fixed
$M$, the components $B^1_M$ and $B^2_M$ are exchanged by the
involution induced by that exchanging $C_1$ and $C_2$ in $\tC_{pq}$.
\end{proof}

An immediate consequence of Lemma \ref{lemtGpq} is

\begin{corollary}

For a general double cover $\tX \ra X$, the double cover $\tG^1_5 \ra G^1_5$ is nontrivial.

\end{corollary}

\subsection{The pair $(A_0,\T_0)$}\label{subsecPrymfamily}

Denote $(\cA, \T) \ra T$ the family of principally polarized Prym varieties of the above
family of double covers of curves.

Following Beauville \cite[pp. 175-176]{beauville771}, the Prym variety $A_0$ associated to the
Wirtinger cover is given by the following diagram
\[
\xymatrix{&&A_0\ar^-{\cong}[r]\ar[d]&J(C)\ar[d]&\\
1\ar[r]&\bC^*\ar[r]\ar@{=}[d]&J(\widetilde{C}_{pq})\ar[r]^-{\nu^*}\ar^-{Nm}[d]&J(\widetilde{C})\ar^-{Nm}[d]\ar[r]&0\\
1\ar[r]&\bC^*\ar[r]&J(C_{pq})\ar[r]&J(C)\ar[r]&0}
\]
where $J(\widetilde{C}_{pq})$ and $J(C_{pq})$ are the generalized
(noncompact) Jacobians, $Nm$ is the Norm map, and $\nu:\tC=C_1\amalg C_2\ra\tC_{pq}$ is the normalization map.

To obtain a canonical theta divisor in $A_0$, we fix a bidegree $(d_1,d_2)$ such that $d_1+d_2=10$ and the following holds. 

\begin{enumerate}
\item[$\Diamond$]  There exists a line bundle $N$\ on $\widetilde{C}_{pq}$ of multidegree $(d_1, d_2)$, such that $h^0(N)=0$.
\end{enumerate}
By \cite[p. 153]{beauville771}, the only bidegrees satisfying $\Diamond$ are $(6,4)$, $(4,6)$ and $(5,5)$. If there is a one parameter family of line bundles $\cN$ on $\widetilde\cX$ with $N_0=\cN|_{\tC_{pq}}$ of bidegree $(d_1,d_2)$, we can modify $\cN$ by twisting with a component of $\tC_{pq}$ such that $N_0$ has either bidegree $(6,4)$ or $(5,5)$.
\begin{proposition} We have a canonical identification $(A_0,\T_0)\cong (Pic^4C,\T_C)$.
\end{proposition}
\begin{proof}Denote
$Pic^{6,4}(\widetilde{C}_{pq})$ the principal homogeneous space over
$J(\widetilde{C}_{pq})$ parametrizing line bundles of bidegree $(6,4)$ on
$\widetilde{C}_{pq}$. We identify $A_0$ with the
subvariety of $Pic^{6,4}(\widetilde{C}_{pq})$ consisting of line
bundles $N$ such that $Nm(N)=\omega_{C_{pq}}$. Note that $Nm^{-1}(\omega_{C_{pq}})$ only has one connected component by the diagram above ($Pic^{6,4} (\tC_{pq})$ is not compact). We then define the theta divisor $\T_0$ of $A_0$ as the locus of line bundles $N\in A_0\subset Pic^{6,4}(\widetilde{C}_{pq})$ such that $h^0(N)>0$.

 We have an isomorphism $A_0\stackrel{\cong}\ra Pic^4C$ which sends $N$ to $N|_{C_2}$. This is an isomorphism because $N$ is determined by $\nu^*N$ by the diagram above and $N|_{C_2}$ determines $\nu^*N$ since $Nm(\nu^*N)=N|_{C_1}\otimes N|_{C_2}\cong \omega_C(p+q)$. We claim that if $h^0(\tC_{pq},N)\ne0$, then
$h^0(C_2,N|_{C_2})\ne0$. If not, let $0\ne s\in H^0(\tC_{pq},N)$ such that $s|_{C_2}=0$, then
 $s|_{C_1}$ vanishes at $p$ and $q$. Thus $0\ne s|_{C_1}\in
H^0(C_1,N|_{C_1}(-p-q)))$. However, since $N|_{C_1}(-p-q)\otimes N|_{C_2}\cong \omega_C$, we have
$h^0(C_1,N|_{C_1}(-p-q))=h^1(C_2,N|_{C_2})=h^0((C_2,N|_{C_2})=0$, a contradiction. Thus the canonical identification sends $\T_0$ isomorphically to $\T_C$.
\end{proof}

\subsection{The family of curves in the limit} \label{subseccW} Denote the total space of the family of the double covers
${G}_{t} := \tG^1_5 (X_t)$ by $\cG$. The space $\cG$ is an \'etale double cover of $\cW^1_5$ and therefore smooth by Proposition \ref{Wsm}.  The central fiber $G_0$ of $\cG$ is described in Lemma \ref{lemtGpq}.

Assume we
are also given a section $r : t\mapsto r_t$ of $\widetilde{\cX} \rightarrow T$. 
Let $\cF$ (resp. $\cF_r$) be the closure in $\cG\times_{T}\widetilde{\cX}^{(5)}\times_T\widetilde{\cX}^{(5)}$ of the family of fourfolds (resp. threefolds) $F_t$ (resp. $F_{r_t}$) constructed in
Section \ref{seccurves} for the general fibers over $t\neq 0$. 

 By construction, the central fiber $F_0$ of ${\cF}$ fibers over $G_0=W_1\cup W_2$ with fiber over $M\in W_1$ (resp. $W_2$) the surface $B^1_M\times B^1_{M'}$ (resp. $B^2_M\times B^2_{M'}$), where $B^1_M$, $B^1_{M'}$ (resp. $B^2_M$, $B^2_{M'}$) live in  
\[
(C_1^{(4)}\times C_2)\cup (C_1^{(2)}\times C_2^{(3)})\cup C_2^{(5)}
\]
\[(resp. \ 
C_1^{(5)}\cup (C_1^{(3)}\times C_2^{(2)})\cup (C_1\times C_2^{(4)})).
\]

\section{The degeneration of theta divisors}\label{sectheta}

Let $(\cA, \Theta) =\{(A_t,\Theta_t)\}_{t \in T}$ be a 1-parameter family of
principally polarized Abelian varieties of dimension 5 with smooth total space
$\cA$. Assume
that for $t \neq 0$, the fiber $\T_t$ of $\T$ is smooth and that the fiber of $(\cA ,\T)$
at $0$ is the polarized Jacobian $(A_0 =JC ,\T_0 = \T_C)$ of a smooth
curve $C$ of genus 5.

We will obtain information about the cohomology of $\Theta_t$ from the cohomology of
$\Theta_0$ using limit mixed Hodge structures. We shall see below that the total space
$\T$ is singular. We first need to modify the family $(\cA,\T)$ using base change and
blowups to obtain a family of theta divisors with smooth total space whose central fiber
is a divisor with simple normal crossings.

\subsection{The singularities of $\T$}\label{subsecsingt}

Denote by $\cH$ the Siegel upper half space and consider the Riemann
theta function $\te (z,\tau)$ on $\bC^5 \times \cH$. After possibly replacing $T$ with a finite cover we can assume that there is a map $\tau : T \ra \cH$
such that the family $(\cA , \T)$ is the inverse image, via $\tau$, of the universal
family of polarized abelian varieties over $\cH$. In particular, we can assume that
the family $\T$ is defined by $\{(z,t)\in \bC^5 \times T \colon
\te(z,\tau(t))=0\}$ (modulo the action of the lattice of $A_t$). Denote $F(z,t):=\te(z,\tau(t))$.

In the case we are interested in, the singularities of the special fiber $\Theta_0$ are
all double points hence the singularities of the total family $\Theta$ are at worst double
points.

We compute the singularities of $\T$ locally, using the heat equation:
$\del_{\tau_{\al\bb}}\te =\del_{z_\al z_\bb}\te$ modulo multiplication
by a constant. Here the $\tau_{\al\bb}$ are coordinates on $\cH$ and the $z_{\al}$ are coordinates on an abelian variety $A$.

Write $\tau (t) =\sum_{ 1\leq i, j\leq 5 }\lambda_{ ij } (t) \tau_{ ij}$ and let $\dot\lambda_{ ij } (t)$ denote the derivative of $\lambda_{ ij } (t)$ with respect to $t$.

Note that, since $\T_t$ is smooth for $t \neq 0$, the total space $\T$ is smooth away from the special fiber $\T_0$.

\begin{proposition}

A point $(z,0)$ is a singular point of $\T$ exactly when $(z,0)$ is a singular point of $\T_0$
such that the equation $q_z\in S^2 H^1 (\cO_{ A_0 } )^*$ of
the quadric tangent cone to $\T_0$ at $z$ vanishes on the
infinitesimal deformation direction $\dot\tau (0) :=\sum_{ 1\leq i, j\leq 5
}\dot\lambda_{ ij } (0) \tau_{ ij }\in S^2 H^1 (\cO_{ A_0 } )$ under
duality.

\end{proposition}

\begin{proof}

By the heat equation, at a point $(z,t) \in \T$ the equation of the tangent
hyperplane to $\T$ in $\cA$ is the pullback from the Siegel space of the
equation
\[
\sum_{ i=1 }^5 Z_i\del_{ z_i } \te +\sum_{ 1\leq i, j\leq 5 } T_{
ij} \del_{ z_i z_j } \te
\]
where the $Z_i$ are the coordinates on the tangent space to a fiber $A_t$ and the $T_{ij}$ coordinates on the tangent space to $\cH$ at $\tau (t)$.

This gives the equation
\[
\sum_{ i=1 }^5 \del_{ z_i } F (z,t)Z_i+\left(\sum_{ 1\leq i, j\leq 5 }\dot\lambda_{
ij} (t)\del_{ z_i z_j } F (z,t)\right)\Omega
\]
where $\Omega$ is the coordinate on the tangent space to $T$ at $t$.

So the point $(z, 0)$ is singular on $\T$ if and only if it is singular on $\T_0$ and
\[
\sum_{ 1\leq i, j\leq 5 }\dot\lambda_{ij} (t)\del_{ z_i z_j } F (z,0) = 0.
\]
Since
\[
q_z = \sum_{ 1\leq i, j\leq 5 } Z_i Z_j \del_{ z_i z_j } F (z,0)
\]
the proposition follows.
\end{proof}

The partial derivatives of $F$ are
(with summation and constant convention):
\[\begin{aligned}
\del_{t}&F(z,t)=\del_{\tau_{\al\bb}}\te(z,\tau(t)) \del_t
\tau_{\al\bb}=\del_{z_\al z_\bb}\te(z,\tau(t))\del_t
\tau_{\al\bb}=\del_{z_\al z_\bb}F(z,t)\del_t \tau_{\al\bb}\\
\del_{z_\al t}&F(z,t)=\del_t( \del_{z_\al}F(z,t))=\del_{z_\al z_\bb
z_\ga}F(z,t)\del_t \tau_{\bb\ga}\\
\del_{tt}&F(z,t) =\del_{z_\al z_\bb z_\ga z_\dd}F(z,t)\del_t
\tau_{\al\bb}\del_t \tau_{\ga\dd}+\del_{z_\al z_\bb}F(z,t)\del_{tt}
\tau_{\al\bb}.
\end{aligned}\]

\subsection{The case $A_0 \simeq J(C) \simeq Pic^4 C$}

In this case the theta-divisor $\T_0 = W^0_4 (C)$ of the special fiber is
smooth outside the curve $W^1_4 := W^1_4 (C)$ and $W^1_4$ is an ordinary double
curve on it. Therefore we have
\[
\begin{array}{rr}
&F(p,0)=0 \quad\forall p \in W^1_4\\
&\del_{z_\al}F(p,0)=0 \quad \forall \al \forall p \in W^1_4\\
 &\text{ rank$\Big(\del_{z_\al z_\bb}F(p,0)\Big)_{1\leq
 \al,\bb \leq 5} =4$, $\forall p \in W^1_4$}.
\end{array}
\]

\begin{theorem}\label{thmtendp}
For $\tau$ sufficiently general, the singularities of $\T$ consist of
ten ordinary double points.
In the case where $(\cA, \T)$ is the family of Prym varieties of a family of double covers $(\tcX, \cX)$ as in \ref{subsecPrymfamily}, the ten distinct singular points $g_1, \ldots , g_5, h_1, \ldots , h_5$ of $\T$ are the $g^1_4$'s cut on $C$ by quadrics of rank $4$ containing $C$ and its secant $\langle p+q \rangle$. In other words, $h^0 (g_i -p-q) > 0$ and $h_i = |K - g_i|$ up to relabeling.
\end{theorem}
\begin{proof}

We use the calculations in Section \ref{subsecsingt}.
The annihilator of the deformation direction
\[
\dot\tau (0) =\sum_{ 1\leq i, j\leq 5
}\dot\lambda_{ ij } (0) \tau_{ ij }\in S^2 H^1 (\cO_{ A_0 } )
\]
is a hyperplane
in $S^2 H^0 (\omega_C)$ which, for $\tau$ sufficiently general, gives
a hyperplane in the space $I_2(C)$ of quadrics containing the canonical image of $C$ and hence a line $l$ in $\bP I_2
(C)\cong\bP^2$. The quadrics of rank $4$ containing the canonical
model of $C$ are the elements of $Q$, a plane quintic in $\bP I_2
(C)$. Those whose equations vanish on $\tau$ are the elements of the
intersection $l\cap Q$ which, for $\tau$ sufficiently general,
consists of $5$ distinct points, say $q_1, \ldots , q_5$. There are ten distinct points in the
singular locus $W^1_4$ of $\T_0$ above these five points: the $g^1_4$'s cut on $C$ by the rulings of $q_1, \ldots , q_5$. Hence we see
that $\T$ has exactly ten distinct singular points.

In the case where our family of abelian varieties is a family of Prym varieties of double
covers with central fiber a Wirtinger cover, the deformation direction $\dot\tau (0)$ is
the image, via the differential of the Prym map, of the infinitesimal deformation
direction, say $\eta$, of double covers induced by the family $(\tcX, \cX)$. As the Prym
map sends the locus $\cW_6$ of Wirtinger covers in $\cR_6$ into the Jacobian locus
$\cJ_5$, its differential induces a linear map from the $1$-dimensional normal space
$N_{C_{pq}}$ to $\cW_6$ to the $3$-dimensional normal space $N_{JC}$ to $\cJ_5$. It is
well-known, see e.g. \cite[p. 45]{DonagiSmith81}, that the normal space to $\cJ_5$ at $JC$ can be
canonically identified with the dual $I_2(C)^*$ to $I_2(C)$. By \cite[p.
86]{DonagiSmith81}, the image of $\bP N_{C_{pq}}$ in $\bP I_2 (C)^* = \bP N_{JC}$ is the pencil of
quadrics containing the canonical image of $C$ together with its secant $\langle p+q
\rangle$. This is also the line that we denoted $l$ above. Therefore the points $q_1, \ldots , q_5$ are the quadrics of
rank $4$ containing $C$ and $\langle p+q \rangle$. The line $\langle p+q \rangle$ is
contained in exactly one ruling of $q_i$ and we denote $g_i$ the $g^1_4$ cut on $C$ by
that ruling. We then have $h^0 (g_i -p-q ) > 0$. The second ruling of $q_i$ cuts $h_i := |K_C - g_i|$ on $C$.

It remains to prove that the ten singular points are ordinary double
points. The degree $2$ term of the Taylor expansion of $F$ near a
singular point $(z,t)$ is (using the heat equation up to a scalar):
\[
\sum_{ 1\leq i, j\leq 5 } Z_i Z_j\del_{ z_i z_j } F +\left(\sum_{ 1\leq
i,j,k\leq 5 } Z_i\dot\lambda_{ jk } \del_{ z_i z_j z_k } F\right)\Omega
+\left(\sum_{ 1\leq i,j,k, l\leq 5 }\dot\lambda_{ ij }\dot\lambda_{ kl }\del_{ z_i
z_j z_k z_l } F+\sum_{ 1\leq i, j\leq 5 }\ddot\lambda_{ ij }\del_{ z_i z_j } F\right)\Omega^2.
\]
The first part of the above is the equation of the quadric $q_z$ which
has rank $4$. In a basis adapted to $q_z$ we have the matrix of second partials
{\Small\Small\Small \[
\left(
\begin{array}{cccccc}
0 & 1 & 0 & 0 & 0 & \sum_{ j, k=1 }^5\dot\lambda_{ jk }\del_{ 1
j k }F \\
1 & 0 & 0 & 0 & 0 & \sum_{ j, k=1 }^5\dot\lambda_{ jk }\del_{ 2
j k }F \\
0 & 0 & 0 & 1 & 0 & \sum_{ j, k=1 }^5\dot\lambda_{ jk }\del_{ 3
j k }F \\
0 & 0 & 1 & 0 & 0 & \sum_{ j, k=1 }^5\dot\lambda_{ jk }\del_{ 4
j k }F \\
0 & 0 & 0 & 0 & 0 & \sum_{ j, k=1 }^5\dot\lambda_{ jk }\del_{ 5 j k }F \\
\sum_{ j, k=1 }^5\dot\lambda_{ jk }\del_{ 1 j k }F & \sum_{
j, k=1 }^5\dot\lambda_{ jk }\del_{ 2 j k }F & \sum_{
j, k=1 }^5\dot\lambda_{ jk }\del_{ 3 j k }F & \sum_{ j,
k=1 }^5\dot\lambda_{ jk }\del_{ 4 j k }F & \sum_{ j, k=1
}^5\dot\lambda_{ jk }\del_{ 5 j k }F & \sum_{ i, j, k, l=1
}^5\dot\lambda_{ ij }\dot\lambda_{ kl }\del_{ i j k l}F \\
&&&&&+\sum_{ 1\leq i, j\leq 5 }\ddot\lambda_{ ij }\del_{ ij} F
\end{array}
\right)
\]}

So we need to see that this matrix has rank $6$ at the points of $W^1_4$. In other words, for
$\tau$ sufficiently general the coefficient $\sum_{ j, k=1
}^5\dot\lambda_{ jk }(0)\del_{ 5 j k }F(z,0)$
is not zero. Taking
$\dot\lambda_{ ij }(0) =\lambda_i\lambda_j$ such that the point $(\lambda_i)$
is on the osculating cone to $\T_0$ and is otherwise general, this means
that
the vertex of the quadric $q_z$ is not contained in the tangent
space to the osculating cone to $\T_0$ at the point $(\lambda_j)$.
For a general choice of the $\lambda_j$ as above this is a consequence of
\cite{kempfschreyer} page 353.
\end{proof}

\section{The semi-stable reduction of the family of theta divisors}\label{secsstheta}

As before denote $(\cA, \T) \ra T$ the family of principally polarized Prym varieties associated to the $\acute{e}$tale cover $\widetilde{\cX}\ra\cX$. The central fiber is the Jacobian $A_0\cong Pic^4C$ of a general curve of genus $5$. By Theorem \ref{thmtendp}, the total space $\T$ has ten ordinary double points on $W^1_4$: $g_1, \ldots, g_5$ which satisfy $h^0(g_i-p-q)>0$ and $h_i := |K_C - g_i|$. We will construct a semistable reduction of $\T$ and, in Section \ref{seccohtheta}, use the Clemens-Schmid exact sequence to compute the cohomology of $\Theta_t$.

\subsection{The base change and first blow-ups}\label{rkcentralfiber} To construct our semistable
reduction, we first make a base change of degree $2$, then resolve singularities. Let $T^{b} \ra T$ be a degree $2$ cover. After possibly shrinking $T$, we assume that the cover $T^b\ra T$ has a unique branch point which is $0\in T$.
 Pulling back, we obtain the family $\T^{b}\subset\cA^b \ra T^b$ and $\T^b$ is singular along $W^1_4\subset\T_0$. We define $\tT$ as the blow up of $\T^b$
along its singular locus $W^1_4$. We will see that $\tT$ is a resolution of $\T^b$ whose special fiber $\tT_0$ is a simple normal crossings divisor ($\cA^b\ra T^b$ is still a smooth family):

\[
\begin{array}{ccccc}
\widetilde{\T} & \lra &\T^b & \lra & \Theta \\
&&\downarrow & & \downarrow \\
&&T^b & \lra & T.
\end{array}
\]

To make our family of curves compatible with the base change, we also need to make a base change of order two on $\cG$ and then blow up along the singular locus of $\cG^b$ to obtain a semistable family. The resulting space is $\widetilde{\cG}$.
\[
\begin{array}{ccccc}
\widetilde{\cG} & \lra &\cG^b & \lra & \cG \\
&&\downarrow & & \downarrow \\
&&T^b & \lra & T.
\end{array}
\]

Recall that, by Lemma \ref{lemtGpq}, the fiber of $\cG$ at $t=0$ is the union of two copies of $W_{pq}$, denoted $W_1$ and $W_2$, where $X_{kp} = W^1_4 (C) +p \subset W_k$ is identified with $X_{3-k,q} =W^1_4 (C) +q \subset W_{3-k}$. After blowing up along the singular locus $X_{1p}\amalg X_{1q}\subset\cG^b$,
the central fiber $\tG_0$ of $\widetilde{\cG}$ has four components: $W_1$, $W_2$, $P_1$ and $P_2$, where $P_1$, resp. $P_2$, is a $\bP^1$ bundle over $X_{1p} = X_{2q}$, resp. $X_{1q} = X_{2p}$. The four components meet as below:

$$
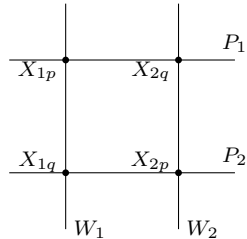
\begin{tikzpicture}[xscale=1.5,yscale=1.5][font=\tiny]
  \coordinate (a) at (0,0); 
  \coordinate (b) at (1,0);
  \coordinate (c) at (1,1); 
  \coordinate (d) at (0,1); 
  \path[fill=black](a)circle[radius=0.03];
  \path[fill=black](b)circle[radius=0.03];
  \path[fill=black](c)circle[radius=0.03];
  \path[fill=black](d)circle[radius=0.03];
  \node[left] at ($(a)!1/10!(d)$){$X_{1q}$};
  \node[left] at ($(d)!1/10!(a)$){$X_{1p}$};
  \node[left] at ($(b)!1/10!(c)$){$X_{2p}$};
  \node[left] at ($(c)!1/10!(b)$){$X_{2q}$};
  \draw ($(a)!1+1/2!(b)$)--($(b)!1+1/2!(a)$);
  \draw ($(b)!1+1/2!(c)$)--($(c)!1+1/2!(b)$);
  \draw ($(c)!1+1/2!(d)$)--($(d)!1+1/2!(c)$);
  \draw ($(d)!1+1/2!(a)$)--($(a)!1+1/2!(d)$);
  \node [above] at ($(a)!1+1/2!(b)$){$P_2$};
  \node [above] at ($(d)!1+1/2!(c)$){$P_1$};
  \node [right] at ($(d)!1+1/2!(a)$){$W_1$};
  \node [right] at ($(c)!1+1/2!(b)$){$W_2$};

\end{tikzpicture}\\$$

\begin{notation}\label{notTb}
From now on we will replace the family $\T\subset\cA \ra T$ by the new family $\T^b\subset\cA^b \ra T^b$, $\cG \ra T$ by $\cG^b \ra T^b$ by their respective base changes to $T^b$. We also replace the family $\cF$ and $\tX\ra X$ by the corresponding spaces after base change.

\end{notation}

\subsection{The central fiber $\widetilde{\T}_0$}
\begin{proposition}\label{proptheta0}

The total space $\widetilde \T$ is smooth. Its special fiber $\tT_0$ is a divisor with
simple normal crossings with the following two irreducible components.
\begin{enumerate}
\item The component $M_1$ which is the blowup of $\T_0$ along $W^1_4$.
\item The component $M_2$ which is the exceptional divisor, i.e. the projectivized normal
cone to $\T$ along $W^1_4$. Therefore $M_2$ is a fibration over $W^1_4$. At the points
$g_i$ and $h_i$ the fibers $Q_{3i}^{sing}$ of $M_2$ are isomorphic to the singular quadric $Q_3^{sing}$ of
rank $4$ in $\bP^4$. At all the other points of $W^1_4$, the fibers of $M_2$ are
isomorphic to the smooth quadric hypersurface $Q_3\subset\bP^4$.
\end{enumerate}
The intersection $M_{12}=M_1\cap M_2$ is a $\bP^1\times\bP^1$ bundle over $W^1_4$. In
particular, it is smooth.
\end{proposition}

\begin{proof}

It immediately follows from the definition of $\tT$ that $\tT_0$ has two components, one
of which is the blow up $M_1$ of $\T_0$ and the other the projectivized normal cone $M_2$
of $W^1_4$ in $\T$. To prove the assertions about the smoothness of $M_1, M_2$ and
$M_{12}$ and the fibers of $M_2$ over $W^1_4$ we work in local coordinates near each of
the ten points $g_i$ and $h_i$ of Theorem \ref{thmtendp}.

By Theorem \ref{thmtendp}, before our base change of degree $2$ the local equation of $\T$
near one of the points $g_i$ or $h_i$ can be written as $xy+zw+st=0$ in $\mathbb{A}^6$
where $t$ is a local analytic coordinate on $T$ centered at $0$. Hence, after base change,
the local equation is
\[
xy+zw+st^2=0.
\]
In the above coordinates, the equations of $W^1_4$ are
\[
x=y=z=w=t=0.
\]
Hence, locally, $\widetilde \T$ is obtained from $\T$ by blowing up the ideal
$\mathcal{I}=(x,y,z,w,t)$. Furthermore, $s$ gives a local coordinate on $W^1_4$. For any
given nonzero value of $s$, the local equation of $\T$ defines a quadric of rank $5$ in
$\bP^4$ and, for $s=0$, the equation defines a quadric of rank $4$ in $\bP^4$. This proves
the assertions about the fibers of $M_2 \ra W^1_4$.

Now let $X$, $Y$, $Z$, $W$, $T$ be the homogeneous coordinates on the blow-up. We have the relations
\[
\text{rank}\left (\begin{matrix} x &y& z& w& t \\ X& Y& Z& W& T \end{matrix}\right )\le1.
\]
By symmetry, we only need to check the following cases.
\begin{enumerate}
\item $\{X\ne0\}$. $\widetilde \T$ is locally isomorphic to
\[
Spec\ \frac{\bC[x,Y,Z,W,T,s]}{(Y+ZW+sT^2)},
\]
which is clearly smooth. $\widetilde{\T}_0$ is given by the equation $t=0$, i.e., $xT=0$,
hence has two smooth components meeting transversely, defined locally by the equations
$T=0$ and $x=0$. The equation $T=0$ locally defines the component $M_1$ while $x=0$
locally defines $M_2$.
\item $\{T\ne0\}$. $\widetilde \T$ is locally isomorphic to
\[
Spec\ \frac{\bC[X,Y,Z,W,t,s]}{(XY+ZW+s)}.
\]
In this open subset, the total space and the central fiber are both smooth and $t=0$
locally defines the component $M_1$.
\end{enumerate}
\end{proof}

\begin{proposition}\label{propM1}
\begin{enumerate}
\item
The divisor $M_1$ can be identified with the correspondence
\[
M_1=\{(D_4,B_4)\in C^{(4)}\times C^{(4)}|\ D_4+B_4\in |K_C|\}
\]
with the two projections $p_1$ and $p_2$ to $C^{(4)}$. We have a fibered diagram  
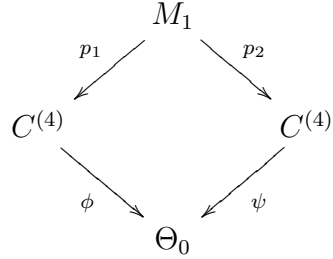
\[
\xymatrix{&M_1\ar[ld]_-{p_1}\ar[rd]^-{p_2}&\\
C^{(4)}\ar[rd]_-{\phi}&&C^{(4)}\ar[ld]^-{\psi}\\
&\T_0}
\]
where $\phi$ is the natural map
sending $D_4$ to $\cO_C(D_4)$ and $\psi$ sends $B_4$ to $\omega_C(-B_4)$.

\item Both $p_1$ and $p_2$ are birational
morphisms and can be realized as the blow-up of $C^{(4)}$ along the smooth surface
\[
C^1_4=\{D\in C^{(4)}\ |\ h^0(\cO_C(D))\ge2\}.
\]

\item The double locus $M_{12}$ is the fiber product
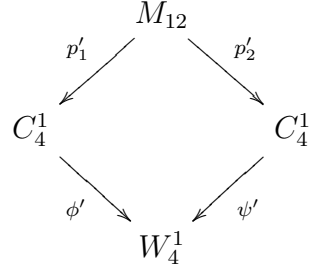
\[
\xymatrix{&M_{12}\ar[ld]_-{p'_1}\ar[rd]^-{p'_2}&\\
C^1_4\ar[rd]_-{\phi'}&&C^{1}_4\ar[ld]^-{\psi'}\\
&W^1_4}
\]

\end{enumerate}

\end{proposition}

\begin{proof}
Immediate.
\end{proof}

\section{General facts about the Clemens-Schmid exact sequence}\label{secgeneralMH}

We briefly review some general facts about the Clemens-Schmid exact sequence in this section. We will apply the general theory in this section to compute the cohomology of $\widetilde{\T}_0$ and $\widetilde{\T}_t$ in Section \ref{seccohtheta}.
\subsection{The Clemens-Schmid exact sequence}
Given any semistable degeneration
$$\xymatrix{Y_0\ar[r]\ar[d]&\cY\ar[d]\\
\{0\}\ar[r]&T}$$ 
of relative dimension $n$, where $\cY$ deformation retracts to $Y_0$,
 denote 
\begin{eqnarray}
 &&H^m_t :=H^m(Y_t,\bQ),\nonumber\\
 &&H^m :=H^m(\cY,\bQ)\cong H^m(Y_0,\bQ),\nonumber\\
 &&H_m :=H_m(\cY,\bQ)\cong H_m(Y_0,\bQ).\nonumber
\end{eqnarray}

It follows from the work of Clemens-Schmid \cite{clemens77}, \cite{schmid73} and
Steenbrink \cite{steenbrink75} that one can define mixed Hodge structures on $H^*_t$,
$H^*$ and $H_*$ such that we have an exact sequence of mixed Hodge structures

\begin{eqnarray}\label{eqCS}
\xymatrix{\ar[r]&H_{2n+2-m}\ar[r]^-{\alpha}&H^m\ar[r]^-{i_t^*}&H^m_t\ar[r]^-{N}&H^m_t\ar[r]^-{\beta}&
  H_{2n-m}\ar[r]^-{\alpha}
  &H^{m+2}\ar[r]&}
\end{eqnarray}
where $N$ is the logarithm of the monodromy operator, $i_t : Y_t \inj \cY$ is the inclusion
of the general fiber into the total space, $\alpha$ is the composition
\[
\xymatrix{H_{2n+2-m}(\cY)\ar[r]^-{\text{PD}}&H^m(\cY,\partial\cY)\ar[r]&H^m(\cY),}
\]
where $PD$ denotes Poincar\'e Duality, and $\beta$ is the composition
\[
\xymatrix{H^m(Y_t)\ar[r]^-{\text{PD}}&H_{2n-m}(Y_t)\ar[r]^-{i_{t*}}&H_{2n-m}(\cY).}
\]

\subsection{The weight filtrations on $H^m$ and $H_m$}\label{subsecspseq} Recall from \cite[p. 103]{morrison84} that there is a Mayer-Vietoris type spectral sequence abutting to $H^{\bullet}
(Y_0)$ with $E_1$ term 
$$E_1^{p,q} = H^q (Y_0^{[p]}).$$
Here $Y_0^{[p]}$ is the disjoint union of the codimension $p$ strata of $Y_0$, i.e.,
\[
Y_0^{[p]} := \coprod_{i_0, \ldots , i_p} Z_{i_0}\cap \ldots \cap Z_{i_p}
\]
where the $Z_{i_j}$ are distinct irreducible components of $Y_0$. 
 
 The differential $d_1$
$$\xymatrix{E_1^{p,q}\ar[d]^-{\cong}\ar[r]^-{d_1}&E_1^{p+1,q}\ar[d]^-{\cong}\\
H^q(Y_0^{[p]})\ar[r]^-{d_1}&H^q(Y_0^{[p+1]})}$$
is the alternating sum of the restriction maps on all the irreducible components.
By \cite[p. 103]{morrison84} this sequence degenerates at $E_2$. 

 The weight filtration is given by
$$W_kH^m := \oplus_{p+q=m,\ q\le k} E_{\infty}^{p,q} = \oplus_{p+q=m,\ q\le k} E_{2}^{p,q}.$$
Therefore the weights on $H^m$ go from $0$ to $m$ and 
$$Gr_kH^m\cong E_2^{m-k,k}=\frac{Ker(d_1: H^{k}(Y^{[m-k]})\ra H^{k}(Y^{[m-k+1]})}{Im(d_1: H^{k}(Y^{[m-k-1]})\ra H^{k}(Y^{[m-k]})}.$$

We also put a weight filtration on $H_m$:
$$W_{-k}H_m:=(W_{k-1}H^m)^{\perp}$$
under the perfect pairing between $H^m$ and $H_m$. With this definition,
$$Gr_{-k}H_m\cong (Gr_k H^m)^{\lor}.$$

\subsection{The monodromy weight filtration on $H^m_t$} 
Associated to the nilpotent operator $N$ is an increasing filtration of $\bQ$-vector spaces 
$$0\subset W_0\subset W_1\subset...\subset W_{2m}=H^m_t.$$

Let $K^m_t :=Ker N\subset H^m_t$ be the monodromy invariant subspace. It inherits an induced weight filtration from $H^m_t$.  We refer to
\cite[pp. 106-109]{morrison84} for the precise definition of the monodromy weight filtration and the fact that this filtration on $H^m_t$ can be computed via its induced filtration on $K^m_t$:
\begin{eqnarray}\label{sl2}Gr_kH^m_t\cong Gr_kK^m_t\oplus Gr_{k-2}K^m_t\oplus...\oplus Gr_{k - 2 \lfloor\frac{k}{2}\rfloor}K^m_t
\end{eqnarray}
for $k\le m$,
and
\begin{eqnarray}\label{sym}Gr_kH^m_t\cong Gr_{2m-k}H^m_t,
\end{eqnarray}
for $k>m$.

The weight filtrations on $H^m$ and $K_t^m$ are related by the Clemens-Schmit exact sequence. Below are the basic facts we will use (see \cite[pp. 107-109]{morrison84})

\begin{enumerate}
\item $i_t^*$ induces an isomorphism
\begin{eqnarray}\label{Grsmallk}\xymatrix{Gr_kH^m\ar[r]^-{\cong}& Gr_kK^m_t}\ for\ k\le m-1.
\end{eqnarray}
 \item There is an exact sequence
\begin{eqnarray}\label{seqGR4}\xymatrix{0\ar[r]&Gr_{m-2}K^{m-2}_t\ar[r]&Gr_{m-2n-2}H_{2n+2-m}\ar[r]^-{\alpha}&Gr_mH^m\ar[r]&Gr_mK^m_t\ar[r]&0}.
\end{eqnarray}
\end{enumerate}
\subsection{Mixed Hodge structures on $H_c^{\bullet}(\cY)$}\label{MHScompact}
Now suppose further more that $\cY$ is an {\bf analytic} open subset of a smooth projective variety $\overline{\cY}$ of dimension $n+1$. We have a sequence of isomorphisms
$$ H_c^{2n+2-m}(\cY)\cong H^{2n+2-m}(\overline{\cY},\overline{\cY}\setminus \cY)\cong H^{2n+2-m}(\overline{\cY},\overline{\cY}\setminus Y_0),$$
where the last isomorphism follows from the fact that $\overline{\cY}\setminus Y_0$ deformation retracts to $\overline{\cY}\setminus \cY$.

Both $H^{\bullet}(\overline{\cY})$ and $H^{\bullet}(\overline{\cY}\setminus Y_0)$ admit canonical mixed Hodge structures (\cite{deligne74}, \cite[1022-1024]{durfee83}).  The relative singular cochain complex $S^{\bullet}(\overline{\cY},\overline{\cY}\setminus Y_0)$ is quasi isomorphic to the mapping cone of the chain map
$$S^{\bullet}(\overline{\cY})\ra S^{\bullet}(\overline{\cY}\setminus Y_0).$$
Using a standard mapping cone construction (see, for instance, \cite[pp. 1205-1207]{durfee83}),  we can put a canonical mixed Hodge structure on $H^{\bullet}(\overline{\cY},\overline{\cY}\setminus Y_0)$, and therefore on $H^{\bullet}_c(\cY)$, such that the maps in the long exact sequence 
\begin{eqnarray}\label{relativecoh}\xymatrix{...\ar[r]&H^{m-1}(\overline{\cY}\setminus Y_0)\ar[r]&H^{m}(\overline{\cY},\overline{\cY}\setminus Y_0)\ar[r]&H^m(\overline{\cY})\ar[r]&
H^m(\overline{\cY}\setminus Y_0)\ar[r]&...}
\end{eqnarray}
are morphisms of mixed Hodge structures.

There is also a spectral sequence \cite[pp. 1025-1027]{durfee83} for the mapping cone, dual to the Mayer-Vietoris type spectral sequence in Section \ref{subsecspseq}, abutting to $H^{\bullet}_c(\cY)$. This spectral sequence is in the second quadrant, degenerates at $E_2$, and has $E_1$ terms
$$E_{1,c}^{p,q}=H^{q+2p-2}(Y_0^{[-p]}),$$
for $p\le0$. The differential
$$\xymatrix{E_{1,c}^{p,q}\ar[r]\ar@{=}[d]&E_{1,c}^{p+1,q}\ar@{=}[d]\\
H^{q+2p-2}(Y_0^{[-p]})\ar[r]^-{d^c_1}&H^{q+2p}(Y_0^{[-p-1]})}$$
is the alternating sum of Gysin morphisms. 
We have the duality
$$(E_1^{p,q})^{\lor}\cong E_{1,c}^{-p,2n+2-q}$$

 The increasing weight filtration is given by
$$W_kH^m_c(\cY) = \oplus_{p+q=m,q\le k}E_{2,c}^{p,q}.$$
The weights on $H^m_c(\cY)$ go from $m$ to $2m-2$ and, for $m\le k\le 2m-2$,
$$Gr_{k}H^m_c({\cY})\cong E_{2,c}^{m-k,k}=\frac{Ker(H^{2m-k-2}(Y_0^{[k-m]})\ra H^{2m-k}(Y_0^{[k-m-1]}))}{Im (H^{2m-k-4}(Y_0^{[k-m+1]})\ra H^{2m-k-2}(Y_0^{[k-m]}))}$$
with the convention that $Y_0^{[-1]}=\varnothing$.

The mixed Hodge structures on $H^m(\cY)$ and $H^{2n+2-m}_c(\cY)\cong H^{2n+2-m}(\overline{\cY},\overline{\cY}\setminus Y_0)$ are dual to each other. We have
$$Gr_kH^m(\cY)^{\lor}\cong Gr_{2n+2-k}H^{2n+2-m}_c(\cY).$$

\section{The monodromy weight filtration on the cohomology of $\T_t$}
\label{seccohtheta}
We apply the general theory in Section \ref{secgeneralMH} to the case $\cY=\widetilde{\T}$ to compute the cohomology of $\T_t$ in this section.
By the Hard Lefschetz Theorem
\begin{eqnarray}H^m(\T_t)\cong H^{8-m}(\T_t)
\end{eqnarray}
and, by the Lefschetz Hyperplane Theorem,
\begin{eqnarray}\label{invariant}H^m(\T_t)\cong H^m(A_t)\cong\bQ^{10\choose m}
\end{eqnarray}
for $m\leq 3$. The only remaining case is the middle cohomology $H^4(\T_t)$. We will describe the monodromy weight filtration on it. 

According to the general theory explained in Section \ref{secgeneralMH}, in order to compute the monodromy weight filtration on the cohomology of $\T_t$ we first need to compute the cohomology of the central fiber $\widetilde{\T}_0 = M_1 \cup M_2$.

\subsection{The cohomology of the strata of $\widetilde{\T}_0$} \label{subsecstrata}
In this subsection we compute
the cohomology of $M_1$, $M_2$ and $M_{12}$ and describe their generators.
The various spaces fit into the commutative diagram with Cartesian squares
\begin{eqnarray}\label{relation}\xymatrix{M_2\ar@/
  _2pc/[rdd]_-{\pi_2}&M_{12}\ar@{_{(}->}[l]_-{j_2}\ar@{^{(}->}[r]^-{j_1}\ar[d]^-{p'_1}\ar@/_1pc/[dd]_-{\pi_{12}}&M_1\ar[d]^-{p_1}\\
  &C^1_4\ar@{^{(}->}^-{l}[r]\ar[d]^-{\phi'}&C^{(4)}\ar[d]^-{\phi}\\
  &W^1_4\ar@{^{(}->}^-{h}[r]&Pic^4(C)=A_0}
\end{eqnarray}
where we denote $p'_1$ (resp. $\phi'$) the restriction of $p_1$ (resp. $\phi$)  to $M_{12}$ (resp. $C^1_4$) and $j_k: M_{12}\ra M_k$ the inclusion map.

\begin{lemma}\label{lemstrata}We have the following table of Betti numbers.
\begin{table}[h]
\caption{}
\begin{tabular}{|c|ccccccccc|}\hline
 &$h^0$&$h^1$&$h^2$&$h^3$&$h^4$&$h^5$&$h^6$&$h^7$&$h^8$\\\hline
$C^1_4$&1&22&2&22&1&0&0&$0$&0\\\hline
$C^{(4)}$&1&10&46&130&256&130&46&10&1\\\hline
$M_1$&1&10&47&152&258&152&47&10&1\\\hline
$Q_3$&1&0&1&0&1&0&1&$0$&0\\\hline
$Q_3^{sing}$&1&0&1&0&2&0&1&0&0\\\hline
$M_{12}$&1&22&3&44&3&22&1&$0$&0\\\hline
$M_2$&1&22&2&22&12&22&2&22&1\\
\hline
   \end{tabular}
\end{table}
\end{lemma} 

\begin{proof} This is a straightforward computation so we only sketch the idea.
\begin{enumerate}
\item By Proposition \ref{propM1}, $M_1$ is the blow up of $C^{(4)}$ along $C^1_4$. So we have
$H^{\bullet}(M_1)= p_1^*H^{\bullet}(C^{(4)})\oplus\
j_{1*}{p'_1}^*H^{\bullet-2}(C^1_4)$. The cohomology of $C^{(4)}$ was computed by Macdonald
\cite{macdonald62}:
  $$H^{k}(C^{(4)})= \oplus_{\beta=0}^{\lceil\frac{k}{2}\rceil}\eta^\beta\cdot H^{k-2\beta}(Pic^4C).$$
  
\item Since $M_{12}$ (resp. $C^1_4$) is a smooth fibration over $W^1_4$ with fibers
$\bP^1\times\bP^1$ (resp. $\bP^1$), we can apply the Leray spectral sequence to 
$\pi_{12}:M_{12}\ra W^1_4$ (resp. $\phi'$)
 to compute the cohomology of $M_{12}$ and $C^1_4$.

\item The variety $M_2$ is a fibration over $W^1_4$ with general fiber isomorphic to the smooth quadric threefold $Q_3$ and ten special
fibers isomorphic to the singular quadric $Q_3^{sing}$ of rank $4$. Since the base is a curve, the Leray spectral sequence for $\pi_2$ degenerates at
$E_2$.
\end{enumerate}
We present the Leray spectral sequence computation for $H^4(M_2,\bQ)$, the other
cohomology groups are similar and somewhat easier to compute. The $E_2$ terms are
\[
E_2^{p,q}=H^p(W^1_4,R^q\pi_{2*}\bQ).
\]
Let $U\subset W^1_4$ be a small analytic disc, open neighborhood of a critical value of
$\pi_2$. Then $\pi_2^{-1}(U)$ is homotopic to a smooth fiber $\pi^{-1}_2(t)=Q_3$ with a real $4$ cell
$B^4$ attached to $\pi_2^{-1}(t)$ along a vanishing sphere $S^3$. Since $h^3(Q_3)=0$, this
amounts to increasing $h^4$ by $1$. Thus
\[
R^q\pi_{2*}\bQ\cong\left\{
\begin{array}{ll}
\bQ\oplus(\oplus_{i=1}^{10}\bQ_i)&\quad q=4,\\
\bQ& \quad q=2,\\
0& \quad otherwise,\\
\end{array}
\right.
\]
where $\bQ_i$ is the skyscraper sheaf with stalk $\bQ$ supported at the $i$-th critical point. Therefore

\[
\dim_{\bQ}E_2^{p,4-p}=\dim_{\bQ}E_\infty^{p,4-p}=h^p(W^1_4,R^{4-p}\pi_{2*}\bQ)=\left\{
\begin{array}{ll}
11&\quad p=0,\\
1 & \quad p=2,\\
0& \quad otherwise.\\
\end{array}
\right.
\]

This gives $h^4(M_2)=12$.
\end{proof}

 \begin{notation}\label{notPi}
Denote $e_k\in H^2(M_k)$ the class of $M_{12}$
in $M_k$, $f\in H^2(M_{12})$ the class of a fiber $\pi_{12}^{-1}(t)$ and
$\tau_1={p_1'}^*l^*\eta\in H^2(M_{12})$ (see Diagram \eqref{relation}). Here $l^*\eta$ is represented by the curve $C^1_4\cap (x+C^{(3)})\subset C^1_4$ for $x\in C$ general, and $\tau_1$ is represented by a $\bP^1$ bundle over this curve.
The product $\tau_1\cdot f\in H^4(M_{12})$ is the class of the ruling of 
$\pi_{12}^{-1}(t)\cong\bP^1\times\bP^1$ which projects to a point under $p_1$, and the product  $j_2^*e_2\cdot f$ is the hyperplane class in
$\pi_{12}^{-1}(t)$. We also have the relation
$$-j_1^*e_1=j_2^*e_2.$$

Furthermore, denote $[\bP^2_i]$, $i= 1, \ldots , 5$ (resp. $i= 6, \ldots , 10$) the class of the projective plane spanned by a line of the ruling corresponding to $\tau_1\cdot f$ and the vertex of the singular quadric $Q_{3i}^{sing}=\pi_2^{-1}(g_i)$ (resp. $Q_{3i}^{sing}=\pi_2^{-1}(h_{i-5})$).
\end{notation}

\begin{lemma}\label{Mcoho}We have the following generators for the cohomology:
\begin{eqnarray}H^2(M_{12})= & \langle f,\ \tau_1,\ j_2^*e_2\rangle &\cong\bQ^3,\nonumber\\
H^3(M_{12})= & \tau_1\cdot\pi_{12}^*H^1(W^1_4)\oplus j_2^*e_2\cdot \pi_{12}^*H^1(W^1_4) & \cong\bQ^{44},\nonumber\\
H^4(M_{12})= & \langle f\cdot \tau_1,\ f\cdot j_2^*e_2,\ \tau_1\cdot j_2^*e_2\rangle & \cong\bQ^3,\nonumber\\
H^3(M_1)= & p_1^*H^3(C^{(4)})\oplus j_{1*}\pi_{12}^*H^1(W^1_4) &\cong\bQ^{152},\nonumber\\
H^4(M_1)= & p_1^*H^4(C^{(4)})\oplus\langle j_{1*}f, j_{1*}\tau_1\rangle & \cong\bQ^{258},\nonumber\\
H^3(M_2)= & e_2\cdot \pi_{2}^*H^1(W^1_4) & \cong\bQ^{22}\nonumber\\
H^4(M_2)= & \langle [\bP^2_i],\ j_{2*}f,\ j_{2*}\tau_1\ |\ i=1,...,10\rangle & \cong\bQ^{12}\nonumber.
\end{eqnarray}

\end{lemma}

\begin{proof} The statements about $M_1$ follow from the formula for the cohomology of a blow-up and the statements about $M_2$ and $M_{12}$ follow from the Leray spectral sequence.
\end{proof}

\subsection{The cohomology of $\tT_0$}

Recall that $Q\subset \bP I_2 (C)$ is the plane quintic parametrizing quadrics of rank $4$ and also the quotient of $W^1_4$ by the involution exchanging $g^1_4$ with $|K_C - g^1_4|$ (see the proof of Theorem \ref{thmtendp}). We have

\begin{proposition}\label{propwH4}The weight filtration on $H^4:= H^4 (\tT_0) = H^4( M_1 \cup M_2)$ is as follows:
\begin{eqnarray}Gr_kH^4= & 0 & \text{for}\ k\le 2;\nonumber\\
Gr_3H^4\cong & \frac{H^1(W^1_4)}{h^*H^1(Pic^4C)} \cong H^1 (Q) & \cong\bQ^{12};\nonumber\\
Gr_4H^4= & Ker(H^4(M_1)\oplus H^4(M_2)\stackrel{j_1^*-j_2^*}\lra H^4(M_{12})) & \cong\bQ^{267}.\nonumber
\end{eqnarray}
\end{proposition}
\begin{proof}We apply the spectral sequence of Section \ref{subsecspseq}, which degenerates at $E_2$, to the case $\cY=\widetilde{\T}$. Since $\widetilde{\T}_0=M_1\cup M_2$, the $E_1$ term of the spectral sequence has only two nonzero columns corresponding to $p=0$ and $p=1$. Thus from the definition of the weight filtration, we obtain
\begin{eqnarray}
&&Gr_kH^4\cong E_2^{4-k,k}=0\ \text{for }k\le2,\nonumber\\
 &&Gr_3H^4\cong E_2^{1,3}= Coker(H^3(M_1)\oplus H^3(M_2)\stackrel{j_1^*-j_2^*}\lra H^3(M_{12})),\nonumber\\
&&Gr_4 H^4\cong E_2^{0,4}=Ker(H^4(M_1)\oplus H^4(M_2)\stackrel{j_1^*-j_2^*}\lra H^4(M_{12})).\nonumber
\end{eqnarray}

We compute $E_2^{1,3}$ in Lemma \ref{d103}. By Lemma \ref{d104}, the image of $j_1^*-j_2^*$ is equal to $H^4(M_{12})$, therefore $E_2^{0,4}\cong\bQ^{267}$ by a dimension count. 
\end{proof}

\begin{lemma} \label{d103} We have isomorphisms $Coker(H^3(M_1)\oplus H^3(M_2)\stackrel{j_1^*-j_2^*}\lra H^3(M_{12}))\cong\frac{H^1(W^1_4)}{h^*H^1(Pic^4C)}\cong H^1 (Q)\cong\bQ^{12}$.
\end{lemma}

\begin{proof} 

By Lemma \ref{Mcoho},
\[
\xymatrix{H^3(M_1){=}p_1^*H^3(C^{(4)})\oplus j_{1*}\pi_{12}^{*}H^1(W^1_4)}
\]
and, by \cite[p. 325]{macdonald62},
$$H^3(C^{(4)})= H^3(Pic^4(C))\oplus \eta\cdot H^1(Pic^4C).$$

Note that
\[
\xymatrix{H^3(Pic^4C)\ar[rr]^-{p_1^*\circ\ \phi^*}&&H^3(M_1)\ar[r]^-{j_1^*}&H^3(M_{12})}
\]
is zero since $\phi\circ p_1\circ j_1= h\circ\phi'\circ p_1'$ (see Diagram \ref{relation}) and $H^3(W^1_4)=0$.

Furthermore, we see from Lemma \ref{Mcoho} that the image of  
$$\xymatrix{H^3(M_2)=e_2\cdot\pi_2^*H^1(W^1_4)\ar[r]^-{j_2^*}&H^3(M_{12})}$$
is equal to $j^*_1(j_{1*}\pi_{12}^{*}H^1(W^1_4))=j_1^*e_1\cdot\pi_{12}^*H^1(W^1_4)$.
This is because
$$j_1^*\circ j_{1*}=j_1^*e_1\cup\bullet=- j_2^*e_2\cup\bullet.$$

Therefore we have
$$Coker(j^*_1-j_2^*)\cong \frac{\tau_1\cdot\pi_{12}^*H^1(W^1_4)}{j_1^*p_1^*(\eta\cdot H^1(Pic^4C))}\cong\frac{H^1(W^1_4)}{h^*H^1(Pic^4C)}\cong H^1 (Q) \cong \bQ^{12}.$$
\end{proof}

\begin{lemma}\label{d104} 
The map $j_2^*:H^4(M_2)\ra H^4(M_{12})$ acts as follows
 \begin{eqnarray} j_2^* : j_{2*}f  & \longmapsto & f\cdot j_2^*e_2,\nonumber\\
 j_{2*}\tau_1 & \longmapsto & \tau_1\cdot j_2^*e_2,\nonumber\\
 { [\bP^2_i]} & \longmapsto & f\cdot \tau_1\ \text{for}\ i=1,...,10.\nonumber
\end{eqnarray}

The map $j_1^*:H^4(M_1)=p_1^*H^4(C^{(4)})\oplus\langle j_{1*}f,j_{1*}\tau_1\rangle\ra H^4(M_{12})$ acts as follows
\begin{eqnarray} j_1^* : -j_{1*}f & \longmapsto & f\cdot j_2^*e_2,\nonumber\\
-j_{1*}\tau_1 & \longmapsto & \tau_1\cdot j_2^*e_2,\nonumber\\
p_1^*\omega & \longmapsto & {p_1'}^*l^*\omega\in\bQ f\cdot\tau_1,\  \forall\ \omega\in H^4(C^{(4)}).
\end{eqnarray}

As a consequence, $j^*_k: H^4(M_k)\ra H^4(M_{12})$ is surjective for $k=1,2$.
\end{lemma}
\begin{proof}The Lemma follows from the formula
$$j_k^*\circ j_{k*}=-\cup j_k^*e_k$$
for $k=1,2$
and the definition of $[\bP^2_i]$ (see Notation \ref{notPi}).

\end{proof}

\subsection{The monodromy weight filtration on $H^4_t$} 
\begin{proposition}\label{H4theta} The weight filtration on $H^4(\T_t)$ is as follows:

\begin{enumerate}\item $Gr_kH^4_t=0,\ \text{for}\ k\le2,\ \text{or} \ k\ge6$.
\item $Gr_5H^4_t\cong Gr_3H^4_t=i_t^*Gr_3H^4\cong\frac{H^1(W^1_4)}{h^*H^1(Pic^4C)} \cong H^1 (Q) \cong\bQ^{12}$.
\item There is an exact sequence
$$\xymatrix{0\ar[r]&H^2(M_{12})\ar[r]^-{(-j_{1*},j_{2*})}&Gr_4H^4\ar[r]^-{i_t^*}&Gr_4H^4_t\ar[r]&0}.$$
Consequently,
$Gr_4H^4_t\cong \bQ^{264}$ and
$H^4(\T_t)\cong\bQ^{288}.$
 \end{enumerate}

\end{proposition}

\begin{proof}If $k\le3$, by (\ref{sl2}) and (\ref{Grsmallk}), 
$$Gr_kH^4_t\cong Gr_kK^4_t\oplus...\oplus Gr_{k-\lfloor\frac{k}{2}\rfloor}K^4_t\cong Gr_kH^4\oplus...\oplus Gr_{k-\lfloor\frac{k}{2}\rfloor}H^4.$$
Therefore, the statements about $Gr_kH^4_t$ for $k\le3$ follow immediately from the computation of the weight filtration on $H^4$ in Proposition \ref{propwH4}.

For $k=4$, 
$$Gr_4H^4_t\cong Gr_4K^4_t\oplus Gr_2K^4_t\oplus Gr_0K^4_t\cong Gr_4K^4_t.$$
The exact sequence (\ref{seqGR4}) becomes, 
\begin{eqnarray}\xymatrix{0\ar[r]&Gr_{2}K^{2}_t\ar[r]&Gr_{-6}H_{6}\ar[r]^-{\alpha}&Gr_4H^4\ar[r]&Gr_4K^4_t\ar[r]&0}.\nonumber
\end{eqnarray}

By Lemma \ref{lemCS} below, the image of $\alpha$ is equal to $(-j_{1*},j_{2*})H^2(M_{12})$ and $(-j_{1*},j_{2*})$ is clearly injective. Therefore (3) holds.
The statements for $k\ge 5$ follow by symmetry (see Section \ref{sym}).
\end{proof}

\begin{proposition}\label{propprim}The induced monodromy filtration on the primitive cohomology $\bK_t\subset H^4_t$ and $\bH_t=\bK_t\oplus\theta H^2(A_t)$ satisfies the following:
\begin{enumerate}
\item
$Gr_k\bK_t\cong Gr_kH^4_t$ for $k=3,5$.
\item We have the exact sequence 
$$\xymatrix{(I\oplus H^4(M_2))\cap Gr_4H^4\ar[r]^-{i_t^*}&Gr_4\bH_t\ar[r]&0,}$$
where $I:=p_1^*(\theta H^2(Pic^4C)\oplus\eta H^2(Pic^4C)\oplus\eta^2)\oplus\langle j_{1*}f, j_{1*}\tau_1\rangle\subset H^4(M_1)$.
\end{enumerate}
\end{proposition}
\begin{proof}Since the family of Prym varieties $A_t$ does not degenerate, we have $Gr_6H^6(A_t)\cong H^6(A_t)$ and $Gr_5H^6(A_t)=Gr_7H^6(A_t)=0$. Therefore, under the Gysin push-forward, $Gr_3H^4(\T_4)$ and $Gr_5H^4(\T_t)$ all map to zero. Thus the first statement of the proposition follows. Now consider the commutative diagram
$$\xymatrix{H^4(M_1\cup M_2)\ar[r]^-{\cong}\ar[d]^-{j_{0*}}&H^4(\widetilde{\T})\ar[r]^-{i_t^*}\ar[d]^-{j_*}&H^4(\T_t)\ar[d]^-{j_{t*}}\\
H^6(Pic^4C)\ar[r]^-{\cong}&H^6(\cA)\ar[r]^-{\cong}&H^6(A_t).}$$
Since the induced Gysin map on the graded piece
$$Gr_4H^4(M_1\cup M_2)\stackrel{j_{0*}}\lra Gr_6H^6(Pic^4C)\cong H^6(Pic^4C)$$ sends $(I\oplus H^4(M_2))\cap Gr_4H^4$ to the subspace $\theta^2H^2(Pic^4C)$ and the bottom horizontal maps are isomorphisms, we see that $i_t^*$ sends $(I\oplus H^4(M_2))\cap Gr_4H^4$ into $Gr_4\bH_t$. By Proposition \ref{H4theta} (3), the kernel of $i_t^*$ is 3 dimensional, therefore $i_t^*$ sends $(I\oplus H^4(M_2))\cap Gr_4H^4$ onto $Gr_4\bH_t$ by a simple dimension count.
\end{proof}

It remains to describe the $3$-dimensional image of $\alpha$ in (\ref{seqGR4}). 
Recall from the definition of the spectral sequence in Section \ref{subsecspseq} that $Gr_4H^4$ fits in the exact sequence
$$\xymatrix{0\ar[r]&Gr_4H^4\ar[r]&H^4(M_1)\oplus
  H^4(M_2)\ar[rr]^-{d_1^{0,4}=j_1^*-j_2^*}&&H^4(M_{12})\ar[r]&0.}$$
The composition of the natural map
$$\xymatrix{H^2(M_{12})\ar[r]^-{(-j_{1*},j_{2*})}&H^4(M_1)\oplus
  H^4(M_2)}$$
with $d_1^{0,4}$ is zero, therefore $(-j_{1*},j_{2*})$ factors through $Gr_4H^4$:
\[
\xymatrix{&&H^2(M_{12})\ar[ld]_{(-j_{1*},j_{2*})}\ar[d]^-{(-j_{1*},j_{2*})}&&\\0\ar[r]&Gr_4H^4\ar[r]&H^4(M_1)\oplus
  H^4(M_2)\ar[r]^-{d_1^{0,4}}&H^4(M_{12})\ar[r]&0.}
\]
\begin{lemma}\label{lemCS} The image of $\alpha: Gr_{-6}H_6\lra Gr_4H^4$ is equal to the image of 
$$\xymatrix{H^2(M_{12})\ar[r]^-{(-j_{1*},j_{2*})}&Gr_4H^4\subset H^4(M_1)\oplus H^4(M_2).}$$
\end{lemma}
\begin{proof} We have the isomorphism $Gr_{-6}H_6^{\lor}\cong Gr_6H^6$ and the latter fits into the exact sequence
$$\xymatrix{0\ar[r]&Gr_6H^6\ar[r]&H^6(M_1)\oplus H^6(M_2)\ar[r]^-{j_1^*-j_2^*}&H^6(M_{12})\ar[r]&0}$$
whose Poincar\'e dual is
$$\xymatrix{0\ar[r]&H^0(M_{12})\ar[r]^-{(j_{1*},-j_{2*})}&H^2(M_1)\oplus H^2(M_2)\ar[r]&(Gr_6H^6)^{\lor}\ar[r]&0.}$$

The map $\alpha$ is induced by 
$$\xymatrix{H_6(\widetilde{\T})\ar[r]^-{\text{PD}}&H^4(\widetilde{\T},\partial\tT)\ar[r]&H^4(\widetilde{\T})\cong H^4(M_1\cup M_2)}.$$
On the graded level, 
$$\alpha: Gr_{-6} H_6 = (Gr_5 H^6)^* = \frac{H^2(M_1)\oplus H^2(M_2)}{H^0(M_{12})}\lra Gr_4H^4$$
 is induced by the map
\begin{eqnarray}
H^2(M_1)\oplus H^2(M_2) & \lra & Gr_4H^4\subset H^4(M_1)\oplus H^4(M_2) \nonumber \\
(\gamma_1,\gamma_2) & \longmapsto & (-j_{1*}(j_1^*\gamma_1-j_2^*\gamma_2),\ j_{2*}(j_1^*\gamma_1-j_2^*\gamma_2)). \nonumber
\end{eqnarray}

Since the map
$$ {j_1^*-j_2^*} : \xymatrix{H^2(M_1)\oplus H^2(M_2)\ar[r]&H^2(M_{12})}$$
is surjective, the image of $\alpha$ is equal to the 3 dimensional image of $H^2(M_{12})$ via $(-j_{1*}, j_{2*})$.
\end{proof}



\section{The semistable reduction of the fiber product}\label{secssprod}

\subsection{} We need to construct a semistable reduction for the fiber product $\widetilde{\cG}\times_T\widetilde{\T}$. The central fibers of $\widetilde{\cG}$ and $\widetilde{\T}$ are described in Section \ref{rkcentralfiber} and Proposition \ref{proptheta0} respectively. We follow the notation there. The total space of $\widetilde{\cG}\times_T\widetilde{\T}$ is singular along $X_{kp}\times M_{12}$ and $X_{kq}\times M_{12}$ for $k=1,2$. The semi-stable reduction is simply the blow up  $\widetilde{\tilde{\cG}\times_{T}\tilde{\T}}$ of $\widetilde{\cG}\times_T\widetilde{\T}$ along the union of $W_1\times M_1$ and $W_2\times M_1$ and it sits in the commutative diagram with Cartesian squares

$$\xymatrix{\widetilde{\tilde{\cG}\times_T\tilde{\T}}\ar[rd]\ar^-{\rho_2}[rrrd]\ar_-{\rho_1}[rddd]&&\\
&\widetilde{\cG}\times_T\widetilde{\T}\ar[rr]\ar[d]&&\widetilde{\T}\ar[d]\\
&\widetilde{\cG}\times_T\T\ar^-{\epsilon_2}[rr]\ar[d]^-{\epsilon_1}&&\T\ar[d]&\\
&\widetilde{\cG}\ar[rr]&&T.}$$

\begin{proposition} The blow up $\widetilde{\tilde{\cG}\times_{T}\tilde{\T}}$ of  $\widetilde{\cG}\times_T\widetilde{\T}$ along the union of $W_1\times M_1$ and $W_2\times M_1$ is a semistable family whose central fiber has eight components:
\begin{enumerate}
\item For $k=1,2$, the total transform $\widetilde{W_k\times M_1}$ of $W_k\times M_1$  which is isomorphic to the blow-up of $W_k\times M_1$ along $X_{kp}\times M_{12}\cup X_{kq}\times M_{12}$.

\item The proper transforms $\widetilde{P_1\times M_2}$ and $\widetilde{P_2\times M_2}$ of $P_1 \times M_2$ and $P_2 \times M_2$ respectively, which are isomorphic to the blow-ups of $P_1\times M_2$ and $P_2\times M_2$ along $X_{1p}\times M_{12}\cup X_{2q}\times M_{12}$ and $X_{1q}\times M_{12}\cup X_{2p}\times M_{12}$ respectively.

\item The proper transforms of $P_1\times M_1$, $P_2\times M_1$, $W_1\times M_2$ and $W_2\times M_2$ which are unchanged under the blow-up.

\end{enumerate}

\end{proposition}
\begin{proof} We check locally that this is indeed a semistable reduction. Locally, the total space of the fiber product near, say, $X_{1p}\times M_{12}$, is isomorphic to the product of an affine space and
\begin{eqnarray}\label{flop}
Spec\ \frac{\bC[x,y,z,w,t]}{( xy-t, zw-t )}\cong Spec\ \frac{\bC[x,y,z,w]}{ xy-zw
 }.
 \end{eqnarray}
In the above local coordinates, $X_{1p}\times M_{12}$ is defined by the ideal $(x,y,z,w)$ and blowing up $\widetilde{\cG}\times_{T}\widetilde{\T}$ along $W_1\times M_1$ amounts to blowing up (\ref{flop}) along the ideal $(x ,z)$. Let $X$, $Z$ be the corresponding homogeneous coordinates in the blow-up. By symmetry, it is sufficient to check the result on the chart $\{X\ne0\}$. Here $\widetilde{\tilde{\cG}\times_{T}\tilde{\T}}$ is isomorphic to the product of an affine space and
\[
Spec\ \frac{\bC[x,y,Z,w]}{y-Zw }\cong Spec\ \bC[x,Z,w]
\]
 which is smooth. The central fiber in this chart is given by
 $t=xy=xZw$ which is a simple normal crossing divisor.
 
 The other assertions about the components of the central fiber are immediate.
 \end{proof}

\subsection{}\label{subseccomp} The eight components of the central fiber meet as follows
 $$
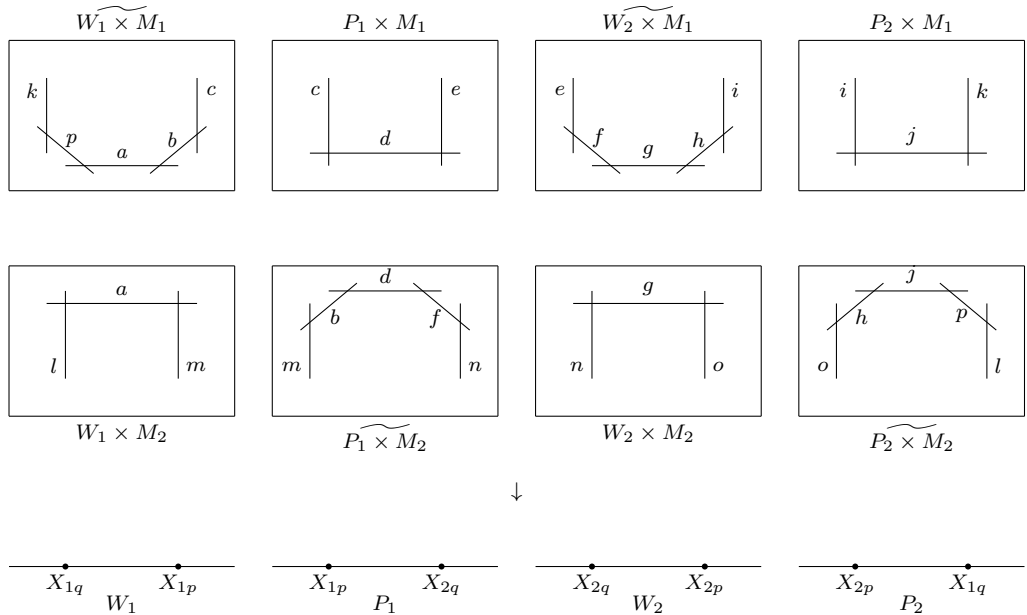
\begin{tikzpicture}[xscale=1,yscale=1][font=\tiny]
   \coordinate (a) at (0,0);
   \coordinate (b) at (3,0);
   \coordinate (c) at (3,2);
   \coordinate (d) at (0,2);
   \coordinate (e) at ($(a)!1/6!(b)$);
   \coordinate (f) at ($(b)!1/6!(a)$);
   \coordinate (g) at ($(b)!1/6!(c)$);
   \coordinate (h) at ($(c)!1/6!(b)$);
   \coordinate (i) at ($(c)!1/6!(d)$);
   \coordinate (j) at ($(d)!1/6!(c)$);
   \coordinate (k) at ($(d)!1/6!(a)$);
   \coordinate (l) at ($(a)!1/6!(d)$);
   \coordinate (m) at ($(e)!1/2-1/8!(j)$);
   \coordinate (n) at ($(l)!1/3!(g)$);
   \coordinate (o) at ($(g)!1/3!(l)$);
   \coordinate (p) at ($(f)!1/2-1/8!(i)$);
   \draw (a)--(b);
   \draw (b)--(c);
   \draw (c)--(d);
   \draw (d)--(a);
   \draw ($(j)!1/4!(e)$)--($(e)!1/4!(j)$);
   \draw ($(l)!1/4!(g)$)--($(g)!1/4!(l)$);
   \draw ($(f)!1/4!(i)$)--($(i)!1/4!(f)$);
   \draw ($(n)!1+1/4!(m)$)--($(m)!1+1/4!(n)$);
   \draw ($(p)!1+1/4!(o)$)--($(o)!1+1/4!(p)$);
   \node [above] at ($(c)!1/2!(d)$){$\widetilde{W_{1}\times M_{1}}$};
   \node [left] at ($(j)!1/3!(e)$){$\mathit{k}$};
   \node [right] at ($(i)!1/3!(f)$){$\mathit{c}$};
   \node [above] at ($(l)!1/2!(g)$){$\mathit{a}$};
   \node [above] at ($(m)!2/3!(n)$){$\mathit{p}$};
   \node [above] at ($(p)!2/3!(o)$){$\mathit{b}$};
   \coordinate (a) at (0+3.5,0);
   \coordinate (b) at (3+3.5,0);
   \coordinate (c) at (3+3.5,2);
   \coordinate (d) at (0+3.5,2);
   \coordinate (e) at ($(a)!1/4!(b)$);
   \coordinate (f) at ($(b)!1/4!(a)$);
   \coordinate (g) at ($(b)!1/4!(c)$);
   \coordinate (i) at ($(c)!1/4!(d)$);
   \coordinate (j) at ($(d)!1/4!(c)$);
   \coordinate (l) at ($(a)!1/4!(d)$);
   \draw (a)--(b);
   \draw (b)--(c);
   \draw (c)--(d);
   \draw (d)--(a);
   \draw ($(j)!1/4!(e)$)--($(e)!1/6!(j)$);
   \draw ($(i)!1/4!(f)$)--($(f)!1/6!(i)$);
   \draw ($(l)!1/6!(g)$)--($(g)!1/6!(l)$);
   \node [above] at ($(c)!1/2!(d)$){$P_{1}\times M_{1}$};
   \node [left] at ($(j)!1/3!(e)$){$\mathit{c}$};
   \node [right] at ($(i)!1/3!(f)$){$\mathit{e}$};
   \node [above] at ($(l)!1/2!(g)$){$\mathit{d}$};
   \coordinate (a) at (0+7,0);
   \coordinate (b) at (3+7,0);
   \coordinate (c) at (3+7,2);
   \coordinate (d) at (0+7,2);
   \coordinate (e) at ($(a)!1/6!(b)$);
   \coordinate (f) at ($(b)!1/6!(a)$);
   \coordinate (g) at ($(b)!1/6!(c)$);
   \coordinate (h) at ($(c)!1/6!(b)$);
   \coordinate (i) at ($(c)!1/6!(d)$);
   \coordinate (j) at ($(d)!1/6!(c)$);
   \coordinate (k) at ($(d)!1/6!(a)$);
   \coordinate (l) at ($(a)!1/6!(d)$);
   \coordinate (m) at ($(e)!1/2-1/8!(j)$);
   \coordinate (n) at ($(l)!1/3!(g)$);
   \coordinate (o) at ($(g)!1/3!(l)$);
   \coordinate (p) at ($(f)!1/2-1/8!(i)$);
   \draw (a)--(b);
   \draw (b)--(c);
   \draw (c)--(d);
   \draw (d)--(a);
   \draw ($(j)!1/4!(e)$)--($(e)!1/4!(j)$);
   \draw ($(l)!1/4!(g)$)--($(g)!1/4!(l)$);
   \draw ($(f)!1/4!(i)$)--($(i)!1/4!(f)$);
   \draw ($(n)!1+1/4!(m)$)--($(m)!1+1/4!(n)$);
   \draw ($(p)!1+1/4!(o)$)--($(o)!1+1/4!(p)$);
   \node [above] at ($(c)!1/2!(d)$){$\widetilde{W_{2}\times M_{1}}$};
   \node [left] at ($(j)!1/3!(e)$){$e$};
   \node [right] at ($(i)!1/3!(f)$){$i$};
   \node [above] at ($(l)!1/2!(g)$){$g$};
   \node [above] at ($(m)!2/3!(n)$){$f$};
   \node [above] at ($(p)!2/3!(o)$){$h$};
   \coordinate (a) at (0+10.5,0);
   \coordinate (b) at (3+10.5,0);
   \coordinate (c) at (3+10.5,2);
   \coordinate (d) at (0+10.5,2);
   \draw (a)--(b);
   \draw (b)--(c);
   \draw (c)--(d);
   \draw (d)--(a);
   \coordinate (e) at ($(a)!1/4!(b)$);
   \coordinate (f) at ($(b)!1/4!(a)$);
   \coordinate (g) at ($(b)!1/4!(c)$);
   \coordinate (i) at ($(c)!1/4!(d)$);
   \coordinate (j) at ($(d)!1/4!(c)$);
   \coordinate (l) at ($(a)!1/4!(d)$);
   \draw ($(j)!1/4!(e)$)--($(e)!1/6!(j)$);
   \draw ($(i)!1/4!(f)$)--($(f)!1/6!(i)$);
   \draw ($(l)!1/6!(g)$)--($(g)!1/6!(l)$);
   \node [above] at ($(c)!1/2!(d)$){$P_{2}\times M_{1}$};
   \node [left] at ($(j)!1/3!(e)$){$i$};
   \node [right] at ($(i)!1/3!(f)$){$k$};
   \node [above] at ($(l)!1/2!(g)$){$j$};
   \coordinate (d) at (0,0-3);
   \coordinate (c) at (3,0-3);
   \coordinate (b) at (3,2-3);
   \coordinate (a) at (0,2-3);
   \draw (a)--(b);
   \draw (b)--(c);
   \draw (c)--(d);
   \draw (d)--(a);
   \coordinate (e) at ($(a)!1/4!(b)$);
   \coordinate (f) at ($(b)!1/4!(a)$);
   \coordinate (g) at ($(b)!1/4!(c)$);
   \coordinate (i) at ($(c)!1/4!(d)$);
   \coordinate (j) at ($(d)!1/4!(c)$);
   \coordinate (l) at ($(a)!1/4!(d)$);
   \draw ($(j)!1/4!(e)$)--($(e)!1/6!(j)$);
   \draw ($(i)!1/4!(f)$)--($(f)!1/6!(i)$);
   \draw ($(l)!1/6!(g)$)--($(g)!1/6!(l)$);
   \node [below] at ($(c)!1/2!(d)$){$W_{1}\times M_{2}$};
   \node [left] at ($(j)!1/3!(e)$){$l$};
   \node [right] at ($(i)!1/3!(f)$){$m$};
   \node [above] at ($(l)!1/2!(g)$){$a$};
   \coordinate (d) at (0+3.5,0-3);
   \coordinate (c) at (3+3.5,0-3);
   \coordinate (b) at (3+3.5,2-3);
   \coordinate (a) at (0+3.5,2-3);
   \coordinate (e) at ($(a)!1/6!(b)$);
   \coordinate (f) at ($(b)!1/6!(a)$);
   \coordinate (g) at ($(b)!1/6!(c)$);
   \coordinate (h) at ($(c)!1/6!(b)$);
   \coordinate (i) at ($(c)!1/6!(d)$);
   \coordinate (j) at ($(d)!1/6!(c)$);
   \coordinate (k) at ($(d)!1/6!(a)$);
   \coordinate (l) at ($(a)!1/6!(d)$);
   \coordinate (m) at ($(e)!1/2-1/8!(j)$);
   \coordinate (n) at ($(l)!1/3!(g)$);
   \coordinate (o) at ($(g)!1/3!(l)$);
   \coordinate (p) at ($(f)!1/2-1/8!(i)$);
   \draw (a)--(b);
   \draw (b)--(c);
   \draw (c)--(d);
   \draw (d)--(a);
   \draw ($(j)!1/4!(e)$)--($(e)!1/4!(j)$);
   \draw ($(l)!1/4!(g)$)--($(g)!1/4!(l)$);
   \draw ($(f)!1/4!(i)$)--($(i)!1/4!(f)$);
   \draw ($(n)!1+1/4!(m)$)--($(m)!1+1/4!(n)$);
   \draw ($(p)!1+1/4!(o)$)--($(o)!1+1/4!(p)$);
   \node [below] at ($(c)!1/2!(d)$){$\widetilde{P_{1}\times M_{2}}$};
   \node [left] at ($(j)!1/3!(e)$){$m$};
   \node [right] at ($(i)!1/3!(f)$){$n$};
   \node [above] at ($(l)!1/2!(g)$){$d$};
   \node [below] at ($(m)!2/3!(n)$){$b$};
   \node [below] at ($(p)!2/3!(o)$){$f$};
   \coordinate (d) at (0+7,0-3);
   \coordinate (c) at (3+7,0-3);
   \coordinate (b) at (3+7,2-3);
   \coordinate (a) at (0+7,2-3);
   \draw (a)--(b);
   \draw (b)--(c);
   \draw (c)--(d);
   \draw (d)--(a);
   \coordinate (e) at ($(a)!1/4!(b)$);
   \coordinate (f) at ($(b)!1/4!(a)$);
   \coordinate (g) at ($(b)!1/4!(c)$);
   \coordinate (i) at ($(c)!1/4!(d)$);
   \coordinate (j) at ($(d)!1/4!(c)$);
   \coordinate (l) at ($(a)!1/4!(d)$);
   \draw ($(j)!1/4!(e)$)--($(e)!1/6!(j)$);
   \draw ($(i)!1/4!(f)$)--($(f)!1/6!(i)$);
   \draw ($(l)!1/6!(g)$)--($(g)!1/6!(l)$);
   \node [below] at ($(c)!1/2!(d)$){$W_{2}\times M_{2}$};
   \node [left] at ($(j)!1/3!(e)$){$n$};
   \node [right] at ($(i)!1/3!(f)$){$o$};
   \node [above] at ($(l)!1/2!(g)$){$g$};
   \coordinate (d) at (0+10.5,0-3);
   \coordinate (c) at (3+10.5,0-3);
   \coordinate (b) at (3+10.5,2-3);
   \coordinate (a) at (0+10.5,2-3);
   \coordinate (e) at ($(a)!1/6!(b)$);
   \coordinate (f) at ($(b)!1/6!(a)$);
   \coordinate (g) at ($(b)!1/6!(c)$);
   \coordinate (h) at ($(c)!1/6!(b)$);
   \coordinate (i) at ($(c)!1/6!(d)$);
   \coordinate (j) at ($(d)!1/6!(c)$);
   \coordinate (k) at ($(d)!1/6!(a)$);
   \coordinate (l) at ($(a)!1/6!(d)$);
   \coordinate (m) at ($(e)!1/2-1/8!(j)$);
   \coordinate (n) at ($(l)!1/3!(g)$);
   \coordinate (o) at ($(g)!1/3!(l)$);
   \coordinate (p) at ($(f)!1/2-1/8!(i)$);
   \draw (a)--(b);
   \draw (b)--(c);
   \draw (c)--(d);
   \draw (d)--(a);
   \draw ($(j)!1/4!(e)$)--($(e)!1/4!(j)$);
   \draw ($(l)!1/4!(g)$)--($(g)!1/4!(l)$);
   \draw ($(f)!1/4!(i)$)--($(i)!1/4!(f)$);
   \draw ($(n)!1+1/4!(m)$)--($(m)!1+1/4!(n)$);
   \draw ($(p)!1+1/4!(o)$)--($(o)!1+1/4!(p)$);
   \node [below] at ($(c)!1/2!(d)$){$\widetilde{P_{2}\times M_{2}}$};
   \node [left] at ($(j)!1/3!(e)$){$o$};
   \node [right] at ($(i)!1/3!(f)$){$l$};
   \node [above] at ($(l)!1/2!(g)$){$j$};
   \node [below] at ($(m)!2/3!(n)$){$h$};
   \node [below] at ($(p)!2/3!(o)$){$p$};
   \node at (6.5/2+7/2,-3-1){$\downarrow$};
   \coordinate (a) at (0,0-5);
   \coordinate (b) at (3,0-5);
   \coordinate (c) at ($(a)!1/2!(b)$);
   \draw (a)--(b);
   \path[fill=black]($(a)!1/4!(b)$)circle[radius=0.04];
   \path[fill=black]($(b)!1/4!(a)$)circle[radius=0.04];
   \node[below] at ($(a)!1/4!(b)$){$X_{1q}$};
   \node[below] at ($(b)!1/4!(a)$){$X_{1p}$};
   \node at ($(c)+(0,-1/2)$) {$W_1$};
   \coordinate (a) at (0+3.5,0-5);
   \coordinate (b) at (3+3.5,0-5);
   \coordinate (c) at ($(a)!1/2!(b)$);
   \draw (a)--(b);
   \path[fill=black]($(a)!1/4!(b)$)circle[radius=0.04];
   \path[fill=black]($(b)!1/4!(a)$)circle[radius=0.04];
   \node[below] at ($(a)!1/4!(b)$){$X_{1p}$};
   \node[below] at ($(b)!1/4!(a)$){$X_{2q}$};
   \node at ($(c)+(0,-1/2)$) {$P_1$};
   \coordinate (a) at (0+7,0-5);
   \coordinate (b) at (3+7,0-5);
   \coordinate (c) at ($(a)!1/2!(b)$);
   \draw (a)--(b);
   \path[fill=black]($(a)!1/4!(b)$)circle[radius=0.04];
   \path[fill=black]($(b)!1/4!(a)$)circle[radius=0.04];
   \node[below] at ($(a)!1/4!(b)$){$X_{2q}$};
   \node[below] at ($(b)!1/4!(a)$){$X_{2p}$};
   \node at ($(c)+(0,-1/2)$) {$W_2$};
   \coordinate (a) at (0+10.5,0-5);
   \coordinate (b) at (3+10.5,0-5);
   \coordinate (c) at ($(a)!1/2!(b)$);
   \draw (a)--(b);
   \path[fill=black]($(a)!1/4!(b)$)circle[radius=0.04];
   \path[fill=black]($(b)!1/4!(a)$)circle[radius=0.04];
   \node[below] at ($(a)!1/4!(b)$){$X_{2p}$};
   \node[below] at ($(b)!1/4!(a)$){$X_{1q}$};
   \node at ($(c)+(0,-1/2)$) {$P_2$};
\end{tikzpicture}\\
$$
The lines with the same label indicate the subvarieties that are glued together to form the double loci of the central fiber. The horizontal lines represent the loci that project onto $M_{12}$ via $\rho_2$ and the vertical lines the loci that project onto either $X_{kp}$ or $X_{kq}$ by $\rho_1$. The slanted lines represent exceptional loci: these are $\bP^1$-bundles over the products $X_{kp} \times M_{12}$ and $X_{kp} \times M_{12}$, hence are contracted by $\rho_1$ and $\rho_2$.
The dual graph of the central fiber is
$$
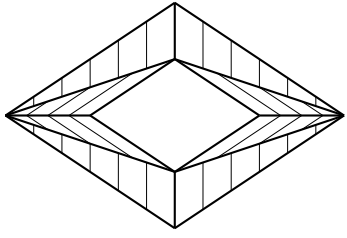
\begin{tikzpicture}[xscale=1.5,yscale=1.5][font=\tiny]
  \coordinate (a) at (0,0); 
  \coordinate (b) at (1.5,-1);
  \coordinate (c) at (3,0); 
  \coordinate (d) at (1.5,1);
  \coordinate (e) at ($(a)!1/4!(c)$);
  \coordinate (f) at ($(b)!1/4!(d)$);
  \coordinate (g) at ($(c)!1/4!(a)$);
  \coordinate (h) at ($(d)!1/4!(b)$);
  \draw [thick](a)--(b);
  \draw [thick](b)--(c);
  \draw [thick](c)--(d);
  \draw [thick](d)--(a);
  \draw [thick](a)--(h);
  \draw [thick](a)--(e);
  \draw [thick](a)--(f);
  \draw [thick](b)--(f);
  \draw [thick](c)--(h);
  \draw [thick](c)--(g);
  \draw [thick](c)--(f);
  \draw [thick](d)--(h);
  \draw [thick](e)--(f);
  \draw [thick](f)--(g);
  \draw [thick](g)--(h);
  \draw [thick](h)--(e);
  \draw [ultra thin]($(a)!1/4!(e)$)--($(a)!1/4!(f)$);
  \draw [ultra thin]($(a)!2/4!(e)$)--($(a)!2/4!(f)$);
  \draw [ultra thin]($(a)!3/4!(e)$)--($(a)!3/4!(f)$);
  \draw [ultra thin]($(a)!1/4!(e)$)--($(a)!1/4!(h)$);
  \draw [ultra thin]($(a)!2/4!(e)$)--($(a)!2/4!(h)$);
  \draw [ultra thin]($(a)!3/4!(e)$)--($(a)!3/4!(h)$);
  \draw [ultra thin]($(c)!1/4!(f)$)--($(c)!1/4!(g)$);
  \draw [ultra thin]($(c)!2/4!(f)$)--($(c)!2/4!(g)$);
  \draw [ultra thin]($(c)!3/4!(f)$)--($(c)!3/4!(g)$);
  \draw [ultra thin]($(c)!1/4!(h)$)--($(c)!1/4!(g)$);
  \draw [ultra thin]($(c)!2/4!(h)$)--($(c)!2/4!(g)$);
  \draw [ultra thin]($(c)!3/4!(h)$)--($(c)!3/4!(g)$);
  \draw [ultra thin]($(a)!1/6!(b)$)--($(a)!1/6!(f)$);
  \draw [ultra thin]($(a)!2/6!(b)$)--($(a)!2/6!(f)$);
  \draw [ultra thin]($(a)!3/6!(b)$)--($(a)!3/6!(f)$);
  \draw [ultra thin]($(a)!4/6!(b)$)--($(a)!4/6!(f)$);
  \draw [ultra thin]($(a)!5/6!(b)$)--($(a)!5/6!(f)$);
  \draw [ultra thin]($(c)!1/6!(b)$)--($(c)!1/6!(f)$);
  \draw [ultra thin]($(c)!2/6!(b)$)--($(c)!2/6!(f)$);
  \draw [ultra thin]($(c)!3/6!(b)$)--($(c)!3/6!(f)$);
  \draw [ultra thin]($(c)!4/6!(b)$)--($(c)!4/6!(f)$);
  \draw [ultra thin]($(c)!5/6!(b)$)--($(c)!5/6!(f)$);
  \draw [ultra thin]($(a)!1/6!(d)$)--($(a)!1/6!(h)$);
  \draw [ultra thin]($(a)!2/6!(d)$)--($(a)!2/6!(h)$);
  \draw [ultra thin]($(a)!3/6!(d)$)--($(a)!3/6!(h)$);
  \draw [ultra thin]($(a)!4/6!(d)$)--($(a)!4/6!(h)$);
  \draw [ultra thin]($(a)!5/6!(d)$)--($(a)!5/6!(h)$);
  \draw [ultra thin]($(c)!1/6!(d)$)--($(c)!1/6!(h)$);
  \draw [ultra thin]($(c)!2/6!(d)$)--($(c)!2/6!(h)$);
  \draw [ultra thin]($(c)!3/6!(d)$)--($(c)!3/6!(h)$);
  \draw [ultra thin]($(c)!4/6!(d)$)--($(c)!4/6!(h)$);
  \draw [ultra thin]($(c)!5/6!(d)$)--($(c)!5/6!(h)$);
\end{tikzpicture}
$$
The four vertices of the inside square correspond to the four components in the top row of the previous picture and the four vertices of the outside square to the bottom row. The shaded triangles correspond to triple intersections in the central fiber.

\subsection{}\label{subsecnotation}
 Let $ Pic^{(10)}(\cX/T)$ be the (noncompact) relative Picard scheme whose central fiber is $Pic^{6,4}(\tC_{pq})$.
There is a rational map 
$\psi:\widetilde{\cX}^{(5)}\times_T\widetilde{\cX}^{(5)}\dashrightarrow Pic^{(10)}(\cX/T)$ which is regular on the fibers over $t\ne0$. We will show in Proposition \ref{propFregular} that the rational map
$id\times\psi:\cG\times_T\widetilde{\cX}^{(5)}\times_T\widetilde{\cX}^{(5)}\dashrightarrow \cG\times_TPic^{(10)}(\cX/T)$ 
restricted to $\cF\subset\cG\times_T\widetilde{\cX}^{(5)}\times_T\widetilde{\cX}^{(5)}$ is regular. In other words, we have the following commutative diagram
 \[
\xymatrix{\cF \ar@{^{(}->}[rr]\ar[d]&&\cG\times_{T}\widetilde{\cX}^{(5)}\times_T\widetilde{\cX}^{(5)}\ar@{-->}^-{id\times\psi}[d]\\
\cG\times_T\T\ar@^{(->}[r]&\cG\times_T\cA\ar@^{(->}[r]&\cG\times_T Pic^{(10)}(\cX/T)
.}
\]

\begin{notation}\label{notationthreeF}
Denote $\cF', \cF_r'$ the images of $\cF, \cF_r$ in $\cG\times_T\T$, and $\cF'', \cF_r''$ and $\cF''', \cF_r'''$ the proper transforms of $\cF'$ and $\cF_r'$ in $\widetilde{\cG}\times_T\T, \widetilde{\tilde{\cG}\times_T\tilde{\T}}$ respectively. We summarize the relations between the various spaces in the diagram below:
$$\xymatrix{&{\cF}_r''' \subset \cF''' \ar@^{(->}[r]\ar[dd]&\widetilde{\tilde{\cG}\times_T\tilde{\T}}\ar[d]&\\&&\widetilde\cG\times_T\widetilde{\T}\ar[d]&\\
&{\cF}_r'' \subset \cF'' \ar@^{(->}[r]\ar[d]&\widetilde{\cG}\times_T\T\ar[d]\ar@^{(->}[r]&\widetilde{\cG}\times_T\cA\ar[d]\\
{\cF}_r\subset\cF\ar@^{-->}[ruuu]\ar[r]&{\cF}_r'\subset \cF'\ar@^{(->}[r]&\cG\times_T\T\ar@^{(->}[r]&\cG\times_T\cA.}$$
\end{notation}

\clearpage
\section{Abel-Jacobi maps on the generic and special fibers: outline of the proof of Theorem \ref{thmrho2rho1}}
\label{secAJgensp}

The Abel-Jacobi map $AJ$ on the total space is the composition
\[
\xymatrix{H^2(\widetilde \cG)\ar[r]^-{\rho_1^*}&H^2(\widetilde{\tilde{\cG}\times_T\tilde{\T}})\ar[r]^-{\cup[{\cF}'''_r]}&H^8(\widetilde{\tilde{\cG}\times_T\tilde{\T}})\ar[r]^-{\rho_{2*}}&H^4(\widetilde{\T}),}
\]
where the Gysin map $\rho_{2*}$ is defined as
\[
\xymatrix{H^8(\widetilde{\tilde{\cG}\times_T\tilde{\T}})\stackrel{PD}{\cong} H^6_c(\widetilde{\tilde{\cG}\times_T\tilde{\T}})^{\lor}\ar[r]^-{(\rho_2^*)^{\lor}}&H^6_c(\widetilde{\T})^{\lor}\stackrel{PD}{\cong} H^4(\widetilde{\T}),}
\]
where $PD$ denotes Poincar\'e duality.
As explained in Section \ref{MHScompact}, there exist canonical mixed Hodge structures on  $H^6_c(\widetilde{\tilde{\cG}\times_T\tilde{\T}})$ and $H^6_c(\widetilde{\T})$, such that $\rho_2^*$ (and therefore $(\rho_2^*)^{\lor}$) is a morphism of mixed Hodge structures. Thus the Abel-Jacobi map $AJ$, as a composition of such, is also a morphism of mixed Hodge structures.

By functoriality of the morphisms involved, we have a commutative diagram 
$$\xymatrix{H^2(\widetilde{\cG})\ar[r]\ar[d]^-{AJ}&H^2(G_t)\ar[d]^{AJ_t}\\
H^4(\widetilde \T)\ar[r]&H^4(\T_t),}$$
where the images of the horizontal maps are the monodromy invariant parts of the cohomology groups of $G_t$ and $\T_t$.  

\subsection{The map $AJ$ on the $E_1$ terms}\label{subsecAJE1} The maps $\rho_1^*$, $\cup[{\cF}'''_r]$ and $\rho_2^*$ are defined on the $E_1$ terms of the spectral sequences in Section \ref{secgeneralMH} and commute with the differentials $d_1$.

For $k=0,1$, the map $\rho_1^*$ on the $E_1$ terms is
\[
\xymatrix{_{\widetilde{\cG}}E_1^{k,2-k}\ar[r]^-{\rho_1^*}\ar@{=}[d]&_{\widetilde{\tilde{\cG}\times_T\tilde{\T}}}E_1^{k,2-k}\ar@{=}[d]\\
H^{2-k}(\tG_0^{[k]})\ar[r]^-{\rho_1^*}&H^{2-k}((\widetilde{\tilde{\cG}\times_T\tilde{\T}})_0^{[k]}).}
\]

If, for a stratum $S$ in $(\widetilde{\tilde{\cG}\times_T\tilde{\T}})_0^{[k]}$, $\rho_1(S)$ is not contained in  $\tG_0^{[k]}$, then the projection of $\rho_1^*$ onto the summand $H^{2-k}(S)\subset H^{2-k}((\widetilde{\tilde{\cG}\times_T\tilde{\T}})_0^{[k]})$ is zero (some components of $(\widetilde{\tilde{\cG}\times_T\tilde{\T}})_0^{[1]}$ map onto components in $\tG_0^{[0]}$, c.f. Section \ref{secssprod}).

Cup-product with $[{\cF}_r''']$ induces the horizontal maps
\[
\xymatrix{_{\widetilde{\tilde{\cG}\times_T\tilde{\T}}}E_1^{k,2-k}\ar[r]^-{\cup [{\cF}_r''']}\ar@{=}[d]&_{\widetilde{\tilde{\cG}\times_T\tilde{\T}}}E_1^{k,8-k}\ar@{=}[d]\\
H^{2-k}((\widetilde{\tilde{\cG}\times_T\tilde{\T}})_0^{[k]})\ar[r]^-{\cup [{\cF}_r''']}&H^{8-k}((\widetilde{\tilde{\cG}\times_T\tilde{\T}})_0^{[k]}),}
\]
where the lower horizontal map is cup-product with the cycle class of the {\em scheme theoretic} intersection of ${\cF}'''_r$ with each component in $(\widetilde{\tilde{\cG}\times_T\tilde{\T}})_0^{[k]}$.

The map $\rho_2^*$ on cohomology with compact supports is
\[
\xymatrix{_{\widetilde{\T}}E_{1,c}^{-k,k+6}\ar[r]^-{\rho_2^*}\ar@{=}[d]&_{\widetilde{\tilde{\cG}\times_T\tilde{\T}}}E_{1,c}^{-k,k+6}\ar@{=}[d]\\
H^{4-k}(\widetilde \T_0^{[k]})\ar[r]^-{\rho_2^*}&H^{4-k}((\widetilde{\tilde{\cG}\times_T\tilde{\T}})_0^{[k]}).}
\]
Similarly to the case of $\rho_1^*$ above, we only pull back to the strata of $(\widetilde{\tilde{\cG}\times_T\tilde{\T}})_0^{[k]}$ which map to $\widetilde \T_0^{[k]}$. Thus the dual map $\rho_{2*} = (\rho_2^*)^{\lor}$ is
induced by the usual Gysin maps between the relevent strata:
\[
\xymatrix{_{(\widetilde{\tilde{\cG}\times_T\tilde{\T}})}E_{1}^{k,8-k}\ar[r]^-{\rho_{2*}}\ar@{=}[d]&_{\widetilde{\T}}E_1^{k,4-k}\ar@{=}[d]\\
H^{8-k}((\widetilde{\tilde{\cG}\times_T\tilde{\T}})_0^{[k]})\ar[r]^-{\rho_{2*}}&H^{4-k}(\widetilde\T_0^{[k]}).}
\]
                                                                                       
 To compute (\ref{Gr4}) and ({\ref{Gr3}), we first compute the Able-Jacobi map $AJ^k$ on the $E_1$ terms for $k=0,1$:  
$$AJ^k : \xymatrix{H^{2-k}(\tG_0^{[k]})\ar[r]^-{\rho_1^*}&H^{2-k}((\widetilde{\tilde{\cG}\times_T\tilde{\T}})_0^{[k]})\ar[r]^-{\cup[{\cF}'''_r]}&H^{8-k}((\widetilde{\tilde{\cG}\times_T\tilde{\T}})_0^{[k]})\ar[r]^-{\rho_{2*}}&H^{4-k}(\widetilde\T_0^{[k]}),}$$
then pass to the $E_2$ terms of the corresponding spectral sequences.

\subsection{Proof of the main theorem}
Notation as in Section \ref{subsecstrata}. We divide the proof of Theorem \ref{thmrho2rho1} into four propositions. 

For the Abel-Jacobi map on the $E_1$ terms, we write $AJ^0=(AJ^0_1,AJ^0_2): H^2(\tG_0^{[0]})\ra H^4(\widetilde\T_0^{[0]})=H^4(M_1)\oplus H^4(M_2)$. We have

\begin{proposition}\label{AJ21} The image of the map $AJ^0_1: H^2(\tG_0^{[0]})\ra H^4(M_1)$ contains the subspace $I:=p_1^*(\theta H^2 (Pic^4 C) \oplus \eta H^2(Pic^4C)\oplus \eta^2)\oplus\langle j_{1*}f,j_{1*}\tau_1\rangle$ modulo $\langle j_{1*}f,j_{1*}\tau_1\rangle\oplus p_1^*(\theta H^2(Pic^4C))$.
\end{proposition}

\begin{proposition}\label{AJ22} The map $AJ^0_2:H^2(\tG_0^{[0]})\ra H^4(M_2)$ is surjective modulo $\langle j_{2*}f,j_{2*}\tau_1\rangle$.\end{proposition}

For $AJ^1$, we have
\begin{proposition}\label{AJ1E1} The image of  $AJ^1:H^1(\tG_0^{[1]})\ra H^3(\widetilde\T_0^{[1]})=H^3(M_{12})$ contains $\tau_1\cdot\pi_{12}^*H^1(W^1_4)$. 
\end{proposition}

Next we pass to the Abel-Jacobi map on the $E_2$ terms.
\begin{proposition}\label{AJ2global}

The image of the restriction of $AJ^0=(AJ^0_1,AJ^0_2)$ to
$$Gr_2H^2(\tG)=Ker(H^2(\tG_0^{[0]})\stackrel{d_1}\ra H^2(\tG_0^{[1]}))$$ contains $(I\oplus H^4(M_2))\cap Gr_4H^4(\widetilde{\T})$ modulo $(-j_{1*},j_{2*})H^2(M_{12})+(p_1^*(\theta H^2(Pic^4C)),0)$.

\end{proposition}

Assuming the above four propositions, we can prove our main theorem.

\begin{proof} of {\bf Theorem \ref{thmrho2rho1}}. Identifying $H^4 (A_t)$ with a subspace of $H^4(\T_t)$ via pull-back, we have $H^4(\T_t) = (\bK_t\otimes\bQ)\oplus H^4(A_t)$, and, since $A_t$ does not degenerate,
 $$Gr_3H^4(\T_t) = Gr_3(\bK_t\otimes\bQ),$$
 and
 $$Gr_4H^4(\T_t) = Gr_4(\bK_t\otimes\bQ)\oplus H^4(A_t).$$  

Consider the commutative diagram
$$\xymatrix{H^2(\widetilde{\cG})\ar[r]^{i_t^*}\ar[d]^-{AJ}&H^2(G_t)\ar[d]^-{AJ_t}\\
H^4(\widetilde{\T})\ar[r]^{i_t^*}&H^4(\T_t).}$$
Proposition \ref{AJ1E1} implies that the image of $AJ_t$ sends $Gr_1H^2(G_t) = i_t^* Gr_1 H^2(\tcG)$ surjectively to $Gr_3H^4(\T_t)=i_t^*Gr_3H^4(\widetilde{\T})\cong \frac{H^1(W^1_4)}{H^1(Pic^4C)} = H^1 (Q)$.
Since the logarithm of the monodromy operator $N$ induces an isomorphism from $Gr_5H^4(\T_t)$ to $Gr_3H^4(\T_t)$ and from $Gr_3H^2(G_t)$ to $Gr_1H^2(G_t)$, we conclude that $AJ_t$ sends $Gr_3H^2(G_t)$ surjectively to $Gr_5H^4(\T_t)$. 

Next, by Lemma \ref{lemCS}, the ambiguity $(-j_{1*},j_{2*})H^2(M_{12})$ restricts to zero under $i_t^*$. Therefore, by Propositions \ref{propprim} and \ref{AJ2global}, the image of $Gr_2H^2(G_t)$ by $AJ_t$ contains $Gr_4(\bH_t\otimes\bQ)$ modulo $\theta_tH^2(A_t)$.

Combining the above, we see that the image of $AJ_t$ contains $\bH_t\otimes\bQ$ modulo $\theta_tH^2(A_t)$. Since, as we observed earlier, $\theta_tH^2(A_t)$ is always contained in the image of $H^2 (\T \cap \T_a)$ for $a\in A$ general, the theorem follows.
\end{proof}

\section{The cycles at time zero: before resolving the family of theta divisors} 

\subsection{The central fiber $F_0$}
We list the intersections of the central fiber $F_0$ of $\cF$ with each component $W_k \times (C_1^{(d_1)}\times C_2^{(d_2)})\times (C_1^{(e_1)} \times C_2^{(e_2)})$ in the tables below.  The left column lists the ambient spaces of all possible bidegrees. The middle column gives the conditions defining the cycles $F_{0}$ in each ambient space. 

For each pair of bidegrees $(d_1,d_2)$ and $(e_1,e_2)$, we define a morphism
\begin{eqnarray}
\psi_{(d_1,d_2)(e_1,e_2)} : F_0\cap \left(W_k \times (C_1^{(d_1)}\times C_2^{(d_2)})\times (C_1^{(e_1)} \times C_2^{(e_2)})\right) & \lra & \T_0\subset Pic^4C\nonumber \\
\nonumber (L,D_{d_1},D_{d_2},D'_{e_1},D'_{e_2}) & \longmapsto & \cO_{C}(D_{d_2}+D'_{e_2}-m(p+q)),
\end{eqnarray}
where $m$ is the integer such that $d_2+e_2=4+2m$. These morphisms are listed case by case in the right most column of the tables below.

\begin{table}[h]

\centering
\begin{tabular}{|c|c|c|c|c|}\hline
 Ambient Spaces& $F_{0}$ & Image under $\psi$\\\hline
$(L,D_4,a,D_4',a')
\in W_1\times (C_1^{(4)}\times C_2)\times (C_1^{(4)}\times C_2)$&$
\begin{cases}D_4+a\in|L| \\D_4'+a'\in|L'|\end{cases}$&$\cO_C(a+a'+p+q)$\\\hline

$(L,D_4,a,D_2',D_3')
\in W_1\times (C_1^{(4)}\times C_2)\times (C_1^{(2)}\times C_2^{(3)})$&$\begin{cases}D_4+a\in|L|\\ D_2'+D_3'\in|L'|\end{cases}$&$\cO_C(a+D'_3)$ \\ \hline

 $(L,D_4,a,D_5')\in W_1\times (C_1^{(4)}\times C_2)\times C_2^{(5)}$&$\begin{cases}D_4+a\in|L|\\D_5\in|L'|\end{cases}$&$K_C(-D_4)$\\\hline
 
$(L,D_2,D_3,D_4',a')\in W_1\times (C_1^{(2)}\times C_2^{(3)})\times (C_1^{(4)}\times C_2)$&$\begin{cases} D_2+D_3\in|L|\\ D_3'+a'\in|L'|\end{cases}$&$\cO_C(D_3+a')$\\ \hline

$(L,D_2,D_3,D_2',D_3')\in W_1\times (C_1^{(2)}\times C_2^{(3)})\times (C_1^{(2)}\times C_2^{(3)})$&$\begin{cases}D_2+D_3\in|L|\\  D_2'+D_3'\in|L'|\end{cases}$&$K_C(-D_2-D'_2)$\\
\hline

$(L,D_2,D_3,D_5)\in W_1\times (C_1^{(2)}\times C_2^{(3)})\times C_2^{(5)}$&$\begin{cases}D_2+D_3\in|L|\\ D_5\in|L'|\end{cases}$&$K_C(-D_2-p-q)$\\\hline

$(L,D_5,D_4',a')\in W_1\times C_2^{(5)}\times (C_1^{(4)}\times C_2)$&$\begin{cases}D_5\in|L|\\ D_4'+a'\in|L'|\end{cases}$&$K_C(-D_4')$\\\hline

$(L,D_5,D_2',D_3')\in W_1\times C_2^{(5)}\times (C_1^{(2)}\times C_2^{(3)})$&$\begin{cases}D_5\in|L|\\ D_2'+D_3'\in|L'|\end{cases}$&$K_C(-D_2'-p-q)$\\\hline

$(L,D_5,D_5')\in W_1\times C_2^{(5)}\times C_2^{(5)}$&$\begin{cases}D_5\in|L|\\ D_5'\in|L'|\end{cases}$&$K_C(-2p-2q)$\\\hline
\end{tabular}
\vspace{1cm}
\centering
\caption{Cycles in $W_1\times C_{pq}^{(5)}\times C_{pq}^{(5)}$}
\label{tableW1}
\end{table}

\clearpage
\begin{table}

\begin{tabular}{|c|c|c|c|c|}\hline
 Ambient Spaces& $F_{0}$ &Image under $\psi$\\\hline
$(L,D_5,D_5')\in W_2\times C_1^{(5)}\times C_1^{(5)}$&$\begin{cases}D_5\in|L|\\D_5'\in|L'|\end{cases}$&$\cO_C(2p+2q)$\\\hline

$(L,D_5,D_3',D_2')
\in W_2\times C_1^{(5)}\times (C_1^{(3)}\times C_2^{(2)})$&$\begin{cases}D_5\in|L|\\ D_3'+D_2'\in|L'|\end{cases}$&$\cO_C(D_2'+p+q)$\\\hline

$(L,D_5,a',D_4')
\in W_2\times C_1^{(5)}\times(C_1\times C_2^{(4)})$&$\begin{cases}D_5\in|L|\\ a'+D_4'\in|L'|\end{cases}$&$\cO_C(D_4')$\\\hline
 
$(L,D_3,D_2,D_5')\in W_2\times(C_1^{(3)}\times C_2^{(2)})\times C_1^{(5)}$&$\begin{cases} D_3+D_2\in|L|\\
 D_5'\in|L'|\\
\end{cases}$&$\cO_C(D_2+p+q)$\\\hline

$(L,D_3,D_2,D_3',D_2')\in 

W_2\times (C_1^{(3)}\times C_2^{(2)})\times(C_1^{(3)}\times C_2^{(2)})$&$\begin{cases} D_3+D_2\in|L|\\
 D_3'+D_2'\in|L'|
 \end{cases}$&$\cO_C(D_2+D'_2)$\\
\hline
$(L,D_3,D_2,a',D_4')\in W_2\times (C_1^{(3)}\times C_2^{(2)})\times(C_1\times C_2^{(4)})$&$\begin{cases}  D_3+D_2\in|L|\\
 a'+D_4'\in|L'|
 \end{cases}$&$K_C(-D_3-a')$\\ \hline
 
 $(L,a,D_4,D_5')\in W_2\times (C_1\times C_2^{(4)})\times C_1^{(5)}$&$\begin{cases}a+D_4\in|L|\\D_5'\in|L'|\end{cases}$&$\cO_C(D_4)$\\ \hline

 $(L,a, D_4,D_3',D_2')\in W_2\times (C_1\times C_2^{(4)})\times (C_1^{(3)}\times C_2^{(2)}$&$\begin{cases}a+D_4\in|L|\\D_3'+D_2'\in|L'|\end{cases}$&$K_C(-a-D_3')$\\ \hline
 
 $(L,a,D_4,a',D_4')\in W_2\times (C_1\times C_2^{(4)})\times (C_1\times C_2^{(4)})$& $\begin{cases}a+D_4\in|L|\\a'+D_4'\in|L'|\end{cases}$&$K_C(-a-a'-p-q)$\\\hline
 
   \end{tabular}
  \vspace{1cm} 
  \centering
   \caption{Cycles in $W_2\times C_{pq}^{(5)}\times C_{pq}^{(5)}$}
\label{tableW2}
\end{table}

\clearpage
 \subsection{The morphism to $\T_0$ }
\begin{proposition}\label{propFregular}The rational map $id\times\psi:\cG\times_T\widetilde{\cX}^{(5)}\times_T\widetilde{\cX}^{(5)}\dashrightarrow \cG\times_TPic^{(10)}(\cX/T)$ extends to a morphism when restricted to $\cF\subset\cG\times_T\widetilde{\cX}^{(5)}\times_T\widetilde{\cX}^{(5)}$ (see Section \ref{subsecnotation} for the notation). 
\end{proposition}
\begin{proof} We need to extend the rational map $\psi$ to the central fiber $F_0$ of $\cF$.  As explained in Section \ref{subsecPrymfamily}, the natural extension of the map $\psi$ to a general point of $W_k \times (C_1^{(d_1)}\times C_2^{(d_2)})\times (C_1^{(e_1)} \times C_2^{(e_2)})$ is given by $\psi_{(d_1,d_2)(e_1,e_2)}$.  Therefore we need to show that the morphisms $\psi_{(d_1,d_2)(e_1,e_2)}$ coincide on the intersection of $F_0$ with the overlaps of the different components of $G_0\times \tC_{pq}^{(5)}\times \tC_{pq}^{(5)}$. For instance, a point $(p+g^1_4,D_2,D_3=B_2+p,D_4'=B_3'+q,a')\in F_0\cap W_1\times(C_1^{(2)}\times C_2^{(3)})\times (C_1^{(4)}\times C_2)$ is identified with $(q+g^1_4,D_2+q,B_2,B_3',a'+p)\in F_0\cap W_2\times(C_1^{(3)}\times C_2^{(2)})\times (C_1^{(3)}\times C_2^{(2)})$. The images under $\psi_{(2,3)(4,1)}$ and $\psi_{(3,2)(3,2)}$ are both equal to $\cO_C(B_2+a'+p)$.
Therefore all the $\psi_{(d_1,d_2)(e_1,e_2)} |_{F_0}$ glue together and we obtain a morphism from $F_0$ to $\T_0$. 
\end{proof}

Recall that we have a tower of blow-ups and algebraic cycles in each blow-up.
$$\xymatrix{&{\cF}_r'''\ar@^{(->}[r]\ar[dd]&\widetilde{\tilde{\cG}\times_T\tilde{\T}}\ar[d]&\\&&\widetilde\cG\times_T\widetilde{\T}\ar[d]&\\
&{\cF}_r''\ar@^{(->}[r]\ar[d]&\widetilde{\cG}\times_T\T\ar[d]\ar@^{(->}[r]&\widetilde{\cG}\times_T\cA\ar[d]\\
{\cF}_r\ar[r]&{\cF}_r'\ar@^{(->}[r]&\cG\times_T\T\ar@^{(->}[r]&\cG\times_T\cA.}$$

Denote $F_{(d_1,d_2)(e_1,e_2)}$ the intersection of $F_{r_0}$ with $W_k \times (C_1^{(d_1)}\times C_2^{(d_2)})\times (C_1^{(e_1)} \times C_2^{(e_2)})$, and $\lambda :=(\lambda_1,\lambda_2):F_{r_0}\ra G_0\times \T_0$ the restriction of $id\times\psi$ to $F_{r_0}$. 
\subsection{The cycle $F_{r_0}''$}

Recall that $\widetilde{G}_0$ has four components $W_1, W_2, P_1, P_2$, where $P_k$ is a $\bP^1$ bundle over $X_{kp}$ for $k=1, 2$ (see Section \ref{rkcentralfiber}).  We use the notation $F''_{r_0}|_{W_k\times\T_0}$, $F''_{r_0}|_{P_k\times\T_0}$ to denote the components of $F''_{r_0}$ which lie in $W_k\times \T_0$, $P_k\times \T_0$ respectively. 

\begin{proposition}

 \label{propcycleP}\begin{enumerate}
\item The cycle $F''_{r_0}|_{W_1\times \T_0}$ is the push-forward under $\lambda$ of 
\[
F_{r_01}:=F_{(4, 1)(4,1)} \amalg F_{(4, 1)(2, 3)} \amalg F_{(2,3)(4,1)} \amalg F_{(2,3)(2,3)}.
\]
\item The cycle $F''_{r_0}|_{W_2\times \T_0}$ is the push-forward of 
\[
F_{r_02}:=F_{(5, 0)(3,2)} \amalg F_{(5, 0)(1, 4)} \amalg F_{(3,2)(3,2)} \amalg F_{(3,2)(1,4)} \amalg F_{(1, 4)(3,2)} \amalg F_{(1, 4)(1,4)}.
\]
\item The cycle $F''_{r_0}|_{P_k\times\T_0}$  is the image of the fiber product
$$\xymatrix{F_{r_0}|_{P_k}\ar[r]\ar[d]&{F}_{r_0}|_{X_{kp}}\ar[d]\\
P_k\ar[r]&X_{kp},}$$
where $F_{r_0}|_{P_k}$ maps to $P_k$ via projection and to $\T_0$ via $\lambda_2$.
\end{enumerate}
\end{proposition}

\begin{proof}Since any component of $F_0$ with bidegree $(0,5)+(e_1,e_2)$ does not intersect $F_{r_0}$, we see that $F_{(0,5)(e_1,e_2)}$ is empty. From Tables \ref{tableW1} and \ref{tableW2}, we see that $F_{(4,1)(0,5)}$, $F_{(2,3)(0,5)}$ and $F_{(d_1,d_2)(5,0)}$ are contracted by $\lambda$ and their image by $\lambda$ is contained in the closure of the image of cycles of other bidegrees. For other bidegrees, $\lambda$ is generically one to one on any irreducible component. This proves the first two statements. The third statement follows immediately from the construction of $\cF''_{r_0}$.
\end{proof}

\section{The cycles at time zero: after resolving the family of theta divisors}
\label{secF'''}

\subsection{}

The cycle ${\cF}_r'''$ is the proper transform of ${\cF}_r''$ under
$$\xymatrix{\widetilde{\tilde{\cG}\times_T\tilde{\T}}\ar[r]&{\widetilde{\cG}\times_T\widetilde{\T}}\ar[r]&\widetilde{\cG}\times_T\T}$$
where the arrow on the right is the blow up of $\widetilde{\cG}\times_T\T$ along $\tG_0\times W^1_4$ and the arrow on the left, which is a small resolution, is the blow up of $\widetilde{\cG}\times_T\widetilde{\T}$ along $\amalg_k (W_k\times
M_1)$. The central fiber of $\widetilde{\tilde{\cG}\times_T\tilde{\T}}$ has 8 components (see Section \ref{secssprod}), where $\widetilde{W_k\times M_1}$ and $P_k\times M_1$ are the main components and $W_k\times M_2$, $\widetilde{P_k\times M_2}$ are the exceptional components.

\begin{proposition}\label{propF'''}
$F_{r_0}'''|_{\widetilde{W_k\times M_1}}$ is the proper transform of $F_{r_0}''|_{W_k\times \T_0}$ under the birational morphism
$$\xymatrix{\widetilde{W_k\times M_1}\ar[r]&W_k\times M_1\ar^-{(id,p_1)}[r]&W_k\times C^{(4)}\ar^-{(id,\phi)}[r]&W_k\times\T_0.}$$
\end{proposition}
\begin{proof}The inverse of the birational morphism 
$$\xymatrix{\widetilde{W_k\times M_1}\ar[r]&W_k\times M_1\ar[r]&W_k\times\T_0}$$ is defined on the open subset $(W_k\setminus(X_p\cup X_q))\times(\T_0\setminus W^1_4)$. This open subset contains an open dense subset of $F_{r_0}''|_{W_k\times \T_0}$.
\end{proof}

\subsection{The center of the blow-up} \label{subsecZ''}
Next we study ${F}_{r_0}'''|_{W_k\times M_2}$, which is the (scheme-theoretic) intersection of ${\cF}_r'''$ with the exceptional divisor $W_k\times M_2$. So ${F}_{r_0}'''|_{W_k\times M_2}$ is the projectivized normal cone to ${F}_{r_0}''\cap(W_k\times W^1_4)\subset F''_{r_0}|_{W_k\times \T_0}$\footnote{We take the scheme theoretic intersection.} in ${\cF}_r''$. We first study the center of the blow-up.

By Proposition \ref{propcycleP}, $F''_{r_0}|_{W_k\times\T_0}$ is the image of 
\[
\xymatrix{F_{r_0k}\ar^-{(\lambda_1,\lambda_2)}[r]&F''_{r_0}|_{W_k\times\T_0}\subset W_k\times \T_0}.
 \]
Denote $Z_k\subset F_{r_0k}$ the inverse image scheme of ${F}_{r_0}''\cap(W_k\times W^1_4)\subset F''_{r_0}|_{W_k\times \T_0}$ and put $Z := Z_1\cup Z_2$ and $Z_{(d_1,d_2)(e_1,e_2)}=Z\cap F_{(d_1,d_2)(e_1,e_2)}$.  Then $Z_k$ maps onto $W^1_4\subset\T_0$ by $\lambda_2$ and we have the Cartesian diagram
   \begin{eqnarray}\label{diagZ''}\xymatrix{Z_{k}\ar[r]\ar@^{(->}[d]& W^1_4\ar@^{(->}[d]\\
  F_{r_0k}\ar[r]^-{\lambda_2}&\T_0.}
  \end{eqnarray}
\begin{proposition}\label{Znu1} $Z_1$ is one dimentional and for $s^1_4\ne g_i$ or $h_i$, the fiber $\lambda_2^{-1}(s^1_4)\cap F_{r_01}$ is finite. For $i = 1, \ldots 5$, the fiber $\lambda_2^{-1}(g_i)\cap F_{r_01}$ is one dimensional (modulo finitely many points) and its support is listed in the following table.

\begin{table}[h]
\centering
\begin{tabular}{|c|c|}\hline
 Ambient Spaces& Support of $\lambda_2^{-1}(g_i)\cap F_{r_01}$\\\hline
$(L,D_4,a,D_4',a')\in
W_1\times (C_1^{(4)}\times C_2)\times (C_1^{(4)}\times C_2)$&$
\begin{cases}a+a'+p+q\equiv g_i\\h^0(L-a-r_0)>0\\D_4\equiv L-a, D_4'\equiv L'-a'\end{cases}$\\\hline
$(L,D_4,a,D_2',D_3')
\in W_1\times (C_1^{(4)}\times C_2)\times (C_1^{(2)}\times C_2^{(3)})$&$\begin{cases}D_3'+a\equiv g_i\\L=a+g_j, j\ne i, a\in C\\r_0\le D_4\equiv g_j\\D_2'\equiv h_i-(g_i-p-q) \end{cases}$ \\ \hline
 $(L,D_4,a,D_2',D_3')
\in W_1\times (C_1^{(4)}\times C_2)\times (C_1^{(2)}\times C_2^{(3)})$&$\begin{cases}D_3'+a\equiv g_i\\L=h_i+p+q-c,c\in C\\h^0(L-r_0-a)>0, D_4\equiv L-a\\D_2'=a+c\end{cases}$\\\hline
$(L,D_2,D_3,D_4',a')\in W_1\times (C_1^{(2)}\times C_2^{(3)})\times (C_1^{(4)}\times C_2)$&$\begin{cases}D_3+r_0\equiv g_i \\L=c+g_i,c\in C\\a'=r_0, D_4'\equiv h_i+p+q-c-r_0\\
  D_2=r_0+c\end{cases}$\\ \hline
$(L,D_2,D_3,D_4',a')\in W_1\times (C_1^{(2)}\times C_2^{(3)})\times (C_1^{(4)}\times C_2)$&$\begin{cases} D_3+a'\equiv g_i\\L=r_0+g_i\\D_4'\equiv h_i+p+q-r_0-a', a'\in C\\D_2=a'+r_0\end{cases}$\\ \hline

\end{tabular}
\vspace{.5cm}
\centering
\end{table}

The fiber $\lambda_2^{-1}(h_i)\cap F_{r_01}$ is also one-dimensional with support described below.
\begin{table}[h]
\centering

\begin{tabular}{|c|c|}\hline
 Ambient Spaces& Support of $\lambda_2^{-1}(h_i)\cap F_{r_01}$\\\hline

$(L,D_4,a,D_2',D_3')
\in W_1\times (C_1^{(4)}\times C_2)\times (C_1^{(2)}\times C_2^{(3)})$&$\begin{cases}D_3'+a\equiv h_i\\L=a+g_i, a\in C\\r_0\le D_4\equiv g_i\\D_2'=p+q\end{cases}$ \\ \hline
\end{tabular}
\vspace{.5cm}
\centering
\end{table}

 \end{proposition}
 \begin{proof}We study $Z_1$ case by case according to the bidegree. The proof is divided into three Lemmas: \ref{lemZ4141}, \ref{lemZ4123} and \ref{lemZ2341}.
 \end{proof}
  
\begin{proposition}\label{Znu2}$Z_2$ is one dimensional and, for $s^1_4\ne g_i$ or $h_i$, the fiber $\lambda_2^{-1}(s^1_4)\cap F_{r_02}$ is finite. For $i = 1, \ldots 5$, the fiber $\lambda_2^{-1}(g_i)\cap F_{r_02}$ is one dimensional with support described below.

\begin{table}[h]
\centering

\begin{tabular}{|c|c|}\hline
 Ambient Spaces& Support of $\lambda_2^{-1}(g_i)\cap F_{r_02}$\\\hline

$(L,D_5,D_3',D_2')
\in W_2\times C_1^{(5)}\times (C_1^{(3)}\times C_2^{(2)})$&$\begin{cases}D_2'+p+q\equiv g_i\\h^0(L'-D_2')>0, D_3'\equiv L'-D_2'\\r_0\le D_5\equiv L\end{cases}$ \\ \hline
$(L,D_5,a',D_4')\in W_2 \times C_1^{(5)}\times(C_1\times C_2^{(4)})$&$\begin{cases}c\le D_4'\equiv g_i\\L=h_i+p+q-c,c\in C\\ a'=c\\r_0\le D_5\equiv L\end{cases}$ \\ \hline

\end{tabular}
\vspace{.5cm}
\centering
\end{table}

The fiber  $\lambda_2^{-1}(h_i)\cap F_{r_02}$ is one dimensional with support described below.

\begin{table}[h]
\centering

\begin{tabular}{|c|c|}\hline
 Ambient Spaces& Support of $\lambda_2^{-1}(h_i)\cap F_{r_02}$\\\hline

$(L,D_3,D_2,a',D_4',)
\in W_2\times (C_1^{(3)}\times C_2^{(2)})\times(C_1\times C_2^{(4)})$&$\begin{cases}D_3+a'\equiv g_i\\r_0\le D_3\\L=c+g_i,c\in C\\D_2=a'+c,D_4'\equiv h_i+p+q-c-a'\end{cases}$ \\ \hline
$(L,a,D_4,D_3',D_2')\in W_2\times(C_1\times C_2^{(4)})\times(C_1^{(3)}\times C_2^{(2)})$&$\begin{cases}D_3'+r_0\equiv g_i\\L=h_i+p+q-c,c\in C\\a=r_0,D_4=h_i+p+q-c-r_0\\D_2'=r_0+c\end{cases}$ \\ \hline

\end{tabular}
\vspace{.5cm}
\centering
\end{table}

\end{proposition}  

\begin{proof}The proof is entirely analogous to that of Proposition \ref{Znu1}, we omit the details.
\end{proof}

\begin{lemma}\label{lemZ4141} 
For any $s^1_4\in W^1_4$,  the intersection $\lambda_2^{-1}(s^1_4)\cap F_{(4,1)(4,1)}$ is empty except when $s^1_4=g_i$. The support of the intersection $\lambda_2^{-1}(g_i)\cap F_{(4,1)(4,1)}$ is of pure dimension $1$ and equal to
$$\{(L,D_4,a,D_4',a') : \ a+a'+p+q\equiv g_i,\ h^0(L-r_0-a)>0,D_4\equiv L-a, D_4'\equiv L'-a'\}.$$
\end{lemma} 
\begin{proof} The map $\lambda:F_{(4,1)(4,1)}\ra W_1\times\T_0$ factors through the projection of $F_{(4,1)(4,1)}\subset W_1\times(C_1^{(4)}\times C_2)\times(C_1^{(4)}\times C_2)$ to $W_{1}\times C_2\times C_2$, which is generically one to one to its image. The image of $F_{(4,1)(4,1)}$ in $W_1\times C_2\times C_2$ consisting of $(L,a,a')$ such that
$$ h^0(L-r_0-a)>0.$$
The image of $Z_{(4,1)(4,1)}=Z\cap F_{(4,1)(4,1)}$ under projection is defined {\em{scheme theoretically}} by imposing an extra condition 
$$ h^0(a+a'+p+q)>1.$$

 If $a+a'+p+q \equiv s^1_4 \in W^1_4$, then $s^1_4$ is equal to one of the $g_i$. This first shows that $\lambda_2^{-1} (s^1_4)$ is empty unless $s^1_4 = g_i$ for some $i$. Then it shows that there are only finitely many choices for $a$, hence $\lambda_2^{-1}(g_i)\cap F_{(4,1)(4,1)}$ of pure dimension one and as described .
\end{proof}

 \begin{lemma}\label{lemZ4123}
For $s^1_4\ne g_i$ or $h_i$, the intersection $\lambda_2^{-1}(s^1_4) \cap F_{(4,1)(2,3)}$ is finite. The intersection $\lambda_2^{-1}(h_i)\cap F_{(4,1)(2,3)}$ (up to finitely many points) has support
$$\{(L,D_4,a,D_2',D_3'):L=a+g_i, a\in C, r_0\le D_4\equiv g_i, D_3'\equiv h_i-a, D_2'=p+q \},$$
and the intersection $\lambda_2^{-1}(g_i)\cap F_{(4,1)(2,3)}$ (again, up to finitely many points) has support
$$\{(L,D_4,a,D_2',D_3'):L=a+g_j,j\ne i, a\in C, r_0\le D_4\equiv g_j,D_3'\equiv g_i-a,D_2'\equiv h_i-(g_i-p-q)\}$$
and $$\{(L,D_4,a,D_2',D_3'):L=h_i+p+q-c,c\in C, h^0(L-r_0-a)>0,D_3'\equiv g_i-a,D_2'=a+c)\}.$$
\end{lemma}
\begin{proof}
 Consider the projection of $Z_{(4,1)(2,3)}$ to  $W_{1}\times C_2\times C_2^{(3)}$ consisting of $(L,a,D_3')$ satisfying  equations
\begin{eqnarray}
&\label{eqnZ1}  h^0(L'-D_3')>0 & \\
\label{eqnZ4123} & h^0(L-r_0-a)>0 & \\
& \label{eqnZ2}h^0(a+D_3')>1. &
 \end{eqnarray}

  Fix any $s^1_4\in W^1_4(C)$. Suppose $a+D'_3\equiv s^1_4$. In the canonical space $|K_C|^*$, the span $\langle D'_3 \rangle$ is a plane ($C$ is not trigonal). By Riemann-Roch, $a \in \langle D_3' \rangle$. We have two cases:
  
  \begin{enumerate} 
  \item $a\not\leq \G_3':=K_C-L'$.  In this case, $h^0(L'-D_3')=h^0(K_C-\G_3'-D'_3)>0$ implies $h^0(K_C-\G'_3-s^1_4)>0$, i.e. $h^0(L'-s^1_4)>0$. If $s^1_4\ne g_i$, then $L'=p+s^1_4$ or $L'=q+s^1_4$ because $h^0 (L' -p-q ) >0$.
     In either case, there are finitely many choices of $a$ satisfying condition \eqref{eqnZ4123} and therefore there are finitely many points in $Z_{(4,1)(2,3)}$ that map to $s^1_4$. If $s^1_4=g_i$, then there exists $c\in C$ such that $L'= c+g_i$ and \eqref{eqnZ4123} becomes
 $$h^0(K+p+q-(c+g_i)-r_0-a)=h^0(h_i+p+q-c-r_0-a)>0.$$
For each $c$, there are 4 choices of $a$ satisfying the above condition, therefore $(L,a,D_3')$ in this case has to be
 $$\{(L=h_i+p+q-c,a,D_3'=g_i-a):\ c\in C,\ h^0(h_i+p+q-c-r_0 -a)> 0\},$$
 Finally, there is a unique lifting of such $(L,a,D_3')$  to a point $(L,D_4,a,D_2',D_3')$ in $Z_{(4,1)(2,3)}$ as described in the statement.

\item $a\leq\G'_3$. Write $\G_3'=a+\G_2'$. The conditions defining the fiber of $Z_{(4,1)(2,3)}$ over $s^1_4$ are

$$\begin{cases}
 h^0(K_C-a-\G_2'-D_3')=h^0(K_C-s^1_4-\G_2')>0\\
 h^0(K_C-\G_3-a-r_0)=h^0(\G_3'+p+q-a-r_0)=h^0(\G_2'+p+q-r_0)>0\\
 a+D_3'\equiv s^1_4
  \end{cases}$$
Put $h^1_4 := |K_C-s^1_4|$ so that, by the above, $h^0 (h^1_4 -\G_2') >0$.
 There are two subcases:
 \begin{enumerate} 
 \item  $h^0(\G_2'+p+q)=2$. So the second condition above is automatically satisfied.
 
Here $\G_2'+p+q\in g_{i}$ for some $i$.
 
\begin{claim}\label{claimhigi} The five $g^1_4$s containing $\G_2'$ are $g_{i}$ and $h_j$ for $j\ne i$.
\end{claim}
To prove this, denote $l_{pq}$ the line in $\bP^2=\bP(I_2(C))$ consisting of quadrics vanishing on the secant line $\langle p+q \rangle$ in $|K_C|^*$. There are five rank 4 quadrics $Q_j$, $j=1,...,5$ in $l_{pq}$, corresponding to the intersection of $l_{pq}$ with the quintic curve parametrizing rank 4 quadrics in $\bP (I_2(C))$. For each $j$, $g_j$ is cut on $C$ by one ruling of $Q_j$. Let $S$ be the base locus of the pencil $l_{pq}$. Then $S$ is a Del Pezzo surface of degree $4$. By construction $\langle p+q \rangle$ is contained in $S$.  Since the span $\langle p+q+\G_2'\rangle$ is a plane in $|K_C|^*$, $S\cap\langle p+q+\G_2'\rangle$ is a conic containing $\langle p+q \rangle$, thus $S\cap\langle p+q+\G_2'\rangle = \langle p+q \rangle\cup\langle \G_2' \rangle$. Therefore the pencil of quadrics containing $\langle \G_2' \rangle$ is also $l_{pq}$. We know that $\langle p+q+\G_2'\rangle \subset Q_{i}$. For all $j \neq i$, $Q_j \cap \langle p+q+\G_2'\rangle = S \cap \langle p+q+\G_2'\rangle = \langle p+q \rangle\cup\langle \G_2' \rangle$. So $\G_2'$ and $\langle p+q \rangle$ belong to different rulings of $Q_j$, i.e., $\G_2'$ is contained in the ruling of $Q_j$ corresponding to $h_j$ for $j\ne i$. The claim is proved.

Thus $h^1_4=g_{i}$ or $h^1_4 = h_j$ for some $j\ne i$.

So those $(L,a,D_3')$ which map to $s^1_4=h_{i}$ are 
$$\{(L=a+g_{i}, a, D_3'\equiv h_{i}-a):\ a\in C\}.$$
Similarly, the $(L,a,D_3')$ which map to $s^1_4=g_j$ for $j\ne i$ are 
$$\{(L = a+g_{i},a,D_3'\equiv g_j-a):\ a\in C,\ j\ne i\}.$$
There are unique liftings to points in $Z_{(4,1)(2,3)}$ as described in the statement of the proposition.

\item $h^0(\G_2'+p+q)=1$. Then the second condition implies $h^0 (\G'_2 -r_0) > 0$. For each $s^1_4$, there are finitely many choices of $\G_2'=r_0+b$ satisfying the first condition and for each choice of $\G_2'$, there are finitely many choices of $a$ such that $a+\G_2'\in W_{pq}$ (because this means $h^0 (K_C -a - \G_2' -p-q ) > 0$, and, since $h^0(\G_2'+p+q)=1$, we have $h^0 (K_C -\G_2' -p-q ) = 1$ as well). Therefore, there are no positive dimensional fibers in this case.
\end{enumerate}
\end{enumerate}
\end{proof}

\begin{lemma}\label{lemZ2341}
\begin{enumerate}
\item The only positive dimensional fibers in $Z_{(2,3)(4,1)}$ are $\lambda_2^{-1}(g_i)\cap Z_{(2,3)(4,1)}$. 
For each $i$, the one-dimensional components of $\lambda_2^{-1}(g_i)\cap Z_{(2,3)(4,1)}$ are supported on the curve
$$\{(L,D_2,D_3,D_4',a'): L=c+g_i,c\in C, D_3\equiv g_i-r_0, D_2=r_0+c, a'=r_0,D_4'\equiv h_i+p+q-c-r_0\}$$
 and
 $$ \{(L,D_2,D_3,D_4',a'):L=r_0+g_i, D_3\equiv g_i-a', D_2=a'+r_0, D_4'\equiv h_i+p+q-r_0-a',a'\in C\}\footnote{This component is contracted by $\lambda=(\lambda_1,\lambda_2)$, and therefore does not contribute to the Abel-Jacobi map in Section \ref{secAJ1}.}.$$
\item All fibers in $Z_{(2,3)(2,3)}$ are finite.
\end{enumerate}
\end{lemma}
\begin{proof}
\begin{enumerate}
\item
The projection of $Z_{(2,3)(4,1)}$ to $W_{pq}\times C_2^{(3)}\times C_2$ is the locus of 
$(L,D_3,a')$ satisfying
$$\begin{cases}h^0(L-D_3-r_0)>0\\
h^0(a'+D_3)>1.
\end{cases}$$
As in the previous lemma, only the inverse image of $g_i$ is positive dimensional, it is equal to
$$\{(L=c+g_i, a'=r_0, D_3\equiv g_i-r_0)\ :\ c\in C\} \cup \{(L=r_0+g_i, D_3\equiv g_i-a', a')\ :\ a'\in C\}.$$
As before, we can uniquely lift these curves to $Z_{(2,3)(4,1)}$.

\item The projection of $Z_{(2,3)(2,3)}$ to $W_{pq}\times C_1^{(2)}\times C_1^{(2)}$ is the locus of $(L,D_2,D_2')$ satisfying

$$\begin{cases} h^0(L-D_2)>0\\
 h^0(L'-D_2')>0\\
 r_0\leq D_2\\
 h^0(K_C-D_2-D_2')>1.
 \end{cases}$$
These cycles are also one dimensional but there are only finitely many points mapping to a fixed $s^1_4$ (we choose $r_0$ general such that $r_0+p+q$ is not in any $g^1_4$). 
\end{enumerate}\end{proof}
By the previous three Lemmas, Proposition \ref{Znu1} is proved.

We also need to describe the components of $Z_1$ which lie over $X_{1p}$ under $\lambda_1$. This will be needed in the computation of the Abel-Jacobi map in Section \ref{secGr3}.
\begin{lemma}\label{X1p} The scheme $Z_1$ has the following components which map onto $X_{1p}$ by $\lambda_1$.  Each component maps onto $W^1_4$ by $\lambda_2$. They are supported on
$$\{(L,D_2,D_3,D_4',a'):L=p + g^1_4,D_3\equiv g^1_4-r_0, D_2=p+r_0,a'=r_0,D_4'\equiv p+g^1_4-r_0 \}\subset Z_{(2,3)(4,1)},$$
$$\{(L,D_4,a,D_2',D_3'):L=p+g^1_4, a=p,r_0\le D_4\equiv g^1_4,a=p, D_3'\equiv K_C-g^1_4 -a\}\subset Z_{(4,1)(2,3)},$$
and
 $$\{(L,D_4,a,D_2',D_3'):L=p+g^1_4, h^0(g^1_4-r_0-a)>0, D_4\equiv p+g^1_4-a, D_3'\equiv K_C-g^1_4 -a,D_2'=a+q\}\subset Z_{(4,1)(2,3)}.$$
 
\end{lemma}
\begin{proof} Fix a general $L=p+g^1_4\in X_{1p}$. One easily sees from Table 2 that only $Z_{(2,3)(4,1)}$ and $Z_{(4,1)(2,3)}$ have a point over $L$. 

In $Z_{(2,3)(4,1)}$, the condition $h^0(p+g^1_4-D_3-r_0)>0$ implies that either $D_3\equiv g^1_4-r_0$ or $D_3=p+B_2$ with $h^0(g^1_4-r_0-B_2)>0$. In the first case, $h^0(a'+D_3)>1$ implies $a'=r_0$. This is because $D_3'$ can only be contained in at most one pencil of degree $4$. Thus we obtain the first curve in the statement of the lemma. The second case can not happen because $|a'+p+B_2|$ can not be a pencil for $p$ and $g^1_4$ general.

In $Z_{(4,1)(2,3)}$, there are four choices of $a$ such that $ h^0(p+g^1_4-r_0-a)>0$. The condition $h^0(L'-D_3')=h^0(q+K_C-g^1_4-D_3')>0$ implies that either $h^0(K_C-g^1_4-D_3')>0$ or $D_3'=p+B_2'$ with $h^0(K_C-g^1_4-B_2')>0$. In the first case, $h^0(a+D_3')>1$ implies that $D_3'\equiv K_C-g^1_4-a$. This is because $D_3'$ can only be contained in at most one pencil of degree $4$. Thus we obtain the curves in the statement of the lemma. In the second case, $|a+q+D_2'|$ can not be a pencil for $q$ and $g^1_4$ general.  Note that the last component is a degree 3 cover of $X_{1p}$ under $\lambda_1$.
\end{proof}

\subsection{Infinitesimal study of $F_{r_0}$ and $Z$} \label{subsecinfinite}
 In this subsection, we prove that each irreducible component of the center of the blow-up $F''_{r_0}\cap (W_k\times W^1_4)$ is generically smooth, or equivalently, generically reduced. We also prove that $F''_{r_0}$ is generically smooth along a general point in each irreducible component of $F''_{r_0}\cap (W_k\times W^1_4)$.  
 
 The infinitesimal study is similar for all components. So let us take one component, say the image in $W_1\times \T_0$ of the curve in $Z_{(4,1)(2,3)}$
 $$\{(L,D_4,a,D_2',D_3'):L=a+g_i, a\in C, r_0\le D_4\equiv g_i, D_3'\equiv h_i-a, D_2'=p+q \}.$$
 This curve projects isomorphically to (with identification $C_1=C_2=C$)
 \begin{eqnarray}\label{sampleZ}Z'_{(4,1)(2,3)}=\{(L,a,D_3'):L=a+g_i,a,D_3'\equiv h_i-a, a\in C\}\subset W_1\times C\times C^{(3)}
\end{eqnarray}
It surfaces to show that the curve $Z'_{(4,1)(2,3)}$ is generically reduced. To this end, recall that by \cite[p 189]{ACGH}, for any line bundle $M$ of degree $d$ on $C$, and $v\in H^1(\cO_C)=T_MPic^dC$ a tangent vector, all sections in $H^0(C,M)$ extend to first order along $v$ if and only if 
$$(v,Im\mu_M)_S=0$$
where $(,)_S$ is the pairing for Serre duality and 
$$\mu_M:H^0(M)\otimes H^0(K_C- M)\ra H^0(K_C)$$
is the multiplication map.

 Note that $Im\mu_{g_i}$ is of codimension $1$ in $H^0(K_C)$ by the base point free pencil trick. 
 
If we embed $W_1\times C\times C^{(3)}$ in $Pic^5C\times Pic^1C\times Pic^3C$, by the previous paragraph, the tangent space to $W_1\times C\times C^{(3)}$ at the point $(L,a,D_3')$ consists of $(v_1,v_2,v_3)\in H^1(\cO_C)^{\oplus3}$ such that 
\begin{eqnarray}\label{conditionv1}&&v_1\in Im\mu_L^\perp\cap Im\mu_{L'}^{\perp},\\
&&\label{conditionv2} v_2\in H^0(K_C-a)^{\perp},\\
&&\label{conditionv3}  v_3\in H^0(K_C-D_3')^\perp.
  \end{eqnarray}
 
 \begin{lemma}A tangent vector $(v_1,v_2,v_3)\in H^1(\cO_C)^{\oplus3}$ of $W_{pq}\times C\times C^{(3)}$ is tangent to $Z'_{(4,1)(2,3)}\subset W_{pq}\times C\times C^{(3)}$ at $(L=a+g_i,a,D_3'=h_i-a)$ if in addition the following holds
\begin{eqnarray}\label{conv13}&&v_1+v_3\in H^0(K_C-p-q)^\perp,\\
\label{conv12}&&v_1-v_2\in H^0(K_C-(g_i-r_0))^\perp,\\
&&\label{conv23}v_2+v_3\in Im\mu_{g_i}^\perp.
\end{eqnarray} 
\end{lemma} 
\begin{proof} The cycle $Z'_{(4,1)(2,3)}$ is defined scheme theoretically by (\ref{eqnZ1}), (\ref{eqnZ4123}), and (\ref{eqnZ2}). These translate into the above conditions for infinitesimal deformations.
\end{proof}

\begin{proposition}\label{propZsmooth}
Each irreducible component of $F_{r_0}''\cap(W_k\times W^1_4)$ is generically smooth. 
 \end{proposition}
 \begin{proof}  We only prove the proposition for the component which is the image in $W_1\times \T_0$ of $Z'_{(4,1)(2,3)}$.

Fix a general point $(L=a+g_i,a,D_3'\equiv h_i-a)$.  
 Consider the linear map form tangent space of $Z_{(4,1)(2,3)}$  to $H^1(\cO_C)$ which sends
 $(v_1,v_2,v_3)$ to $v_2+v_3$. Its image is 1 dimensional by (\ref{conv23}).
 To show the tangent space of $Z_{(4,1)(2,3)}$ is $1$ dimensional, it suffices to show that the kernel of this linear map is trivial, i.e. if $v_2+v_3=0$, then $v_1=v_3=0$.
 
 So assume $v_2+v_3=0$. Then
 $$v_1+v_3=v_1-v_2\in H^0(K_C-p-q)^\perp\cap H^0(K_C-(g_i-r_0))^\perp.$$
 Since the pencil $K_C-(g_i-r_0)=h_i+r_0$ does not have base points at $p$ or $q$ and can separate $p$ and $q$, we conclude 
 $$H^0(K_C-p-q)^\perp\cap H^0(K_C-(g_i-r_0))^\perp=(H^0(K_C-p-q)+H^0(K_C-(g_i-r_0)))^\perp=0.$$
 Therefore $v_1=-v_3$. Now by (\ref{conditionv1}) and (\ref{conditionv2}),
 $v_1=v_2=-v_3\in Im\mu_L^\perp\cap Im\mu_{L'}^\perp\cap H^0(K_C-a)^\perp=0$ for $a\in C$ general, this implies $v_1=v_2=v_3=0$.

  \end{proof}

\begin{proposition}\label{propFsmooth}
The scheme $F_{r_0k}$ is smooth at a general point of each component of $Z_k$.
\end{proposition}
\begin{proof} Again we only check the proposition for a general point of the image in $W_1\times \T_0$ of (\ref{sampleZ}). The defining equation for $F_{(41)(23)}\subset W_1\times C\times C^{(3)}$ is (\ref{eqnZ1}) and (\ref{eqnZ4123}). The tangent space of $F_{r_0}$ at $(L=a+g_i,a,D_3'=h_i-a)$ consists of $(v_1,v_2,v_3)$ satisfying the conditions from (\ref{conditionv1}) to (\ref{conv12}).  Projection to the $v_1$ summand of $(v_1,v_2,v_3)$ is surjective and the kernel of this projection is $1$ dimensional. The proposition follows.
\end{proof}

\subsection{The structure of the projectivized normal cone }

Note that $F'''_{r_0}|_{W_k\times M_2}$ is the projectivized normal cone of $F_{r_0}\cap (W_k\times\T_0)$ in ${\cF}_r''$. 

We have a commutative diagram 
\begin{eqnarray}\label{diagcone}\xymatrix{\cC_k\ar[r]\ar[d]&F'''_{r_0}|_{W_k\times M_2}\ar[r]^-{\rho_2}\ar[d]&M_2\ar[d]^-{\pi_2}\\
Z_k\ar[rd]_-{\lambda_1}\ar[r]^-{(\lambda_1,\lambda_2)}&F''_{r_0}\cap(W_k\times W^1_4)\ar[r]^-{Pr_2}\ar[d]^-{Pr_1}&W^1_4\\
&W_k}
\end{eqnarray}
where $\cC_k$ is defined by the fibered diagram. 
\begin{proposition}\label{conestructure}
$\cC_k$ is generically a $\bP^2$ bundle over the curve $\lambda_2^{-1}(\cup_i \{g_i, h_i\})\cap Z_k$.
\end{proposition}

\begin{proof} Since $W_k\times M_2$ is a divisor in the total space $\widetilde{\tilde{\cG}\times_T\tilde{\T}}$,  $F'''_{r_0}|_{W_k\times M_2} = \cF'''_{r} \cap (W_k\times M_2)$ is purely three dimensional. Furthermore, by Propositions \ref{propZsmooth} and \ref{propFsmooth}, at a generic point of any component of $\lambda_2^{-1}(\cup_i \{g_i, h_i\})\cap Z_k$, both $Z_k$ and $F_{r_0k}$ are smooth. Thus there is an open dense subset of $\lambda_2^{-1}(\cup_i \{g_i, h_i\})\cap Z_k$ where the dominant map $\cC_k\ra Z_k$ is a $\bP^2$ bundle. The general fiber of $\cC_k$ is therefore a $2$-dimensional linear subspace of the singular quadric threefold $Q_3^{sing}$ which is the fiber of $M_2$ over one of the $g_i$ or $h_i$. Therefore the general fiber is a $\bP^2$ passing through the vertex of $Q_3^{sing}$. 
\end{proof}

\section{The Abel-Jacobi map}
\label{secAJ1}
We are now ready to prove Propositions \ref{AJ21} to \ref{AJ2global}.

\subsection{Proof of Proposition \ref{AJ21}:} The map $AJ^0_1 : H^2(\tG_0^{[0]})\ra H^4(M_1)$.

We will show that it is enough to compute the restriction of $AJ^0_1$ to the direct summand $H^2(W_1)$ of $H^2(\tG_0^{[0]})$. This map is the correspondence induced by the cycle $[F_{r_0}'''|_{\widetilde{W_1\times M_1}}]\in H^6(\widetilde{W_1\times M_1})$. We use the notation introduced in Section \ref{subsecstrata}.



 There are two reduction steps. First, since we are computing $AJ^0_1$ modulo $\langle j_{1*}f,j_{1*}\tau_1\rangle$ in Proposition \ref{AJ21} (recall that $H^4(M_1)\cong p_1^*H^4(C^{(4)})\oplus \langle j_{1*}f,j_{1*}\tau_1\rangle$), 
it suffices to check that the image of the composition 
$$\xymatrix{H^2(W_1)\ar[r]^-{AJ^0_1}&H^4(M_1)\ar[r]^-{p_{1*}}&H^4(C^{(4)})}$$ 
contains $\eta H^2(Pic^4C)\oplus \eta^2$ modulo $\theta H^2(Pic^4C)$. Recall (see Proposition \ref{propF'''}) that ${F}_{r_0}'''|_{\widetilde {W_1\times M_1}}$ is the proper transform of ${F}_{r_0}''|_{ W_1\times\T_0}$ under
\[
\widetilde{W_1\times M_1}\lra W_1\times M_1\lra W_1\times C^{(4)}\lra W_1\times \T_0.
\]
By the projection formula, $p_{1*}\circ{AJ^{0}_1}$ is induced as a correspondence map by the proper transform $F''_{r_0}|_{W_1\times C^{(4)}}$ of  ${F}_{r_0}''|_{ W_1\times\T_0}$ in the intermediate space $W_1\times C^{(4)}$:
$$\xymatrix{H^2(W_1)\ar[r]&H^2(W_{1}\times C^{(4)})\ar[rr]^-{\cup[F''_{r_0}|_{W_1\times C^{(4)}}]}&&H^8(W_{1}\times C^{(4)})\ar[r]&H^4(C^{(4)}).}$$

Secondly, we will prove that in fact the image by $p_{1*}\circ{AJ^{0}_1}$ of the subspace $(q_1,q_2)^*H^2(C^{(3)}\times C^{(3)})$ of $H^2(W_1)$ contains $\eta H^2(Pic^4C)\oplus \eta^2$ modulo $\theta H^2 (Pic^4 C)$. We therefore compute the composition $\overline{AJ^{0}_1}$ 
\[
\xymatrix{\overline{AJ^{0}_1} : \: H^2(C^{(3)}\times C^{(3)})\ar[r]^-{(q_1,q_2)^*}&H^2(W_1)\ar[r]^-{p_{1*}\circ AJ^0_1}&H^4 (C^{(4)})\ar[r]&\frac{H^4 (C^{(4)})}{\theta H^2(Pic^4C)},}
\]
where
\[
\begin{array}{rcc}
(q_1, q_2) : W_{pq} & \lra &  C^{(3)}\times C^{(3)}\\
L \ \ &\longmapsto &(\Gamma_3,\Gamma_3')
\end{array}
\]
is the embedding used in Section \ref{ssectWCpqsupp}.




\begin{lemma}\label{totalclass}The $H^2(W_1)\otimes H^4(C^{(4)})$ Kunneth component of  $[{F}_{r_0}''|_{ W_1\times C^{(4)}}]\in H^6(W_1\times C^{(4)})$ is the restriction to $W_1\times C^{(4)} \subset C^{(3)}\times C^{(3)}\times C^{(4)}$ of 
\begin{eqnarray}\label{restriclass}(-2\theta_1+4\eta_1+4\eta_2)\eta_3^2+4\delta_{13}^2\eta_3+(\theta_1-\eta_1)\theta_3\eta_3.
\end{eqnarray}
in $H^6(C^{(3)}\times C^{(3)}\times C^{(4)})$
modulo $\theta_3 H^2(Pic^4C)$, where $\delta_{kl}=\sum_{i=1}^5(\xi_{ki}\xi_{li}'+\xi_{li}\xi_{ki}')$.
\end{lemma}
\begin{proof}The $H^2(W_1)\otimes H^4(C^{(4)})$ Kunneth component of $[{F}_{r_0}''|_{ W_1\times C^{(4)}}]$ is computed case by case for each bidegree in Appendix \ref{cycleclasses}. It is the sum of the classes in (\ref{class1a}), (\ref{class1b}), (\ref{class1c}), (\ref{class2}), (\ref{class3}), which is equal to the restriction to $W_1\times C^{(4)} \subset C^{(3)}\times C^{(3)}\times C^{(4)}$ of  
\begin{eqnarray}
\label{eqntotalclass}[-2\theta_1+4\eta_1+4\eta_2]\eta_3^2+[2\delta_{23}^2+\delta_{13}^2-\delta_{13}\delta_{23}+(\theta_1-\eta_1)\theta_3]\eta_3.\end{eqnarray}
Consider the commutative diagram
$$\xymatrix{&W_{pq}\ar_-{q_1}[ld]\ar^-{q_2}[rd]&\\
C^{(3)}\ar[d]&&C^{(3)}\ar[d]\\
Pic^3(C)\ar^{\tau}[rr]&&Pic^3(C)}$$
where $\tau$ is the involution sending $M$ to $K_C-p-q-M$. Since $\tau^*(\xi_i)=-\xi_i$, we see immediately that 
\begin{eqnarray}\label{eqnpullback1}&&q_1^*(\xi_i)=-q_2^*(\xi_i),\nonumber\\
\label{eqnpullback2}&&\delta_{13}|_{W_1\times C^{(4)}}=-\delta_{23}|_{W_1\times C^{(4)}}.\nonumber
\end{eqnarray}
Therefore (\ref{eqntotalclass}) simplifies to (\ref{restriclass}).
\end{proof}

Now, using the class (\ref{restriclass}), we obtain, for any $\omega\in H^2(C^{(3)})$,
\begin{eqnarray}
\overline{AJ^{0}_1}(\omega_1)&=&pr_{C^{(4)}*}\left\{\omega_1\left[(-2\theta_1+4\eta_1+4\eta_2)\eta_3^2+4\delta_{13}^2+(\theta_1-\eta_1)\theta_3\eta_3\right]\right\}|_{W_1\times C^{(4)}}.\nonumber
\end{eqnarray}
Expanding $\delta_{13}^2=\sum_{i,j=1}^5\left[2\xi_{1i}\xi_{1j}'\xi_{3i}'\xi_{3j}-\xi_{1i}\xi_{1j}\xi_{3i}'\xi_{3j}'-\xi_{1i}'\xi_{1j}'\xi_{3i}\xi_{3j}\right]$, we obtain
\begin{eqnarray}
\overline{AJ^{0}_1}(\omega_1)&=&  \left[\int_{W_1} \omega_1 (-2\theta_1+4\eta_1+4\eta_2)\right]\eta^2+8\sum_{i,j=1}^5\left[\int_{W_1}\omega_1\xi_{1i}\xi_{1j}'\right]\xi_i'\xi_j\nonumber\\
&&-4\sum_{i,j=1}^5\left[\int_{W_1}\omega_1\xi_{1i}\xi_{1j}\right]\xi_i'\xi_j'-4\sum_{i,j=1}^5\left[\int_{W_1}\omega_1\xi_{1i}'\xi_{1j}'\right]\xi_i\xi_j\nonumber
+\left[\int_{W_1}\omega_1(\theta_1-\eta_1)\right]\theta\eta . \nonumber
\end{eqnarray}
Noting that the class of $W_1$ in $C^{(3)}$ is $\theta-\eta$ under embedding $q_1$ or $q_2$, and $q_{1*}q_2^*\eta=\frac{1}{2}\theta^2-\theta\eta+\eta^2\in H^2(C^{(3)})$, the above formula becomes\footnote{We have a similar formula for $\overline{AJ^{0}_1}$ for $\omega\in H^2(W_2)$.}
\begin{eqnarray}\label{eqnAJ21}\overline{AJ^{0}_1}(\omega_1)&=&  \left[\int_{C^{(3)}} \omega (-2\theta+4\eta)(\theta-\eta)+4\int_{C^{(3)}}\omega(\frac{1}{2}\theta^2-\theta\eta+\eta^2)\right]\eta^2\\
&&+8\sum_{i,j=1}^5\left[\int_{C^{(3)}}\omega\xi_{i}\xi_{j}'(\theta-\eta)\right]\xi_i'\xi_j
-4\sum_{i,j=1}^5\left[\int_{C^{(3)}}\omega\xi_{i}\xi_{j}(\theta-\eta)\right]\xi_i'\xi_j'\nonumber\\
&&-4\sum_{i,j=1}^5\left[\int_{C^{(3)}}\omega\xi_{i}'\xi_{j}'(\theta-\eta)\right]\xi_i\xi_j
+\left[\int_{C^{(3)}}\omega(\theta-\eta)^2\right]\theta\eta . \nonumber
\end{eqnarray}
Now a simple computation using the ring structure of $H^{\bullet}(C^{(3)})$ described in Macdonald \cite{macdonald62}  gives 
\begin{eqnarray}
&&\overline{AJ^{0}_1}(\eta_1) =10\eta^2-11\theta\eta,\nonumber\\
&&\overline{AJ^{0}_1}(\xi_{1i}\xi_{1j})=c_{ij}\xi_{i}\xi_{j}\eta\  \text{for}\ 0\ne c_{ij}\in\bZ, \ j\ne i\pm5,\nonumber\\
&&\overline{AJ^{0}_1}(\sigma_{1k})=8\eta^2-11\theta\eta+16\sigma_{k}\eta,\nonumber
\end{eqnarray}

Thus the image of $\overline{AJ^{0}_1}$ contains $\eta H^2(Pic^4C)\oplus \eta^2$ modulo ${\theta H^2(Pic^4C)}$. \hfill \qed


\subsection{Proof of Proposition \ref{AJ22}:}\label{subsecAJ22} The map $AJ^0_2 : H^2(\widetilde{G}_0^{[0]}) \lra H^4 (M_2)$.

We will work with the restriction of $AJ^0_2$ to the direct summand $H^2(W_1)\oplus H^2(W_2)$ of $H^2(\widetilde{G}_0^{[0]})$:
\[
\xymatrix{H^2(W_{k})\ar[r]^-{\rho_1^*}&H^{2}(W_{k}\times M_2)\ar[rr]^-{\cup [{F}'''_{r_0}|_{W_k\times M_2}]}&&H^{8}(W_{k}\times M_2)\ar[r]^-{\rho_{2*}}&H^4(M_2)}.
\]
The relations between the various spaces involved are summarized in diagram (\ref{diagcone}).
The projection of $F'''_{r_0}|_{W_k\times M_2}$ to $W_k$ is supported on curves. By Section \ref{subsecZ''}, the image curve contains the following special curves in $W_k$
$$C_i :=\{c+g_i\ |\ c\in C\}, C_i' := \iota (C_i), i=1,...5,$$
$$X_{1p}=\{p+g^1_4\ |\ g^1_4\in W^1_4(C)\},$$
$$X_{1q}=\{q+g^1_4\ |\ g^1_4\in W^1_4(C)\}'$$
where $\iota (L) = |K_C +p+q -L|$.
 By Lemma \ref{Mcoho}, $H^4(M_2)$ is generated by $j_{2*}f$, $j_{2*}\tau_1$, $[\bP^2_i]$, $[\bP^2_{i+5}]$ (recall that $f$ is the class of the fiber of $\pi_{12}:M_{12}\lra W^1_4$ and see Lemma \ref{Mcoho} for the definition of $\bP^2_i$).

\begin{lemma}\label{lemAJ22}Put $[C]_{tot} :=[C_1]+...+[C_5]$. For any $(\alpha,\beta)\in H^2(W_1)\oplus H^2(W_2)$, 
$$AJ^0_2(\alpha) = \sum_{i=1}^5\left(\int_{W_{1}} \alpha\cdot [C_i]\right)[{\bP}^2_{i+5}]+\sum_{i=1}^5\left(\int_{W_{1}}\alpha\cdot([C]_{tot}+4[C'_i]+2q_1^*(\theta-\eta))\right)[\bP^2_i]$$
modulo $\langle j_{2*}f\rangle$, and,
$$AJ^0_2(\beta) = -\sum_{i=1}^5\left(\int_{W_2} \beta\cdot(3[C_i]+[C_i'])\right)[\bP_{i+5}^2]+\sum_{i=1}^5\left(\int_{W_2}\beta\cdot([C_i']+q_2^*(\theta-\eta))\right)[\bP^2_i]$$
modulo $\langle j_{2*}f\rangle$. 
\end{lemma}
\begin{proof} By Sections \ref{subsecZ''} and \ref{subsecinfinite}, the scheme $F''_{r_0}\cap(W_k\times W^1_4)$ is of pure dimension $1$ and generically reduced on each of its components.

Represent $\alpha$ as the cohomology class of a real $2$-chain in general position. By definition, $AJ^0_2(\alpha)$ is the push-forward to $M_2$ of the pull-back of $\lambda_1^*\alpha\cup[Z_1]$ to $\cC_k$. By Proposition \ref{conestructure}, the fibers of $\cC_k$ over $\lambda_2^{-1}(\cup_i \{g_i, h_i\})\cap Z_k$ are isomorphic to $\bP^2$.

Since we are computing $AJ^0_2$ modulo $\langle j_{2*}f\rangle \in H^4 (M_2)$, we only have to compute the intersection of $\lambda_1^*\alpha$ with $\lambda_2^{-1}(\cup_i \{g_i, h_i\})\cap Z_1$. The components of $\lambda_2^{-1}(\cup_i \{g_i, h_i\}))\cap Z_1$ are described in Proposition \ref{Znu1}. For instance, the curve supported on
$$\{(L,D_4,a,D_4',a') : \ a+a'+p+q\equiv g_i,\ h^0(L-r_0-a)>0,D_4\equiv L-a, D_4'\equiv L'-a'\}$$
has two components since we can switch $a$ and $a'$. Each component projects to a curve in $W_1$ whose class is $(\theta_1-\eta_1) |_{W_1}$ by the secant plane formula (Section \ref{appsecplane}). Thus the contribution of this curve is $\int_{W_1}\alpha\cdot2(\theta_1-\eta_1)[\bP^2_i]$. The formula for $AJ^0_2(\alpha)$ now easily follows.

The computation of $AJ^0_2(\beta)$ is analogous. The minus sign in the formula for $AJ^0_2(\beta)$ comes from the fact that the maps to $\T_0$ on the curves
$$\{(L,D_3,D_2,a',D_4'):L=c+g_i, c\in C,r_0\leq D_3, a' + D_3 \equiv g_i,D_2=a'+c,D_4'\equiv h_i+p+q-c-a'\}\subset Z_{(3,2)(1,4)}$$
$$\{(L,a,D_4,D_3',D_2'):L=h_i+p+q-c, a=r_0, r_0 + D_3'\equiv g_i, c\in C,D_2'=r_0+c\}\subset Z_{(1,4)(3,2)}.$$
are given by $\cO_C(K_C-D_3-a')$ and $\cO_C(K_C-D_3'-a)$ respectively (instead of $\cO_C(D_3+a')$ and $\cO_C(D_3'+a)$). Thus the $\bP^2$ fibers over these curves are in the rulings opposite to those of $\bP^2_{i+5}$. Since we work modulo $j_{2*}f$, the two rulings differ by a minus sign. \end{proof}

We need the following Lemma to study the rank of $AJ^0_2$.
\begin{lemma} \label{lemCiCj}
We have the following intersection numbers in the smooth surface $W_{pq}$
\[
C_i^2=C_i'^2=-2, C_i C_i'= C_i C_j=C_i' C_j'= 0, C_i C_j'=2, \text{ for } i\ne j.
\]
\end{lemma}
\begin{proof} Clearly $C_iC_j = C'_i C_j'=0$ for $i\ne j$. To compute $C_i^2$, consider the exact sequence
$$\xymatrix{0\ar[r]&N_{C_i|W_{pq}}\ar[r]&N_{C_i|C^{(3)}}\ar[r]&N_{q_2(W_{pq})|C^{(3)}}|_{C_i}\ar[r]&0}.$$
Under the embedding $q_2:W_{pq}\ra C^{(3)}$ sending $L$ to $|L-p-q|$, $C_i$ is a complete intersection with cohomology class $\eta^2\in H^4(C^{(3)})$. Therefore, $c_1(N_{C_i|C^{(3)}})=2$. We also have $c_1(N_{W_{pq}|C^{(3)}}|_{C_i})=\int_{C^{(3)}}[W_{pq}]\cdot[C_i]=\int_{C^{(3)}}(\theta-\eta)\eta^2=4$. We conclude that $C_i^2=-2$.

Now we compute $C_iC_j'$. Suppose $x+g_i\sim p+q+h_j-y$ for some $x,y\in C$. Then
$$D_{2i}:=g_i-p-q=h_j-x-y.$$
By Claim \ref{claimhigi}, for a fixed $i$, the $g^1_4$s containing $D_{2i}$ are $g_i$ and $h_l$ for $l\ne i$. This implies $C_iC_i'=0$ and $C_iC_j'=2$ (embedding $C^{(3)}$ in $Pic^3 C$, one easily sees that the intersection of $C_i$ and $C_j'$ is transverse for a general choice of $p+q$). 
\end{proof}

Using the formula in Lemma \ref{lemAJ22} and the intersection numbers in Lemma \ref{lemCiCj} we compute
\begin{eqnarray}
AJ^0_2 \: : \: H^2(W_1) \oplus H^2(W_2) \:\: & \lra & H^4 (M_2) / \langle j_{2*}f\rangle \nonumber \\
(12[C_i]-3[C]_{tot}+{2}q_2^*(\eta-\sigma_i),\ 3[C_i]) & \longmapsto & -58[\bP^2_i]+44\sum_{j\ne i,j=1}^5[\bP_j^2]\ \text{mod}\ \langle j_{2*}f\rangle. \label{eqnAJ22}
\end{eqnarray}
It immediately follows that the image of $AJ^0_2 $ contains $\langle[\bP^2_i]: i=1, \ldots , 5\rangle$ modulo $j_{2*}f$. We then compute that  
\begin{eqnarray}\label{eqnAJ22'} AJ^0_2([C_i],[C_i'])=-6\sum_{j\ne i,j=1}^5[\bP^2_{j+5}]
\end{eqnarray}
 modulo $\langle[\bP^2_i],  j_{2*}f|\ i=1,...,5\rangle$.
 Proposition \ref{AJ22} follows immediately.
\hfill\qed 
 \subsection{Proof of Proposition \ref{AJ1E1}:} The map $AJ^1 : H^1 (\tG_0^{[1]}) \lra H^3 (M_{12})$.\label{secGr3}

It follows from Sections \ref{subseccomp} and \ref{subsecAJE1} that the only double loci of the central fiber $\widetilde{\tilde{\cG}\times_T\tilde{\T}}_0$ inducing non-trivial Abel-Jacobi maps are those which map to $X_{kp}$ or $X_{kq}$ under $\rho_1$ and map to $M_{12}$ under $\rho_2$. These are the slanted lines in the picture in Section \ref{subseccomp}. Recall (see Section \ref{rkcentralfiber}) that $H^1(\tG_0^{[1]})=H^1(X_{1p})\oplus H^1(X_{1q})\oplus H^1(X_{2p})\oplus H^1(X_{2p})$ and $H^3 (M_{12}) = \tau_1\cdot\pi_{12}^*H^1(W^1_4)\oplus j_2^*e_2\cdot \pi_{12}^*H^1(W^1_4)$ (see Lemma \ref{Mcoho}). To prove Proposition \ref{AJ1E1}, it is sufficient to prove that the image of the summand $H^1(X_{1q})$ by $AJ^1$ contains $\tau_1\cdot\pi_{12}^*H^1(W^1_4)$.
The map $AJ^1$ on this direct summand is given by 
\begin{eqnarray}\label{doublemap}\xymatrix{H^1(X_{1p})\ar[r]^-{\rho_1^*}&H^1(E_{1p})\ar[r]^-{\cup[F_{r_0}'''|_{E_{1p}}]}&H^7(E_{1p}))\ar[r]^-{\rho_{2*}}&H^3(M_{12})},
\end{eqnarray}
where $E_{1p}$ corresponds to the slanted line labeled $b$ in the picture in Section \ref{subseccomp}. Therefore $E_{1p}$ is a $\bP^1$-bundle over $X_{1p}\times M_{12}$ and fits into the diagram
\[
\xymatrix{E_{1p}\ar[d]&\\
X_{1p}\times M_{12}\ar[r]\ar[d]&M_{12}\\
X_{1p}.}
\]

 By the projection formula, to compute (\ref{doublemap}), it suffices to compute the correspondence induced by the push-forward cycle of $[F_{r_0}'''|_{E_{1p}}]$ to $X_{1p}\times M_{12}$. Denote $Y'$ the projectivized normal cone of $F''_{r_0}\cap(W_{1}\times W^1_4)$ in $F_{r_0}''|_{W_{1}\times\T_0}$. By construction, $Y'$ has dimension $2$ and $Y'=(W_1\times M_{12})\cap F_{r_0}'''|_{W_{1}\times M_2}$. 
 

 The components of $Z_1$ which dominate $X_{1p}$ are described in  Lemma \ref{X1p}. Let $Z_{1p}$ denote the union of these components and let $Y$ be the fiber product $Z_{1p}\times_{F''_{r_0}\cap(W_k\times W^1_4)} Y'$, which is
 generically a $\bP^1$-bundle over $Z_{1p}$ (the $\bP^1$ in the ruling corresponds to $\tau_1$ because the map $\lambda_2$ from $F_{(4,1)(2,3)}$ and $F_{(2,3)(4,1)}$ factors through $C^{(4)}\stackrel{\phi}\ra\T_0$):
\[
\xymatrix{Y\ar[r]\ar[d]&Y'\ar[rr]\ar[d]&&M_{12}\ar[d]^-{\pi_{12}}\\
Z_{1p}\ar[r]^-{(\lambda_1,\lambda_2)}\ar[d]^-{\lambda_1}&F''_{r_0}\cap(W_k\times W^1_4) \ar[d]^-{Pr_1}\ar[rr]^-{Pr_2}&&W^1_4\\
X_{1p}\ar[r]& W_1.}
\]

For a real one cycle $\alpha$ in general position in $X_{1p}$,  the inverse image of $\alpha$ in $Y$ is a $\bP^1$-bundle over $\alpha$. The push-forward of the class of this $\bP^1$-bundle to $M_{12}$ is a class in $H^3(M_{12})$.  As the class of $\alpha$ varies in $H^1(X_{1p}) \cong H^1 (W^1_4)$, the class in $H^3(M_{12})$ spans $\tau_1\cdot\pi_{12}^*H^1(W^1_4)$ because $X_{1p}$ and $W^1_4$ are isomorphic to each other. \hfill \qed 

\subsection{Proof of Proposition \ref{AJ2global}:} Passage to the $E_2$ terms.

Recall that $\tG_0$ has four components and $E_2^{0,2}=Gr_2H^2(\tG_0)$ is the kernel of 
$$\xymatrix{H^2(\tG_0^{[0]})\ar^-{d_1}[r]\ar[d]^-{\cong}&H^2(\tG_0^{[1]})\ar[d]^-{\cong}\\
\oplus_{k=1}^2H^2(W_k)\oplus H^2(P_k)\ar[r]&\oplus_{k=1}^2H^2(X_{kp})\oplus H^2(X_{kq})}$$
Consider the subspace of $Gr_2H^2(\tG_0)$ consisting of $(x_1, x_2, \beta_1, \beta_2)$ with $x_k\in H^2(W_k)$ and $\beta_k\in H^2(P_k)$ such that $\beta_k$ is a multiple of the class of fiber of the $\bP^1$-bundle $P_k$. Since we always have
$$\int_{X_{kp}}\beta_k=\int_{X_{3-k,q}}\beta_k,$$
the compatibility condition defining $Ker(d_1)$ becomes 
\begin{eqnarray}\label{compatible}\int_{X_{kp}}x_k=\int_{X_{{3-k},q}}x_{3-k}.
\end{eqnarray}
Because the cycles $F'''_{r_0}|_{P_k\times M_{1}}$ and $F'''_{r_0}|_{P_k\times M_{2}}$ come from a base change (Proposition \ref{propcycleP}), the maps $AJ^0_1$ and $AJ^0_2$ are trivial on $\beta_k\in H^2(P_k)$. We will therefore write $AJ^0_i (x_1, x_2) := AJ^0_i (x_1, x_2, \beta_1, \beta_2)$.

Now start with $(\gamma_1,\gamma_2)\in (I\oplus H^4(M_2))\cap Gr_4 H^4 (\tT_0)$. The condition $(\gamma_1,\gamma_2)\in Gr_4 H^4 (\tT_0)$ means $j_1^*\gamma_1=j_2^*\gamma_2\in H^4(M_{12})$ by Proposition \ref{propwH4}. By Proposition \ref{AJ22}, we can choose $(x_1,x_2)\in H^2(W_1)\oplus H^2(W_2)$ such that:
 \[
 \gamma_2-AJ^0_2(x_1,x_2)\in \langle j_{2*}f,j_{2*}\tau_1\rangle.
 \]
Furthermore, note that in formula \eqref{eqnAJ22} and \eqref{eqnAJ22'}, we have chosen  $x_1$ and $x_2$ so that 
\[
\int_{X_{1p}}x_1=\int_{X_{2q}}x_2,\quad \int_{X_{1q}}x_1=\int_{X_{2p}}x_2.
\]
Subtracting $(AJ^0_1(x_1,x_2),AJ^0_2(x_1,x_2))$ from  $(\gamma_1,\gamma_2)$, we may assume $\gamma_2\in  \langle j_{2*}f,j_{2*}\tau_1\rangle\subset H^4(M_2)$.
 Now choose $\omega\in H^2(C^{(3)})$ such that for $i=1,...,5$ (see Section \ref{subsecAJ22} for the notation),
\begin{eqnarray}\label{eqncompatible}
\int_{X_{1p}} q_1^* \omega=\int_{X_{1q}} q_1^* \omega =0
\end{eqnarray}
and
\begin{eqnarray}\label{eqnvanish}
\int_{C_i} q_1^* \omega=\int_{W_{1}} q_1^* \omega \cdot([C]_{tot}+4[C'_i]+2q_1^*(\theta-\eta)))=0.
\end{eqnarray}
The equations (\ref{eqncompatible}) imply $(q_1^*\omega,0)\in Gr_2H^2(\tG_0)$. The equations (\ref{eqnvanish}) imply 
\[
AJ^0_2(q^*_1\omega,0)\in \langle j_{2*}f,j_{2*}\tau_1\rangle
\]
by the formula for $AJ^0_2$ in Lemma \ref{lemAJ22}.

 By the secant plane formula,
\begin{eqnarray}
&&q_{1*}[C_i]=\frac{1}{2}\theta^2-\theta\eta+\eta^2,\nonumber\\
&&q_{1*}[C_i']=\eta^2,\nonumber\\
&&q_{1*}[X_{1p}]=q_{1*}[X_{1q}]=\frac{1}{2}\theta^2-\theta\eta.\nonumber
\end{eqnarray}
Therefore the equations (\ref{eqncompatible}) and (\ref{eqnvanish}) together impose two conditions on $\omega$ since
\[
 \langle q_{1*}[C_i],q_{1*}(2[C_i']+q_1^*(\theta-\eta)),q_{1*}[X_{1p}],q_{1*}[X_{1q}]\rangle=\langle\theta^2,\theta\eta,\eta^2\rangle=\langle\theta\eta,\eta^2\rangle \subset H^4 (C^{(3)}).
\]
So, if we choose
\[
\omega \in\langle\xi_i\xi_j,\sigma_{k}-\sigma_{1}|\ i\ne j\pm5, k=2,...,5\rangle=\langle\theta\eta,\eta^2\rangle^{\perp},
\]
by the formula for $AJ^0_2$ in Lemma \ref{lemAJ22},
  $AJ^0_2(q^*_1\omega,0)\in \langle j_{2*}f,j_{2*}\tau_1\rangle$. Similarly,
 we can choose $\omega'\in H^2(C^{(3)})$ such that $q_1^*(\omega')$ satisfies (\ref{eqncompatible}) and
 $$AJ^0_2(0,q_1^*\omega')\in \langle j_{2*}f,j_{2*}\tau_1\rangle.$$
  
  By formula (\ref{eqnAJ21}), if we modify $(\gamma_1,\gamma_2)$ by a linear combination of $(AJ^0_1(q_1^*\omega,0),AJ^0_2((q_1^*\omega,0))$, $(AJ^0_1(0,q_1^*\omega'),AJ^0_2(0,q_1^*\omega'))$ and $(p_1^*(\theta H^2(Pic^4C)),0)$, we have $\gamma_1=-j_{1*}y_1$ and $\gamma_2=j_{2*}y_2$ for $y_1,y_2\in H^2(M_{12})$. But since  $j_1^*\gamma_1=j_2^*\gamma_2\in H^4(M_{12})$, we conclude immediately that $y_1=y_2$, thus $(\gamma_1,\gamma_2)\in Im(-j_{1*},j_{2*})$. \hfill \qed
 
 
 


\section{Appendix}
\label{secAppx}
\subsection{The cohomology of $C^{(k)}$.}\label{sum}

Let $C^k$ be the $k$-fold Cartesian product of a smooth curve $C$ of genus $g$ and $m$ be the natural map from $C^k$ to $C^{(k)}$. We identify the cohomology $H^{\bullet}(C^{(k)})$ with its image under $m^*$, which is the invariant subring of $H^\bullet(C^{k})$ under the action of the Symmetric group $\Sg_k$. 

Macdonald \cite{macdonald62} proved that the cohomology ring $H^{\bullet}(C^{(k)},\bZ)$ is generated by (see Notation and Conventions (2))
$$\xi_i\in H^1(C^{(k)},\bZ)\cong H^1(Pic^k(C),\bZ),\ i=1,...,2g$$
and the class $\eta\in H^2(C^{(k)},\bZ)$ subject to the following relations:
\begin{eqnarray}\label{symrelation}\xi_A\xi'_B(\sigma_C-\eta)\eta^d=0
\end{eqnarray}
where $A$,$B$,$C$ are mutually disjoint subsets of $\{1,...,g\}$ and $|A|+|B|+2|C|+d=k+1$,
$\xi_A=\Pi_{i\in A}\xi_i$, $(\sigma_C-\eta)=\Pi_{i\in C}(\sigma_i-\eta)$, etc.

\subsection{The secant plane formula \cite[p. 342]{ACGH}}\label{appsecplane}

Let $|V|\subset |L|$ be a $g^r_d$. Fix $d\ge k\ge r$, consider the following cycle
$$\{D\in C^{(k)}|\ E-D\ge0 \text{ for some } E\in|V|\}\subset C^{(k)}.$$
The cohomology class of the above cycle is given by
\begin{eqnarray}\label{secant}
\sum_{l=0}^{k-r}{d-g-r\choose l}\frac{\eta^l\theta^{k-r-l}}{(k-r-l)!}
\end{eqnarray}

\subsection{The Gysin maps}
If $\omega\in H^\bullet(C^{k},\bZ)$, the Gysin push-forward for the sum map  
$$m_*: H^{\bullet}(C^k,\bZ)\ra H^{\bullet}(C^{(k)},\bZ)$$
is given by
$$m_*(\omega)=\sum_{\sigma\in \Sg_k}\sigma^*(\omega).$$

If $\omega$ is $\Sg_k$-invariant, then 
$$m_*(\omega)=k! \ \omega$$
reflecting the fact that $m$ is generically $k!$ to 1.

Fix $k_1+k_2=k$, and let $m_1$ and $m_2$ be the symmetrization maps
$$\xymatrix{C^k\ar[r]^-{m_1}&C^{(k_1)}\times C^{(k_2)}\ar[r]^-{m_2}&C^{(k)}.}$$
For a cohomology class $\omega'\in H^\bullet(C^{(k_1)}\times C^{(k_2)})$ we have
$$m_{2*}(\omega')=\frac{1}{deg(m_1)}m_{*}(m_1^*\omega')=\frac{1}{deg(m_1)}\sum_{\sigma\in S_k}\sigma^*(m_1^*\omega').$$
In our case $g_C=5$ and we have the following lemmas (whose proofs are straightforward computations).
 \begin{lemma}\label{pushH2}The Gysin map $m_*: H^2(C^{(2)}\times C^{(2)})\longrightarrow H^2(C^{(4)})$ acts as follows:
\[
\begin{array}{lcll}1\otimes\theta & \longmapsto & \theta+10\eta, & \\
1\otimes\eta & \longmapsto & 3\eta, & \\
\xi_i\otimes\xi_j & \longmapsto & 2\xi_i\xi_j & \text{for}\ j\ne i\pm5,\\
\xi_i\otimes\xi_{i\pm5} & \longmapsto & 2\xi_i\xi_{i\pm5}\mp2\eta, &\\
\xi_i\xi_j\otimes 1 & \longmapsto & \xi_i\xi_j & \text{for}\ j\ne i\pm5.
\end{array}
\]
\end{lemma}

\begin{lemma}\label{pushH4}The Gysin map $m_*: H^4(C^{(2)}\times C^{(2)})\longrightarrow H^4(C^{(4)})$ acts as follows:
\[
\begin{array}{lcll}\eta\otimes(\xi_i\cdot\xi_{i+5}) & \longmapsto & \eta\xi_i\xi_{i+5}+\eta^2, & \\
\eta\otimes\xi_i\xi_j & \longmapsto & \eta\xi_i\xi_j & \text{for} \ j\ne i\pm5,\\
\eta\otimes\eta & \longmapsto &  2\eta^2, &\\
\eta\xi_i\otimes\xi_j & \longmapsto & \eta\xi_i\xi_j & \text{for}\ j\ne i\pm5,\\
\eta\xi_i\otimes\xi_{i\pm5} & \longmapsto & \eta\xi_i\xi_{i\pm5}\mp\eta^2, &\\
\sigma_k\otimes\sigma_k & \longmapsto & 2\sigma_k\eta & \text{for}\ k=1,..5,\\
\sigma_k\otimes\sigma_l & \longmapsto & \sigma_k\sigma_l+\eta^2\ k\ne l, &\\
\sigma_k\otimes\xi_k\xi_j & \longmapsto & \xi_k\xi_j\eta & \text{for}\ j\ne k+5,\\
\sigma_k\otimes\xi_i\xi_j & \longmapsto & \sigma_k\xi_i\xi_j & \text{for}\ i,j\notin\{k,k+5\},\\
\eta^2\otimes1 & \longmapsto &  \eta^2. &
\end{array}
\]
\end{lemma}

 \begin{lemma}\label{push} The Gysin map 
$m_*: H^4(C\times C^{(3)})\longrightarrow H^4(C^{(4)})$ acts as follows:

\[
\begin{array}{lcll}
\eta\otimes\xi_i\xi_j & \longmapsto &  \eta\cdot\xi_i\xi_j & \text{for}\ 1\le i,j\le 10,\\
1\otimes\eta\xi_i\xi_j & \longmapsto &  \eta\cdot\xi_i\xi_j & \text{for}\ j\ne i\pm5,\\
1\otimes\eta\sigma_i & \longmapsto &  \eta\cdot\sigma_i+\eta^2, &\\
\xi_i\otimes\xi_j\xi_k\xi_l & \longmapsto &  \xi_i\xi_j\xi_k\xi_l & \text{for}\ j,k,l\ne i\pm5,\\
\xi_i\otimes\xi_{i\pm5}\xi_k\xi_l & \longmapsto &  \xi_i\xi_{i\pm5}\xi_k\xi_l\mp\eta\cdot\xi_k\xi_l & \text{for}\ k,l\ne i\pm5,\\
\xi_i\otimes\eta\xi_j & \longmapsto & \eta\xi_i\xi_j & \text{for}\ j\ne i\pm5,\\
\xi_i\otimes\eta\cdot\xi_{i\pm5} & \longmapsto & \eta\cdot\xi_i\xi_{i\pm5}\mp\eta^2, &\\
\eta\otimes\eta & \longmapsto & \eta^2, &\\
1\otimes\eta^2 & \longmapsto &  2\eta^2. &
\end{array}
\]

\end{lemma}

\subsection{The cycle class of $F_{r_0}''|_{W_1\times C^{(4)}}$}\label{cycleclasses} 
\label{ssecAJ2}

We use the secant plane formula (Section \ref{appsecplane}) to compute the cycle class of $F_{r_0}''|_{W_1\times C^{(4)}}$ in each bidegree. For each bidegree $(d_1, d_2) + (e_1, e_2)$, the corresponding cycle $F_{(d_1, d_2)(e_1, e_2)}\subset W_k \times C_1^{(d_1)}\times C_1^{(e_1)} \times C_2^{(d_2)} \times C_2^{(e_2)}$  projects generically injectively to a product of some of the factors. Since the map $\lambda:F_{(d_1, d_2)(e_1, e_2)}\ra W_k\times\T_0$ factors through these projections, we only need the cycle class of the projection of $F_{(d_1, d_2)(e_1, e_2)}$.

\begin{enumerate}
\item (4,1)+(2,3)
We first compute the class of the projection of $F_{(4,1)(2,3)}$ to $C^{(3)}\times C^{(3)}\times C\times C^{(3)}$ (with the identification $C_1=C_2=C$ and the embedding of $W_1$ into $C^{(3)}\times C^{(3)}$ via $(q_1,q_2)$)
The cycles are given by the following conditions
$$(\Gamma_3,\Gamma_3',a,D_3')\in C^{(3)}\times C^{(3)}\times C\times C^{(3)}$$
$$\begin{cases}
h^0(K_C-p-q-\G_3-\G_3')>0\\
  h^0(K_C-\Gamma_3'-D_3')>0\\
 h^0(K_C-\G_3-r_0-a)>0 \\
 \end{cases}$$
 The map $\lambda_2|_{F_{(4,1)(2,3)}}$ factors through $m$ which sends $(\Gamma_3,\Gamma_3',a,D_3')$ to $(\Gamma_3,\Gamma_3',a+D_3')\in C^{(3)}\times C^{(3)}\times C^{(4)}$.

By the secant plane formula (\ref{secant}), the cycle class is given by the pull-back under the sum map from $C^{(3)}\times C^{(3)}$ (the first and fourth factor) to $C^{(6)}$ of the class
$$\frac{1}{2}\theta^2-\eta\theta+\eta^2\in H^4(C^{(6)})$$
cupped with the pull back to $C^{(3)}\times C$ (second and third factor) of 
$$\theta-\eta\in H^2(C^{(4)}),$$
then restriction to $W_1\times C\times C^{(3)}$.
Thus we obtain (c.f. Notation (\ref{notesym}))
$$[\frac{1}{2}(\theta_2+\theta_4+\delta_{24})^2-(\eta_2+\eta_4)(\theta_2+\theta_4+\delta_{24})+(\eta_2+\eta_4)^2]\cdot[(\theta_1+\theta_3+\delta_{13})-(\eta_1+\eta_3)].$$
We only need the $H^2(C^{(3)}\times C^{(3)})\otimes H^4(C\times C^{(3)})$ Kunneth component of this cycle class.  We organize the terms according to the types in the Kunneth decomposition.
\begin{enumerate}
\item Type $(2,0,0,4)$. 
\begin{eqnarray}&&(\frac{1}{2}\theta_4^2-\theta_4\eta_4+\eta_4^2)(\theta_1-\eta_1)\nonumber\\
&=&(\sum_{i<j}[\sigma_{4i}\sigma_{4j}]-\theta_4\eta_4+\eta_4^2)(\theta_1-\eta_1)\nonumber\\
&=&(\sum_{i<j}[(\sigma_{4i}+\sigma_{4j})\cdot\eta_4-\eta_4^2]-\theta_4\eta_4+\eta_4^2)(\theta_1-\eta_1)\nonumber\\
&=&3(\theta_4\eta_4-3\eta_4^2)(\theta_1-\eta_1)\nonumber
\end{eqnarray}
\item Type $(0,2,2,2)$
 \begin{eqnarray}
&&(\frac{1}{2}\delta_{24}^2+\theta_2\theta_4-\theta_2\eta_4-\theta_4\eta_2+2\eta_2\eta_4)(\theta_3-\eta_3)\nonumber\\
&=&(\frac{1}{2}\delta_{24}^2+\theta_2\theta_4-\theta_2\eta_4-\theta_4\eta_2+2\eta_2\eta_4)4\eta_3\nonumber
\end{eqnarray}

\item Type $(1,1,1,3)$
$$(\theta_4\delta_{24}-\eta_4\delta_{24})\delta_{13}$$
 \end{enumerate}
By Lemma \ref{push}, the  push-forwards  of these classes to $C^{(3)}\times C^{(3)}\times C^{(4)}$ are
 \begin{enumerate}
\item \begin{eqnarray}\label{class1a}
&&m_*(3\theta_4\eta_4-9\eta_4^2)(\theta_1-\eta_1)\\
&=&3(\theta_3\eta_3-\eta_3^2)(\theta_1-\eta_1)\nonumber
\end{eqnarray}
\item \begin{eqnarray}\label{class1b}
&&m_*(\frac{1}{2}\delta_{24}^2+\theta_2\theta_4-\theta_2\eta_4-\theta_4\eta_2+2\eta_2\eta_4)4\eta_3\\
&=&m_*(2\eta_3\delta_{24}^2)+4m_*[(\theta_2\theta_4-\theta_2\eta_4-\theta_4\eta_2+2\eta_2\eta_4)\eta_3]\nonumber\\
&=&m_*2\eta_3\sum_{i,j=1}^5[-\xi_{2i}\xi_{2j}\xi_{4i}'\xi_{4j}'+2\xi_{2i}\xi_{2j}'\xi_{4i}'\xi_{4j}-\xi_{2i}'\xi_{2j}'\xi_{4i}\xi_{4j}]+4m_*[(\theta_2\theta_4-\theta_2\eta_4-\theta_4\eta_2+2\eta_2\eta_4)\eta_3]\nonumber\\
&=&2\eta_3\sum_{i,j=1}^5[-\xi_{2i}\xi_{2j}\xi_{3i}'\xi_{3j}'+2\xi_{2i}\xi_{2j}'\xi_{3i}'\xi_{3j}-\xi_{2i}'\xi_{2j}'\xi_{3i}\xi_{3j}]+4[\theta_2\eta_3\theta_3-\theta_2\eta^2_3-\eta_2\eta_3\theta_3+2\eta_2\eta_3^2]\nonumber\\
&=&2\eta_3\delta_{23}^2+4[\theta_2\eta_3\theta_3-\theta_2\eta^2_3-\eta_2\eta_3\theta_3+2\eta_2\eta_3^2] \nonumber
\end{eqnarray}
\item \begin{eqnarray}\label{class1c}
&&m_*(\theta_4-\eta_4)\delta_{13}\delta_{24}\\
&=&m_*(\theta_4-\eta_4)\sum_{i,j=1}^5[-\xi_{1i}\xi_{2j}\xi_{3i}'\xi_{4j}'-\xi_{1i}'\xi_{2j}'\xi_{3i}\xi_{4j}+\xi_{1i}\xi_{2j}'\xi_{3i}'\xi_{4j}+\xi_{1i}'\xi_{2j}\xi_{3i}\xi_{4j}']\nonumber\\
&=&m_*(\theta_4-\eta_4)\sum_{i=1}^5[-\xi_{1i}\xi_{2i}\xi_{3i}'\xi_{4i}'-\xi_{1i}'\xi_{2i}'\xi_{3i}\xi_{4i}+\xi_{1i}\xi_{2i}'\xi_{3i}'\xi_{4i}+\xi_{1i}'\xi_{2i}\xi_{3i}\xi_{4i}']\nonumber\\
&&+m_*(\theta_4-\eta_4)\sum_{i,j=1,i\ne j}^5[-\xi_{1i}\xi_{2j}\xi_{3i}'\xi_{4j}'-\xi_{1i}'\xi_{2j}'\xi_{3i}\xi_{4j}+\xi_{1i}\xi_{2j}'\xi_{3i}'\xi_{4j}+\xi_{1i}'\xi_{2j}\xi_{3i}\xi_{4j}']\nonumber\\
&=&0+\sum_{i=1}^5[\xi_{1i}\xi_{2i}'(-\sigma_{3i}\theta_3+\eta_3\theta_3-\eta_3^2)+\xi_{1i}'\xi_{2i}(\sigma_{3i}\theta_3-\eta_3\theta_3+\eta_3^2)]\nonumber\\
&&+\sum_{i,j=1, i\ne j}^5[-\xi_{1i}\xi_{2j}\xi_{3i}'\xi_{3j}'-\xi_{1i}'\xi_{2j}'\xi_{3i}\xi_{3j}+\xi_{1i}\xi_{2j}'\xi_{3i}'\xi_{3j}+\xi_{1i}'\xi_{2j}\xi_{3i}\xi_{3j}']\theta_3\nonumber\\
&=&\delta_{12}(\eta_3\theta_3-\eta_3^2)+\delta_{13}\delta_{23}\theta_3 \nonumber
\end{eqnarray}
\footnote{Note that for $i\ne j$, 
$m_*\xi_{3i}'\xi_{4j}'\sigma_{4k}=\begin{cases}
0,\ k=j\\
\xi_{3i}'\xi_{3j}'\eta_3,\ k=i\\
\xi_{3i}'\xi_{3j}'\sigma_{3k},\ k\ne i,j
\end{cases}$
and 
$m_*\xi_{3i}'\xi_{4j}'\eta_4=\xi_{3i}'\xi_{3j}'\eta_3$}

\end{enumerate}
\item (2,3)+(4,1)

The cycle is 
$$\{(\Gamma_3,\G_3',a',D_3)\in C^{(3)}\times C^{(3)}\times C\times C^{(3)}\ |\ h^0(\cO_C(K_C-r_0-\Gamma_3-D_3)>0\}.$$
The map $m$ sends $(\Gamma_3,\G_3',a',D_3)$ to $(\Gamma_3,\G_3',a'+D_3)\in C^{(3)}\times C^{(3)}\times C^{(4)}$.

Its class is given by pulling back under the sum map to $H^6(C^{(3)}\times C^{(3)})$ of 

$$\frac{\theta^3}{6}-\frac{\eta\theta^2}{2}+\eta^2\theta-\eta^3\in H^6(C^{(6)}).$$

which is equal to
$$\frac{1}{6}(\theta_1+\theta_4+\delta_{14})^3-\frac{1}{2}(\eta_1+\eta_4)(\theta_1+\theta_4+\delta_{14})^2+(\eta_1+\eta_4)^2(\theta_1+\theta_4+\delta_{14})-(\eta_1+\eta_4)^3.$$

The contributing terms in the Kunneth decomposition have type $(2,0,0,4)$:
\begin{eqnarray}&&\frac{1}{2}(\theta_1\theta_4^2+\theta_4\delta_{14}^2)-\frac{1}{2}(\eta_1\theta_4^2+\eta_4\delta_{14}^2)-\theta_1\eta_4\theta_4+2\eta_1\eta_4\theta_4+\theta_1\eta_4^2-3\eta_1\eta_4^2\nonumber\\
&=&\frac{1}{2}(\theta_1-\eta_1)(8\theta_4\eta_4-20\eta_4^2)+\frac{1}{2}(\theta_4-\eta_4)\delta_{14}^2+(2\eta_1-\theta_1)\eta_4\theta_4+(\theta_1-3\eta_1)\eta_4^2\nonumber\\
&=&2(\theta_1-\eta_1)(2\theta_4\eta_4-5\eta_4^2)+(\eta_4\delta_{14}^2+4\theta_1\eta_4^2-\theta_1\theta_4\eta_4)+(2\eta_1-\theta_1)\eta_4\theta_4+(\theta_1-3\eta_1)\eta_4^2\nonumber\\
&=&\eta_4\delta_{14}^2+2(\theta_1-\eta_1)\eta_4\theta_4+(-5\theta_1+7\eta_1)\eta_4^2\nonumber
\end{eqnarray}

 Pushing forward to $C^{(3)}\times C^{(3)}\times C^{(4)}$ :
\begin{eqnarray}\label{class2}&&m_*[\eta_4\delta_{14}^2+2(\theta_1-\eta_1)\eta_4\theta_4+(-5\theta_1+7\eta_1)\eta_4^2]\\
&=&(\eta_3\delta_{13}^2-2\theta_1\eta_3^2)+2(\theta_1-\eta_1)(\eta_3\theta_3+5\eta_3^2)+2(-5\theta_1+7\eta_1)\eta_3^2\nonumber\\
&=&\eta_3\delta_{13}^2+2(\theta_1-\eta_1)\eta_3\theta_3+2(-\theta_1+2\eta_1)\eta_3^2\nonumber
\end{eqnarray}

\item (2,3)+(2,3)

The cycle consists of $(\Gamma_3,\Gamma_3',D_2,D_2')\in C^{(3)}\times C^{(3)}\times C^{(2)}\times C^{(2)}$  given by the conditions
$$\begin{cases}h^0(K_C-\Gamma_3-D_2)>0,\\
h^0(K_C-\Gamma_3'-D_2')>0,\\
r_0\in D_2 .
\end{cases}$$
The map $m$ sends $(\Gamma_3,\Gamma_3',D_2,D_2')$ to $(\Gamma_3,\Gamma_3',D_2+D_2')\in C^{(3)}\times C^{(3)}\times C^{(4)}$. Note that in this bidegree, $m$ is not a lifting of $\lambda_2|_{F_{(2,3)(2,3)}}$. 
As in the previous case, the cycle class is the restriction to $W_1\times C^{(2)}\times C^{(2)}$ of
$$[(\theta_1+\theta_3+\delta_{13})-(\eta_1+\eta_3)]\cdot[(\theta_2+\theta_4+\delta_{24})-(\eta_2+\eta_4)]\cdot\eta_3 .$$
The contributing terms in the Kunneth decomposition are
\begin{enumerate}
\item Type $(2,0,2,2)$
$$[\theta_1\theta_4+\eta_1\eta_4-\eta_1\theta_4-\theta_1\eta_4]\cdot\eta_3 .$$
\item Type $(2,0,4,0)$
$$(\theta_3-\eta_3)(\theta_2-\eta_2)\eta_3=4(\theta_2-\eta_2)\cdot\eta_3^2 .$$
\item Type $(1,1,3,1)$
$$\delta_{13}\delta_{24}\cdot\eta_3 .$$
\end{enumerate}
Pushing these classes forward to $H^4(C^{(3)}\times C^{(3)}\times C^{(4)})$ by $m_*$, we obtain%
\begin{enumerate}
\item
\begin{eqnarray}&&m_*[(\theta_1-\eta_1)\eta_3\theta_4+(\eta_1-\theta_1)\eta_3\eta_4]\nonumber\\
&=&(\theta_1-\eta_1)(\eta_{3}\theta_{3}+5\eta_{3}^2)+(\eta_1-\theta_1)(2\eta_{3}^2)\nonumber\\
&=&(\theta_1-\eta_1)\eta_{3}\theta_{3}+3(\theta_1-\eta_1)\eta_{3}^2\nonumber
\end{eqnarray}

\item\begin{eqnarray}
&&m_*[4(\theta_2-\eta_2)\cdot\eta_3^2]\nonumber\\
&=&4(\theta_2-\eta_2)\cdot\eta_{3}^2\nonumber
\end{eqnarray}
\item
\begin{eqnarray}&&m_*[(\delta_{13}\delta_{24})\eta_3]\nonumber\\
&=&m_*\sum_{i,j=1}^5[-\xi_{1i}\xi_{2j}\xi_{3i}'\xi_{4j}'+\xi_{1i}\xi_{2j}'\xi_{3i}'\xi_{4j}+\xi_{1i}'\xi_{2j}\xi_{3i}\xi_{4j}'-\xi_{1i}'\xi_{2j}'\xi_{3i}\xi_{4j}]\eta_3\nonumber\\
&=&\sum_{i,j=1}^5[-\xi_{1i}\xi_{2j}\xi_{3i}'\xi_{3j}'+\xi_{1i}\xi_{2j}'\xi_{3i}'\xi_{3j}+\xi_{1i}'\xi_{2j}\xi_{3i}\xi_{3j}'-\xi_{1i}'\xi_{2j}'\xi_{3i}\xi_{3j}]\eta_{3}+\eta_{3}^2\sum_{i=1}^5[\xi_{1i}\xi_{2i}'-\xi_{1i}'\xi_{2i}]\nonumber\\
&=&\delta_{13}\delta_{23}\eta_{3}+\delta_{12}\eta_{3}^2\nonumber
\end{eqnarray}
\end{enumerate}

\end{enumerate}
Finally, since $\lambda_2$ sends $(\Gamma_3,\Gamma_3',D_2,D_2')$ to $K_C(-D_2-D_2')$ (instead of $\cO_C(D_2+D_2')$), we apply the involution $p_{2*}p_1^*$ to the sum of the classes in (a), (b), (c) as in Lemma \ref{involution} to obtain the cycle class  

\begin{eqnarray}\label{class3}&&(\theta_1-\eta_1)\theta_{3}(\theta_3-\eta_{3})
+3(\theta_1-\eta_1)(\frac{1}{2}\theta_3^2-\theta_3\eta_3+\eta_3^2)\\
&&+4(\theta_2-\eta_2)(\frac{1}{2}\theta_3^2-\theta_3\eta_3+\eta_3^2)+\delta_{13}\delta_{23}(\theta_3-\eta_{3})+\delta_{12}(\frac{1}{2}\theta_3^2-\theta_3\eta_3+\eta_3^2).\nonumber
\end{eqnarray}

\begin{lemma}\label{involution} The correspondence 
$$M_1=\{(D_4,B_4)\in C^{(4)}\times C^{(4)}|\ D_4+B_4\equiv K_C\}$$ induces an involution $p_{2*}p_1^*: H^4(C^{(4)})\rightarrow H^4(C^{(4)})$ where $p_1$ and $p_2$ are the two birational projections to $C^{(4)}$. Under the decomposition 
$$H^4(C^{(4)})\cong H^4(Pic^4C)\oplus\eta H^2(Pic^4C)\oplus \bC\cdot\eta^2,$$
$p_{2*} p_1^*$ acts as identity on $H^4(Pic^4C)$, sends $\eta\cdot\omega$ to $(\theta-\eta)\cdot\omega$ for any $\omega\in H^2(Pic^4C)$, and $\eta^2$ to $\frac{\theta^2}{2}-\eta\theta+\eta^2$.
\end{lemma}
\begin{proof}First note that the proper transform of the algebraic cycle
$r_0+C^{(3)}$ under the birational map $p_2 p_1^{-1}$ is the cycle 
$$\{B_4\in C^{(4)}|\ h^0(K_C-r_0-B_4)>0)\}$$ 
whose cohomology class is $\theta-\eta$ by the secant plane formula (\ref{secant}). Therefore $p_{2*}p_1^*$ sends $\eta$ to $\theta-\eta$. 
Similarly, the proper transform of $2r_0+C^{(2)}$ is
$$\{B_4\in C^{(4)}|\ h^0(K_C-2r_0-B_4)>0)\}$$
 whose cohomology class is $\frac{\theta^2}{2}-\eta\theta+\eta^2$, i.e.
 $$p_{2*}p_1^*\eta^2=\frac{\theta^2}{2}-\eta\theta+\eta^2.$$

Now let us prove the statement on the summand $\eta H^2(Pic^4C)$. Consider the commutative diagram
$$\xymatrix{&M_1\ar[ld]_-{p_1}\ar[rd]^-{p_2}&\\
C^{(4)}\ar[d]^{\phi}&&C^{(4)}\ar[d]^-{\phi}\\
Pic^4C&&Pic^4C\ar[ll]_-{\tau}}$$ 
where $\tau$  sends $L$ to $K_C-L$. For any $\omega\in H^2(C^{(4)})$,
$$p_1^*(\eta\cdot\phi^*\omega)=p_1^*\eta\cdot\ p_1^*\phi^*\omega\\
=p_1^*\eta\cdot (p_2^*\phi^* \tau^*\omega)$$
By the projection formula, 
$$p_{2*} p_1^*(\eta\cdot\phi^*\omega)=(p_{2*} p_1^*\eta)\cdot(\phi^* \tau^*\omega)=(\theta-\eta)\cdot\phi^*\omega$$
Similarly, the statement about the $H^4(Pic^4C)$ summand is a consequence of the projection formula.
\end{proof}

\subsection{The reducedness of $W^1_5 (C_{ pq })$ and of its
compactification $\oW^1_5 (C_{ pq })$}

\begin{lemma}
\label{lemWflat}

The surface $\oW^1_5 (C_{ pq })$ is reduced and is the flat limit of
the family of $W^1_5 (C_t)$ as $t$ goes to $0$.

\end{lemma}

\begin{proof}

  We will prove that $\oW^1_5 (C_{ pq })$ with its
  reduced scheme structure is the flat limit of the family of $W^1_5
  (C_t)$ as $t$ goes to $0$. By \cite{soucaris94} the family of
  theta divisors specializes to the ample Cartier divisor
\[
\T_{ pq } := \{M\in J^5 C_{ pq } : h^0 (M ) >0\}
\]
on $J^5 C_{ pq }$. We will prove that the Hilbert
polynomial of $\oW^1_5 (C_{ pq })$ with its reduced scheme structure
and with respect to $\T_{ pq }$ is equal to the Hilbert polynomial of
$W^1_5 (C_t)$ with respect to $\T_t$ for $t\neq 0$.

To compute the Hilbert polynomial of $\oW^1_5 (C_{ pq })$, we use the
normalization map (see Lemma \ref{lemW15bar})
\[
\mu = (\nu^*)^{-1} : W_{ pq }\lra\oW^1_5 (C_{ pq }).
\]
From this we obtain an exact sequence
\[
0\lra\cO_{ \oW^1_5 (C_{ pq }) }\lra\mu_*\cO_{ W_{ pq }}\lra\cM\lra 0
\]
where $\cM$ is a sheaf supported on the image of $W^1_4 (C)$ in
$\oW^1_5 (C_{ pq })$. It is immediately seen, by restricting the above
sequence to $W^1_4 (C)$, that
\[
\cM\cong\cO_{ W^1_4 (C) }
\]
so that we have the exact sequence
\begin{equation}\label{eqnoWWpq}
0\lra\cO_{ \oW^1_5 (C_{ pq }) }\lra\mu_*\cO_{ W_{ pq }}\lra\cO_{ W^1_4
(C) }\lra 0.
\end{equation}
To compute $\chi (\cO_{ \oW^1_5 (C_{ pq }) } ( n\T_{ pq }))$, we
therefore compute $\chi (\cO_{ W_{ pq }} ( n\T_{ pq }))$ and $\chi
(\cO_{ W^1_4 (C) }( n\T_{ pq }))$.

By \cite{bartonclemens} page 57 the inverse image of the divisor
$\T_{ pq }$ in $\bP Pic^5 C_{ pq }$ is
numerically equivalent to the sum of reduced divisors
\[
\overline{(\nu^*)^{-1}\T_{C,x}} + Pic^5_0
\]
where we use the notation of \ref{sssectcomp}, $\T_{C,x}$ is the image
of $\T_C\subset Pic^4 C$ in $Pic^5 C$ by the addition of the general
point $x\in C$ and $\overline{(\nu^*)^{ -1}\T_{C,x}}$ is the
closure of $(\nu^*)^{ -1}\T_{C,x}\subset Pic^5 C_{ pq }$ in $\bP Pic^5
C_{ pq }$.

Now, we have
\[
(\nu^* )^{ -1 }\T_{C,x} =\{M\in Pic^5 C_{ pq } : h^0
(\nu^*M(-x) ) > 0\}.
\]

The trace of $\overline{(\nu^*)^{ -1}\T_{C,x}}$ on the image of $W_{
pq }$ in $\bP Pic^5 C_{ pq }$ is reduced for a general choice of $x$
and is equal to
\[
\T_{ C,x } |_{ W_{ pq }} =\{ L\in W_{ pq } : h^0 ( L(-x)) > 0\}.
\]
Furthermore, it is immediate that
\[
Pic_0^5 |_{ W_{ pq }} =X_q.
\]
To compute the degree of $\T_{ pq }$ on $W^1_4 (C)$, we use
the isomorphism $W^1_4 (C)\cong X_p$. In this way we immediately see
that the restriction of $Pic^5_0$ to $W^1_4 (C)$ is zero while $\overline{
(\nu^* )^{ -1 }\T_{ C,x }}$ pulls back to $\T_C |_{ W^1_4 (C) }$ via
the natural embedding $W^1_4 (C)\subset Pic^4 C$.
Therefore, summarizing the above, we have
\[
\chi (n\T_{ pq } |_{ W_{ pq }}) = \chi (n\T_{C,x} |_{ W_{ pq }} +X_q )
\]
and,
\[
\chi (n\T_{ pq } |_{ W^1_4 (C) }) = \chi (n\T_C |_{W^1_4 (C)}).
\]

To compute $\chi (n\T_C |_{ W_{ pq }})$, we use the embedding $q_1$ of $W_{
pq }$ in $C^{ (3 ) }$ given by $g^1_5\mapsto |K-g^1_5|$. Via this
embedding $W_{ pq }$ is identified with the reduced surface in $C^{
(3) }$
\[
\{ \G_3 : h^0 (K_C -p-q -\G_3 ) >0\}
\]
whose cohomology class by the secant plane formula (Section \ref{appsecplane}) is $\theta -\eta$. By
Hirzebruch-Riemann-Roch
\[
\chi (n(\T_C |_{ W_{ pq }} +X_q )) = \frac{1}{2} n(\T_C |_{
W_{ pq }} +X_q ) . (c_1 (T_{ W_{ pq }}) + n(\T_C |_{ W_{ pq
}} +X_q )) +\frac{1}{12} ( c_1 (T_{ W_{ pq }})^2 + c_2 (T_{
W_{ pq }})).
\]
By \cite{macdonald62} (14.5) page 332, the total Chern class of $C^{
(3) }$ is
\[
(1+\eta )^{-6 } \prod_{ i=1 }^5 ( 1+\eta +\sigma_i ) = 1-\eta -\theta - 9\eta^2 + 6\eta\theta - 56\eta^3.
\]
So, using the tangent bundle sequence
\[
0\lra T_{ W_{ pq }}\lra T_{ C^{ (3) }} |_{ W_{ pq }}\lra\cO_{ W_{ pq }
} ( W_{ pq })\lra 0,
\]
we compute
\[
c( T_{ W_{ pq }} ) = \left( 1 - 2\theta - 9\eta^2 +
4\eta\cdot\theta + 2 \theta^2 \right) |_{ W_{ pq }}.
\]
Now, since $X_q$ is the restriction of the zero
section of a $\bP^1$-bundle to $W_{ pq }$, we have
\[
X_q^2 = 0.
\]
Furthermore, the degree of $\T_C$ on $X_q$ is $10$ since
this is a Prym-embedded curve in $Pic^4 C$. By the above,
\[
c_1 ( T_{ W_{ pq }} ) = 2\theta |_{ W_{ pq }},
\]
hence the degree of $c_1 ( T_{ W_{ pq }})$ on $X_q$ is
$20$. Putting all this together with the relations in
\cite[(6.3) page 325]{macdonald62}, we obtain
\[
\chi ( n\T_{ pq }|_{ W_{ pq }} ) = 30 n^2 - 50 n + 22.
\]

To compute $\chi (n\T_C |_{ W^1_4 (C) })$, note that $W^1_4$ has genus $11$ and its
cohomology class in $Pic^4 C$ is twice the minimal
class, i.e.,
\[
[W^1_4 (C) ] = 2\frac{ [\T_C ]^4 }{ 4! }.
\]
Therefore, by Riemann-Roch for curves,
\[
\chi (n\T_C |_{ W^1_4 (C) }) = 1 - 11 + degree(n\T_C |_{ W^1_4 (C) })
= 10 n -10.
\]

Finally, by \eqref{eqnoWWpq},
\[
\chi (\cO_{ \oW^1_5 (C_{ pq }) } ( n\T_{ pq })) = \chi (n\T_{ pq } |_{
W_{ pq }}) -\chi (n\T_{ pq } |_{ W^1_4 (C)}) = 30 n^2 - 60 n + 32.
\]

To compute the Hilbert polynomial of $W^1_5 (C_t )$ for $t\neq 0$, we
first note that the variety $W^1_5 (C_t )$ is always defined as a
determinantal variety as the degeneracy locus of a map of
bundles. Therefore the family of $W^1_5 (X)$ for all smooth $X$ of
genus $6$ is flat where its relative dimension is constant. So, to
compute the Hilbert polynomial, we only need to do so for one smooth
curve of genus $6$. If $X$ is trigonal, $W^1_5 (X)$ is the reduced
union of two copies of $X^{ (2) }$ \cite{teixidor85}:
\[
W^1_5 (X) = X^{ (2) } + g^1_3\cup K_X - (X^{ (2) } + g^1_3).
\]
The intersection of these two components is the reduced curve
\[
X_2 (g^1_4) =\{ D_2 : h^0 (g^1_4 - D_2 ) > 0\}\subset X^{ (2) }
\]
where $g^1_4 = |K_X - 2 g^1_3|$.
As in the previous case, we have the normalization exact sequence
\[
0\lra \cO_{W^1_5 (X) }\lra\mu_*\cO_{ X^{ (2) }\coprod X^{ (2) }
}\lra\cO_{ X_2 ( g^1_4 )}\lra 0.
\]
So
\[
\chi ( n\T_X |_{ W^1_5 (X)} ) = 2\chi ( n\T_X |_{ X^{ (2) }} ) -\chi (
n\T_X |_{ X_2 ( g^1_4 )} ).
\]
This time, using similar methods, we compute
\[
\chi (n\T_X |_{ X^{ (2) }}) = 15 n^2 -24 n + 10,
\]
\[
\chi ( n\T_X |_{ X_2 ( g^1_4 )} ) = 12 n - 12
\]
and
\[
\chi ( n\T_X |_{ W^1_5 (X)} ) = 30 n^2 - 60 n + 32.
\]
\end{proof}


\providecommand{\bysame}{\leavevmode\hbox to3em{\hrulefill}\thinspace}
\providecommand{\MR}{\relax\ifhmode\unskip\space\fi MR }
\providecommand{\MRhref}[2]{%
  \href{http://www.ams.org/mathscinet-getitem?mr=#1}{#2}
}
\providecommand{\href}[2]{#2}

\end{document}